\documentclass[10pt,english,reqno]{amsart} 
\usepackage[mathscr]{euscript}
\newcommand{\fr}{\mathfrak}
\newcommand{\cal}{\mathscr}

\newcommand{\op}{\operatorname}

\newcommand{\dR}{\mathrm{dR}}

\newcommand{\et}{\textnormal{\'et}}

\newcommand{\Sch}{\textbf{Sch}}

\newcommand{\Mod}{\textnormal{-}\textbf{Mod}}

\newtheorem{thm}{Theorem}[section]

\newtheorem{prop}[thm]{Proposition}
\newtheorem{lem}[thm]{Lemma}

\newtheorem{cor}[thm]{Corollary}

\theoremstyle{definition}

\newtheorem{claim}[thm]{Claim}

\newtheorem{eg}[thm]{Example}
\newtheorem{eg-eg}[thm]{Example of Example}

\newtheorem{nota}[thm]{Notation}
\newtheorem{rem}[thm]{Remark}

\usepackage[colorlinks=true]{hyperref}
\usepackage{graphicx}
\usepackage{amssymb}
\usepackage{epstopdf}
\usepackage{enumerate}
\usepackage{tikz}
\usepackage{marginnote}

\usepackage{soul}

\usepackage{mdframed}
\newmdenv[
  topline=false,
  bottomline=false,
  rightline=false,
  skipabove=\topsep,
  skipbelow=\topsep
]{siderules}

\numberwithin{equation}{section}

\usepackage{mathtools}

\usepackage{epigraph}



\usepackage{color}
\usepackage[color,matrix,arrow]{xy}
\xyoption{all}

\title{Quantum parameters of the geometric Langlands theory}

\author{Yifei Zhao}
\date{Fall, 2016}
\email{yifei@math.harvard.edu}

\begin{document}

\begin{abstract}
Fix a smooth, complete algebraic curve $X$ over an algebraically closed field $k$ of characteristic zero. To a reductive group $G$ over $k$, we associate an algebraic stack $\op{Par}_G$ of quantum parameters for the geometric Langlands theory. Then we construct a family of (quasi-)twistings parametrized by $\op{Par}_G$, whose module categories give rise to twisted $\cal D$-modules on $\op{Bun}_G$ as well as quasi-coherent sheaves on the DG stack $\op{LocSys}_G$.
\end{abstract}

\maketitle

\setcounter{tocdepth}{1}
\tableofcontents

\section{Introduction}

\subsection{The geometric Langlands conjecture} The goal of the Langlands program can be broadly described as establishing a bijective correspondence between automorphic forms attached to a reductive group $G$ and Galois representations valued in the Langlands dual group $\check G$.

\subsubsection{}
In the (global, unramified) geometric theory, we fix a smooth, connected, projective curve $X$ over a field $k$. Automorphic functions correspond to certain sheaves on the stack $\op{Bun}_G$ parametrizing $G$-bundles over $X$, and the role of Galois representations is played by $\check G$-local systems on $X$. If we further specialize to the case where $k$ is algebraically closed of characteristic zero, then $\check G$-local systems also form a moduli stack, denoted by $\op{LocSys}_{\check G}$.

\subsubsection{}
Unlike $\op{Bun}_G$, the stack $\op{LocSys}_{\check G}$ is not smooth; furthermore, it is a DG algebraic stack in general and the correct formulation of the geometric Langlands conjecture has to take into account its DG nature.

After Arinkin and Gaitsgory \cite{AG15}, one conjectures an equivalence of DG categories:
\begin{equation}
\label{eq-arinkin-gaitsgory}
\mathbb L_G : \cal D\Mod(\op{Bun}_G) \xrightarrow{\sim} \op{IndCoh}_{\op{Nilp}}(\op{LocSys}_{\check G}).
\end{equation}
Here, the right-hand-side is the DG category of ind-coherent sheaves on $\op{LocSys}_{\check G}$ whose singular support is contained in the global nilpotent cone. This DG category is an enlargement of $\op{QCoh}(\op{LocSys}_{\check G})$, and the appearance of singular support is the geometric incarnation of Arthur parameters.

\subsection{What do we mean by ``quantum''?}
\label{sec-quantum} The quantum geometric Langlands theory seeks to simultaneously deform both sides of \eqref{eq-arinkin-gaitsgory} in a way to make them look more symmetric. The main idea, due to Drinfeld and expounded on by Stoyanovsky \cite{St06} and Gaitsgory \cite{Ga16}, is to consider the DG category of \emph{twisted} $\cal D$-modules on $\op{Bun}_G$.

\subsubsection{}
To explain this approach, let us assume $G$ is simple, and let $\cal L_{G,\det}$ be the determinant line bundle over $\op{Bun}_G$. To every value $c\in k$ one can associate the DG category $\cal D\Mod^c(\op{Bun}_G)$ of $\cal L_{G,\det}^{\frac{c-h^{\vee}}{2h^{\vee}}}$-twisted $\cal D$-modules over $\op{Bun}_G$, where $h^{\vee}$ denotes the dual Coxeter number of $G$.

Let $r=1,2$, or $3$ be the maximal multiplicity of arrows in the Dynkin diagram of $G$. One expects an equivalence of DG categories:
\begin{equation}
\label{eq-quantum-langlands}
\mathbb L_G^{(c)} : \cal D\Mod^c(\op{Bun}_G) \xrightarrow{\sim} \cal D\Mod^{-\frac{1}{rc}}(\op{Bun}_{\check G})
\end{equation}
The equivalence $\mathbb L_G^{(c)}$ should vary continuously in $c$, and degenerate to \eqref{eq-arinkin-gaitsgory} as $c$ tends to zero.\footnote{Indeed, the left-hand-side of \eqref{eq-arinkin-gaitsgory} should more naturally be the DG category of $\cal L_{G,\det}^{-\frac{1}{2}}$-twisted $\cal D$-modules, otherwise known as $\cal D$-modules at the \emph{critical level}. The two DG categories are equivalent by the existence of the Pfaffian.} For a survey on the conjecture \eqref{eq-quantum-langlands}, see \cite{Sc14}.

\subsubsection{}
\label{sec-two-steps}
However, \eqref{eq-quantum-langlands} is conjectured prior to the formulation of \eqref{eq-arinkin-gaitsgory}. For the correct degeneration to $\op{IndCoh}_{\op{Nilp}}(\op{LocSys}_{\check G})$ to take place, one has to renormalize the DG category $\cal D\Mod^c(\op{Bun}_{\check G})$.

The renormalized DG category $\cal D\Mod^c_{\op{ren}}(\op{Bun}_G)$ has apparently different nature depending on the rationality and positivity of $c$, so fitting them in a family is not a trivial matter. Yet, the author expects that they do, and the construction of this family of categories should follow a two-step procedure:
\begin{enumerate}[(a)]
	\item Construct a family of non-commutative algebras $\underline{\cal A}$ over $\op{Bun}_G$, whose generic fiber (at $c<\infty$) is a TDO on $\op{Bun}_G$ and whose special fiber (at $c=\infty$) is $\cal O_{\op{LocSys}_G}$;
	\item Choose certain objects in module category of $\underline{\cal A}$ and let them generate the ``renormalized'' module category that should appear in the quantum geometric Langlands conjecture.
\end{enumerate}
In the present article, we fulfill step (a).

\subsection{What's in this article?} Let us acknowledge right away that for a simple group $G$, the space of quantum parameters is just a copy of $\mathbb P^1$, and when the genus of the curve $X$ is at least 2, the stack $\op{LocSys}_G$ is classical. In this case, the $\mathbb P^1$-family of step (a) has already been constructed by Stoyanovsky \cite{St06}, making use of the line bundle $\cal L_{G,\det}$.

\subsubsection{} Nonetheless, the approach taken in the present article is independent of \cite{St06}. It is motivated by the following considerations for a general reductive group $G$:
\begin{itemize}
	\item In order to treat geometric Eisenstein series and constant term functors, we need to introduce additional quantum parameters to account for \emph{anomalies}; these parameters give rise to TDOs that do not arise from the determinant line bundle.
	
	\item The DG nature of $\op{LocSys}_G$ requires us to consider (generalizations of) TDO's whose underlying $\cal O$-modules are chain complexes; as complexes interact poorly with explicit formulas, we are forced to make a geometric construction, inspired by the theory of \emph{twistings} developed in \cite{GR14}.
\end{itemize}

The result is a construction of $\underline{\cal A}$ that completely dispenses of the line bundle $\cal L_{G,\det}$ and contains more information as soon as the center $Z(G)$ is nontrivial. The key steps in this construction are summarized by the following chart:\footnote{For objects that depend on $\fr g^{\kappa}$ (resp.~$(\fr g^{\kappa}, E)$), we only retain the character $\kappa$ (resp.~$(\kappa,E)$) in the notation.}
\begin{align*}
\left\{
\begin{array}{c}
\text{quantum} \\
\text{parameter $(\fr g^{\kappa},E)$} \end{array}
\right\} \leadsto &
\left\{
\begin{array}{c}
\text{Lie-$*$ algebra} \\
\text{$\widehat{\fr g}_{\cal D}^{(\kappa, E)}$ over $X$} \end{array}
\right\}  \\
& \leadsto
\left\{
\begin{array}{c}
\text{classical quasi-twisting} \\
\text{$\tilde{\cal T}_G^{(\kappa, E)}$ over $\op{Bun}_{G,\infty x}$} \end{array}
\right\} \leadsto
\left\{
\begin{array}{c}
\text{quasi-twisting} \\
\text{$\cal T_G^{(\kappa,E)}$ over $\op{Bun}_G$} \end{array}
\right\}.
\end{align*}
The family of algebras $\underline{\cal A}$ ultimately arises as the universal enveloping algebra of $\cal T_G^{(\kappa, E)}$, when we vary the quantum parameter. From our point-of-view, the family of quasi-twistings $\cal T_G^{(\kappa, E)}$ is a more fundamental object than $\underline{\cal A}$.

\subsection{Organization of this article}

We now give a more detailed summary of the content in each section. In particular, we will explain what quasi-twistings are and how they enter naturally into the picture.

\subsubsection{} We start in \S\ref{sec-quantum-parameter} with the definition of $\op{Par}_G$, the space of quantum parameters. It is a fiber bundle over a compactification of $\op{Sym}^2(\fr g^*)^G$, with fibers being linear stacks describing the ``additional parameters.''

The aforementioned compactification of $\op{Sym}^2(\fr g^*)^G$ is simply the space of $G$-invariant Lagrangian subspaces of $\fr g\oplus\fr g^*$, where a $G$-invariant symmetric bilinear form embeds as its graph. The level ``at $\infty$'' is understood as the Lagrangian subspace $\fr g^{\infty}:=0\oplus\fr g^*$.

\subsubsection{The main idea}
Let us take a $k$-point in $\op{Par}_G$, which is a Lagrangian subspace $\fr g^{\kappa}\subset\fr g\oplus\fr g^*$ together with an additional parameter $E$. Using the theory of Lie-$*$ algebras developed in \cite{BD04}, we construct a central extension
\begin{equation}
\label{eq-qtw-with-full-level-structure}
0\rightarrow\cal O_{\op{Bun}_{G,\infty x}} \rightarrow \widehat{\cal L}^{(\kappa, E)} \rightarrow \cal L^{\kappa} \rightarrow 0
\end{equation}
of Lie algebroids over the scheme $\op{Bun}_{G,\infty x}$ parametrizing $G$-bundles trivialized over the formal neighborhood $D_x$ of a fixed closed point $x\in X$.  We refer to central extensions of Lie algebroids as \emph{classical quasi-twistings}.

For $\fr g^{\kappa}$ arising from a symmetric bilinear form, the reduced universal envelope of \eqref{eq-qtw-with-full-level-structure}:
$$
\op U_{\op{red}}(\widehat{\cal L}^{(\kappa, E)}) := \op U(\widehat{\cal L}^{(\kappa, E)})/(1-\mathbf 1)
$$
defines a TDO over $\op{Bun}_{G,\infty x}$. At $(\fr g^{\kappa},E)=(\infty,0)$, the algebra $\op U_{\op{red}}(\widehat{\cal L}^{(\infty, 0)})$ becomes commutative, and identifies with the ring of functions on the ind-scheme $\op{LocSys}_{G,\infty x}(X-\{x\})$ parametrizing a point $(\cal P_T,\eta)\in\op{Bun}_{G,\infty x}$ together with a connection $\nabla$ over $\cal P_T|_{X-\{x\}}$.

To obtain a central extension of Lie algebroids over $\op{Bun}_G$, we ``descend'' \eqref{eq-qtw-with-full-level-structure} along the torsor $\op{Bun}_{G,\infty x}\rightarrow\op{Bun}_G$, and the algebra $\underline{\cal A}^{(\kappa,E)}$ is set to be its universal envelop. The family of algebras $\underline{\cal A}$ is obtained by letting the point $(\fr g^{\kappa}, E)$ in $\op{Par}_G$ vary.

\subsubsection{The main challenge} There is, however, a caveat in what it means to ``descend'' the classical quasi-twisting \eqref{eq-qtw-with-full-level-structure}. We need a procedure that simultaneously does the following:
\begin{itemize}
	\item For $\fr g^{\kappa}$ arising from a symmetric bilinear form, it performs the strong quotient of a TDO, in the sense of \cite{BB93};
	\item For $\fr g^{\kappa}=\fr g^{\infty}$, it transforms (the ring of functions over) $\op{LocSys}_{G,\infty x}(X-\{x\})$ into the DG stack $\op{LocSys}_G$, a procedure usually understood as symplectic reduction.
\end{itemize}

It turns out that one needs to form what we call the \emph{quotient} of a classical quasi-twisting. In general (and in the way we will apply it), this notion belongs to the DG world, i.e., the quotient of a classical quasi-twisting may cease to be classical.

\subsubsection{}
A (non-classical) \emph{quasi-twisting} over a finite type scheme $Y$ is defined as a $\widehat{\mathbb G}_m$-gerbe in the $\infty$-category of formal moduli problems under $Y$. They make up the geometric theory of central extensions of Lie algebroids over $Y$, and are studied in \S\ref{sec-qtw}. The theory of quasi-twistings is made possible by the machinery of formal groupoids and formal moduli problems, as developed in \cite{GR16}.

The quotient of quasi-twistings fits into the general paradigm of taking the quotient of an inf-scheme by a group inf-scheme. The latter procedure is rather elaborate, as it mixes prestack quotient with formal groupoid quotient. This is the content of \S\ref{sec-quotient}.

\subsubsection{} Finally, we need to deal with the technical annoyance that the theory of \cite{GR16} is built for prestacks locally (almost) of finite type, whereas $\op{Bun}_{G,\infty x}$ is of infinite type. Hence the actual quotient process has to be performed in two steps, one classical and one geometric, along the torsors:
$$
\op{Bun}_{G,\infty x}^{(\le\theta)}\rightarrow \op{Bun}_{G,nx}^{(\le\theta)}\rightarrow\op{Bun}_G^{(\le\theta)},
$$
where $\op{Bun}_G^{(\le\theta)}$ is a Harder-Narasimhan truncation of $\op{Bun}_G$ and $n$ is sufficiently large so that $\op{Bun}_{G,nx}^{(\le\theta)}$ is a scheme (of finite type.) For this reason, we need to prove a number of results communicating between the classical and derived worlds in \S\ref{sec-qtw} and \S\ref{sec-quotient}. It is the author's hope that an extension of \cite{GR16} to $\infty$-dimensional algebraic geometry will render this trick obsolete.

\subsubsection{The main results}
In \S\ref{sec-construction}, we perform the main construction of the quasi-twisting $\cal T_G^{(\kappa,E)}$ over $\op{Bun}_G$ and check that it gives rise to the expected TDOs when $\fr g^{\kappa}$ is the graph of a bilinear form and $E=0$.

In \S\ref{sec-infty}, we show that the DG category of modules over $\cal T_G^{(\infty, 0)}$ recovers $\op{QCoh}(\op{LocSys}_G)$; in doing so, we also obtain a description of the underlying quasi-coherent sheaf of the TDO at an arbitrary level. We end the article with remarks on the ``meaning'' of certain additional parameters at level $\infty$.

\subsection{Quantum vs.~metaplectic parameters}
\label{quantum-v-metaplectic} There is another approach of deforming the DG category $\cal D\Mod(\op{Bun}_G)$\footnote{or in the context of curves over $\mathbb F_q$, the category of $l$-adic sheaves on $\op{Bun}_G$.} that undergoes the name ``metaplectic geometric Langlands program'' (see \cite{GL16}, for example.) We briefly explain the relation between metaplectic and quantum parameters.

For simplicity, let us focus on the points $(\fr g^{\kappa}, E)$ of $\op{Par}_G$ where $\fr g^{\kappa}$ arises from a symmetric bilinear form. Such quantum parameters form an open substack isomorphic to $\op{Sym}^2(\fr g^*)^G\times\mathbf{Ext}^1(\fr z_G\otimes\cal O_X, \omega_X)$, and the quasi-twistings on $\op{Bun}_G$ they produce are in fact twistings.

\subsubsection{} 
Metaplectic parameters give rise to \emph{gerbes}, as opposed to twistings, on $\op{Bun}_G$. In the context of $\cal D$-modules, a \emph{gerbe} on a prestack $\cal Y$ refers to a map from $\cal Y_{\dR}$ to $\op B^2\mathbb G_m$. Note that a gerbe on $\op{Bun}_G$ is sufficient to form the DG category of twisted $\cal D$-modules, but the additional data of a twisting equip this DG category with a forgetful functor to $\op{QCoh}(\op{Bun}_G)$.\footnote{By analogy with the $l$-adic context, gerbes are supposed to be ``topological'' gadgets. However, the existence of the exponential local system on $\mathbb A^1$ shows that the above definition of a gerbe is too na\"ive. In order to retain only topological information, we ought to adjust the definition of a gerbe slightly, as a (2-)torsor over the groupoid of \emph{regular singular} local systems. We will ignore this subtlety for now.}

\subsubsection{}
Let $\op{Gr}_G$ denote the affine Grassmannian associated to $G$, regarded as a factorization prestack over the Ran space of $X$. Conjecturally, the spaces of quantum, respectively metaplectic, parameters have the following intrinsic meanings: they are the spaces of factorization twistings, respectively gerbes, on $\op{Gr}_G$. The corresponding objects on $\op{Bun}_G$ then arise from descent along $\op{Gr}_G\rightarrow\op{Bun}_G$.

Furthermore, there should be a fiber sequence of $k$-linear groupoids:
\begin{equation}
\label{eq-parameter-fiber-sequence}
\mathbf{Pic}^{\op{fact}}(\op{Gr}_G) \rightarrow \mathbf{Tw}^{\op{fact}}(\op{Gr}_G) \rightarrow \mathbf{Ge}^{\op{fact}}(\op{Gr}_G),
\end{equation}
relating three concepts of different origins:
\begin{center}
\renewcommand{\arraystretch}{1.5}
\begin{tabular}{ c | c | c }
  algebro-geometric & differential-geometric & topological \\
   \hline
  $\mathbf{Pic}^{\op{fact}}(\op{Gr}_G)$ & $\mathbf{Tw}^{\op{fact}}(\op{Gr}_G)$ & $\mathbf{Ge}^{\op{fact}}(\op{Gr}_G)$ \\
  Brylinski-Deligne extensions & quantum parameters & metaplectic parameters
\end{tabular}
\end{center}

\subsection{Acknowledgement} The author is deeply indebted to his Ph.D.~advisor Dennis Gaitsgory. Many ideas here arose during conversations with him---in fact, the idea of using quotient by group inf-schemes is essentially his. The author also thanks Justin Campbell for many helpful discussions.

\medskip

\section{The space of quantum parameters}
\label{sec-quantum-parameter}

Throughout the text, we work over an algebraically closed ground field $k$ of characteristic zero. We write $X$ for a smooth, connected, projective curve and $G$ a connected, reductive group.

In this section, we define the smooth algebraic stack $\op{Par}_G$ of quantum parameters for the geometric Langlands theory. We will define a natural isomorphism $\op{Par}_G \xrightarrow{\sim} \op{Par}_{\check G}$, and explain how $\op{Par}_G$ behaves when we change $G$ into the Levi quotient $M$ of a parabolic of $G$.

\subsection{The space $\op{Par}_G$} Let $\fr g$ denote the Lie algebra of $G$.

\subsubsection{}

Consider the symplectic form on $\fr g\oplus\fr g^*$ defined by the pairing:
\begin{equation}
\label{eq-symplectic-form}
\langle\xi\oplus\varphi, \xi'\oplus\varphi'\rangle := \varphi(\xi')-\varphi'(\xi).
\end{equation}
Let $\op{Gr}_{\op{Lag}}^G(\fr g\oplus\fr g^*)$ denote the smooth, projective variety parametrizing $G$-invariant Lagrangian subspaces of $\fr g\oplus\fr g^*$. We will denote an $S$-point of $\op{Gr}_{\op{Lag}}^G(\fr g\oplus\fr g^*)$ by $\fr g^{\kappa}$, regarded as a Lagrangian subbundle of $(\fr g\oplus\fr g^*)\otimes\cal O_S$ stable under the ($\cal O_S$-linear) $G$-action.

\subsubsection{} Clearly, the space $\op{Sym}^2(\fr g^*)^G$ of $G$-invariant symmetric bilinear forms on $\fr g$ embeds into $\op{Gr}_{\op{Lag}}^G(\fr g\oplus\fr g^*)$, where a form $\kappa$, regarded as a linear map $\fr g\rightarrow\fr g^*$, is sent to its graph $\fr g^{\kappa}$. We will use the following notations:
\begin{itemize}
	\item $\fr g^{\infty}$ denotes the $k$-point $\fr g^*$ of $\op{Gr}_{\op{Lag}}^G(\fr g\oplus\fr g^*)$;
	\item $\fr g^{\op{crit}}$ is the graph of the \emph{critical} form $\op{crit}:=-\frac{1}{2}\op{Kil}$, where $\op{Kil}$ is the Killing form of $\fr g$.
	\item for every $S$-point $\fr g^{\kappa}$, the notation $\fr g^{\kappa-\op{crit}}$ denotes the Lagrangian subbundle of $(\fr g\oplus\fr g^*)\otimes\cal O_S$ defined by the property:
	$$
	\xi\oplus\varphi \in \fr g^{\kappa} \iff \xi\oplus(\varphi - \op{crit}(\xi)) \in \fr g^{\kappa-\op{crit}}.
	$$
\end{itemize}

\begin{rem}
Note that if $\kappa\in\op{Sym}^2(\fr g^*)^G$, then $\fr g^{\kappa-\op{crit}}$ is the graph of $\kappa-\op{crit}$, so the above notation is unambiguous; we also have $\fr g^{\infty-\op{crit}} = \fr g^{\infty}$.
\end{rem}

\begin{rem}
More generally, one may replace $\fr g^{\kappa-\op{crit}}$ in the above construction by $\fr g^{\kappa + \kappa_0}$ for any $\kappa_0\in\op{Sym}^2(\fr g^*)^G$. This construction defines an action of $\op{Sym}^2(\fr g^*)^G$ on $\op{Gr}_{\op{Lag}}^G(\fr g\oplus\fr g^*)$ that extends addition on $\op{Sym}^2(\fr g^*)^G$.
\end{rem}

\subsubsection{} We study the $k$-points of $\op{Par}_G$ a bit more closely. Let $\fr g=\fr z\oplus\sum_i\fr g_i$ be the decomposition of $\fr g$ into its center $\fr z$ and simple factors $\fr g_i$.

\begin{lem}
\label{lem-subspace-decomposition}
Any Lagrangian, $G$-invariant subspace $L\hookrightarrow\fr g\oplus\fr g^*$ takes the form $L=L_{\fr z}\oplus\sum_iL_i$ where:
\begin{itemize}
	\item $L_{\fr z}$ is a Lagrangian subspace of $\fr z\oplus\fr z^*$;
	\item each $L_i$ is a Lagrangian, $G$-invariant subspace of $\fr g_i\oplus\fr g_i^*$.
\end{itemize}
\end{lem}
\begin{proof}
The decomposition of $\fr g$ induces a decomposition $\fr g\oplus\fr g^* = (\fr z\oplus\fr z^*)\oplus\sum_i(\fr g_i\oplus\fr g_i^*)$ where the summands are mutually orthogonal with respect to the symplectic form \eqref{eq-symplectic-form}. We may also decompose $L=L_{\fr z}\oplus\sum_j L_j$, where $L_{\fr z}$ is the $G$-fixed subspace and each $L_j$ is irreducible. Obviously, the embedding $L\hookrightarrow\fr g\oplus\fr g^*$ sends $L_{\fr z}$ into $\fr z\oplus\fr z^*$ as an isotropic subspace.

We \emph{claim} that each embedding $L_j\hookrightarrow\fr g\oplus\fr g^*$ factors through $\fr g_i\oplus\fr g_i^*$ for a unique $i$. In other words, the composition $L_j \hookrightarrow \fr g\oplus\fr g^* \twoheadrightarrow \fr g_i\oplus\fr g_i^*$ must vanish for all but one $i$. Suppose, to the contrary, we have $i\neq i'$ such that both
$$
L_j \rightarrow \fr g_i\oplus \fr g_i^*,\quad\text{and}\quad L_j \rightarrow \fr g_{i'}\oplus\fr g_{i'}^*
$$
are nonzero. Without loss of generality, we may assume that the projections onto the first factors $L_j\rightarrow\fr g_i$, $L_j\rightarrow\fr g_{i'}$ are nonzero. Hence we have
\begin{itemize}
	\item $L_j\cong\fr g_i\cong\fr g_{i'}$ as $G$-representations; and
	\item the image of $L_j$ under the projection $\fr g\oplus\fr g^*\twoheadrightarrow \fr g_i\oplus\fr g_{i'}$ is a $G$-invariant subspace with nonzero projection onto both factors.
\end{itemize}
The second statement implies that this image is the entire space $\fr g_i\oplus\fr g_{i'}$, contradicting the equality $\dim(L_j) = \dim(\fr g_i)$ from the first statement. This prove the claim.

Now, suppose $j\neq j'$ and both embeddings $L_j,L_{j'}\hookrightarrow\fr g\oplus\fr g^*$ factor through the same $\fr g_i\oplus\fr g_i^*$. This is obviously impossible since $L_j\oplus L_{j'}\hookrightarrow\fr g\oplus\fr g^*$ would factor through an isomorphism $L_j\oplus L_{j'}\xrightarrow{\sim}\fr g_i\oplus\fr g_i^*$, so it is \emph{not} isotropic. We conclude that there is a bijection between the sets $\{L_j\}$ and $\{\fr g_i\oplus\fr g_i^*\}$ such that each $L_j\hookrightarrow \fr g\oplus\fr g^*$ factors through the corresponding item $\fr g_i\oplus\fr g_i^*$.

Finally, since each $L_j$ is an isotropic subspace of $\fr g_i\oplus\fr g_i^*$, we have:
$$
\dim(\fr g)=\dim(L_{\fr z}) + \sum_j\dim(L_j) \le \dim(\fr z) + \sum_i\dim(\fr g_i) =\dim(\fr g).
$$
Hence the equality is achieved, and each $L_j$ (resp.~$L_{\fr z}$) is a Lagrangian subspace of $\fr g_i\oplus\fr g_i^*$ (resp.~$\fr z\oplus\fr z^*$).
\end{proof}

\begin{cor}
\label{cor-subspace-representation}
Let $L$ be a Lagrangian, $G$-invariant subspace of $\fr g\oplus\fr g^*$. Then there is a (non-canonical) isomorphism $L\xrightarrow{\sim}\fr g$ of $G$-representations. \qed
\end{cor}

Note that we have an obvious morphism:
\begin{equation}
\label{eq-structure-of-par-base}
\op{Gr}_{\op{Lag}}(\fr z\oplus\fr z^*) \times \prod_i\op{Gr}^G_{\op{Lag}}(\fr g_i\oplus\fr g_i^*) \rightarrow \op{Gr}_{\op{Lag}}^G(\fr g\oplus\fr g^*)
\end{equation}
sending a series of vector bundles $\fr z^{\kappa}, \{\fr g_i^{\kappa}\}$ over $S$ to their direct sum $\fr z^{\kappa}\oplus\sum_i\fr g_i^{\kappa}$, which is a subbundle of $(\fr g\oplus\fr g^*)\otimes\cal O_S$.

\begin{cor}
\label{cor-structure-of-par-base}
The morphism \eqref{eq-structure-of-par-base} is an isomorphism.
\end{cor}
\begin{proof}
Indeed, \eqref{eq-structure-of-par-base} is a proper morphism between smooth schemes. Lemma \ref{lem-subspace-decomposition} shows that it is bijective on $k$-points, so in particular quasi-finite, and therefore finite (by properness). A finite morphism of degree $1$ between smooth schemes is an isomorphism.
\end{proof}

Furthermore, any $G$-invariant symmetric bilinear form on $\fr g_i$ fixes an isomorphism $\op{Gr}_{\op{Lag}}^G(\fr g_i\oplus\fr g_i^*)\xrightarrow{\sim}\mathbb P^1$. Therefore $\op{Gr}_{\op{Lag}}^G(\fr g\oplus\fr g^*)$ is \emph{non-canonically} isomorphic to the product of a Lagrangian Grassmannian together with finitely many copies of $\mathbb P^1$.

\subsubsection{}
\label{sec-reduction-to-center}
A particular consequence of Corollary \ref{cor-structure-of-par-base} is that we have a morphism given by projection:
\begin{equation}
\label{eq-reduction-to-center}
\op{Gr}_{\op{Lag}}^G(\fr g\oplus\fr g^*) \rightarrow \op{Gr}_{\op{Lag}}(\fr z\oplus\fr z^*)
\end{equation}
Note that $\fr z$ identifies with the subspace of $G$-invariants of $\fr g$. Although $\fr z^*$ is more naturally the space of $G$-coinvariants of $\fr g^*$, we will identify it with the invariants $(\fr g^*)^G$ via the isomorphism $(\fr g^*)^G\hookrightarrow\fr g^*\twoheadrightarrow\fr z^*$.

More intrinsically, the morphism \eqref{eq-reduction-to-center} is defined on $S$-points by:
$$
\fr g^{\kappa} \leadsto (\fr g^{\kappa})^G:= \fr g^{\kappa} \cap ((\fr z\oplus \fr z^*)\otimes\cal O_S).
$$
where $(\fr z\oplus\fr z^*)\otimes\cal O_S$ is regarded as a submodule of $(\fr g\oplus\fr g^*)\otimes\cal O_S$.

\begin{rem}
We refer to $(\fr g^{\kappa})^G$ as the \emph{$G$-invariants} of $\fr g^{\kappa}$. The same terminology is used in the sequel when we replace $G$ by a different group $H$ and $\fr g^{\kappa}$ by an $H$-invariant subspace of $V\oplus V^*$, where $V$ is any $H$-representation for which the composition $(V^*)^H\hookrightarrow V^*\twoheadrightarrow(V^H)^*$ is an isomorphism.
\end{rem}

\begin{rem}
Note that $(\fr g^{\kappa-\op{crit}})^G\cong(\fr g^{\kappa})^G$, since $\op{crit}$ vanishes on $\fr z$.
\end{rem}

Since the embedding $\fr z\hookrightarrow\fr g$ canonically splits with kernel $\fr g_{\op{s.s.}}:=[\fr g,\fr g]$, there is a surjection $(\fr g\oplus\fr g^*)\otimes\cal O_S \twoheadrightarrow (\fr z\oplus\fr z^*)\otimes\cal O_S$. Note that the image of $\fr g^{\kappa}$ identifies with $(\fr g^{\kappa})^G$, and the composition $(\fr g^{\kappa})^G\hookrightarrow\fr g^{\kappa}\twoheadrightarrow(\fr g^{\kappa})^G$ is the identity. In other words,

\begin{lem}
\label{lem-center-splits}
The morphism $(\fr g^{\kappa})^G \hookrightarrow \fr g^{\kappa}$ canonically splits.\qed
\end{lem}

\noindent
We denote the complement of $(\fr g^{\kappa})^G$ in $\fr g^{\kappa}$ by $\fr g_{\op{s.s.}}^{\kappa}$; it corresponds to the semisimple part of the Lie algebra $\fr g$.

\subsubsection{} We define the stack $\op{Par}_G$ as follows: $\op{Maps}(S,\op{Par}_G)$ is the groupoid of pairs $(\fr g^{\kappa}, E)$, where $\fr g^{\kappa}$ is an $S$-point of $\op{Gr}_{\op{Lag}}^G(\fr g\oplus\fr g^*)$, and $E$ is an extension of $\cal O_{\cal X}$-modules:
\begin{equation}
\label{eq-additional-parameter}
0 \rightarrow \omega_{\cal X/S} \rightarrow E \rightarrow (\fr g^{\kappa})^G\boxtimes\cal O_X \rightarrow 0.
\end{equation}
where $\cal X:=S\times X$, and $\omega_{\cal X/S}\cong\cal O_S\boxtimes\omega_X$ is the relative dualizing sheaf.

In other words, $\op{Par}_G$ is a fiber bundle over $\op{Gr}_{\op{Lag}}^G(\fr g\oplus\fr g^*)$, whose fiber at a $k$-point $\fr g^{\kappa}$ is the linear stack $\mathbf{Ext}((\fr g^{\kappa})^G\boxtimes\cal O_X, \omega_X)$ of extensions over $X$. We think of $\fr g^{\kappa}$ as a generalized symmetric bilinear form on $\fr g$ and $E$ as an \emph{additional} parameter.

\begin{rem}
The substack of $\op{Par}_G$ corresponding to the points $(\fr g^{\kappa}, E)$ where $\fr g^{\kappa}$ arises from a bilinear form conjecturally parametrizes \emph{factorization twistings} on the affine Grassmannian $\op{Gr}_G$, subject to a certain regularity condition (see \S\ref{quantum-v-metaplectic}). Hence, one may view $\op{Par}_G$ as a (partial) compactification of the stack of factorization twistings. We hope to address this conjecture in a forthcoming work.
\end{rem}

\subsection{Langlands duality of $\op{Par}_G$} We now fix a maximal torus $T\hookrightarrow G$. The Langlands dual of $(G,T)$ consists of a reductive group $\check G$ together with a maximal torus $\check T\hookrightarrow\check G$.

\subsubsection{} Let $W:=N_G(T)/T$ denote the Weyl group of $T$. It acts on $\fr t\oplus\fr t^*$ in the standard way. There is a symplectic isomorphism:
\begin{equation}
\label{eq-self-duality-torus}
\fr t\oplus\fr t^* \xrightarrow{\sim} \check{\fr t}\oplus\check{\fr t}^*,\quad \xi\oplus\varphi \leadsto \varphi\oplus(-\xi)
\end{equation}
defined using the canonical identifications $\fr t^*\xrightarrow{\sim}\check{\fr t}$ and $\fr t\xrightarrow{\sim}\check{\fr t}^*$. Furthermore, \eqref{eq-self-duality-torus} intertwines the $W$ and $\check W$ actions (again, under the identification $W\xrightarrow{\sim}\check W$).

\begin{rem}
The sign in \eqref{eq-self-duality-torus} is present not just for matching up the symplectic forms; it is a feature of the Langlands theory.
\end{rem}

Let $\op{Gr}_{\op{Lag}}^W(\fr t\oplus\fr t^*)$ denote the smooth, projective variety parametrizing $W$-invariant, Lagrangian subspaces of $\fr t\oplus\fr t^*$. The isomorphism \eqref{eq-self-duality-torus} induces an isomorphism:
\begin{equation}
\label{eq-self-duality-torus-gr}
\op{Gr}_{\op{Lag}}^W(\fr t\oplus\fr t^*) \xrightarrow{\sim} \op{Gr}_{\op{Lag}}^{\check W}(\check{\fr t}\oplus\check{\fr t}^*).
\end{equation}
We denote the image of $\fr t^{\kappa}$ under \eqref{eq-self-duality-torus-gr} by $\check{\fr t}^{\check{\kappa}}$, and view it as the graph associated to the ``dual'' form.

\subsubsection{} We define a morphism
\begin{equation}
\label{eq-reduction-to-torus}
\op{Gr}_{\op{Lag}}^G(\fr g\oplus\fr g^*) \rightarrow \op{Gr}_{\op{Lag}}^W(\fr t\oplus\fr t^*)
\end{equation}
by sending an $S$-point $\fr g^{\kappa}$ to $(\fr g^{\kappa})^T$, the $T$-invariants of $\fr g^{\kappa}$. An argument similar to the one in \S\ref{sec-reduction-to-center} shows that we have a well-defined map $\op{Gr}_{\op{Lag}}^G(\fr g\oplus\fr g^*) \rightarrow \op{Gr}_{\op{Lag}}(\fr t\oplus\fr t^*)$; it is clear that the image lies in the $W$-fixed locus.

\begin{lem}
\label{lem-reduction-to-torus}
The morphism \eqref{eq-reduction-to-torus} is an isomorphism.
\end{lem}
\begin{proof}
Indeed, a decomposition of $\fr g=\fr z\oplus\sum_i\fr g_i$ into simple factors induces a decomposition $\fr t=\fr z\oplus\sum_i\fr t_i$, where each $\fr t_i$ is the maximal torus of the factor $\fr g_i$. Note that $\fr t_i$ is irreducible as a $W$-representation. An analogue of Corollary \ref{cor-structure-of-par-base} asserts an isomorphism $\op{Gr}_{\op{Lag}}^W(\fr t\oplus\fr t^*)\xrightarrow{\sim} \op{Gr}_{\op{Lag}}(\fr z\oplus\fr z^*)\times\prod_i\op{Gr}^W_{\op{Lag}}(\fr t_i\oplus\fr t_i^*)$, making the following diagram commute:
$$
\xymatrix@C=1.5em@R=1.5em{
	\op{Gr}_{\op{Lag}}^G(\fr g\oplus\fr g^*) \ar[r]^{\eqref{eq-reduction-to-torus}} \ar[d]^{\rotatebox{90}{$\sim$}} & \op{Gr}_{\op{Lag}}^W(\fr t\oplus\fr t^*) \ar[d]^{\rotatebox{90}{$\sim$}} \\
	\op{Gr}_{\op{Lag}}(\fr z\oplus\fr z^*)\times\prod_i\op{Gr}^G_{\op{Lag}}(\fr g_i\oplus\fr g_i^*) \ar[r] & \op{Gr}_{\op{Lag}}(\fr z\oplus\fr z^*)\times\prod_i\op{Gr}^W_{\op{Lag}}(\fr t_i\oplus\fr t_i^*).
}
$$
Note that the bottom arrow is an isomorphism since the choice of a $G$-invariant, symmetric bilinear form on $\fr g_i$ (hence a $W$-invariant form on $\fr t_i$) identifies both $\op{Gr}_{\op{Lag}}^G(\fr g_i\oplus\fr g_i^*)$ and $\op{Gr}_{\op{Lag}}^W(\fr t_i\oplus\fr t_i^*)$ with $\mathbb P^1$.
\end{proof}

\begin{rem}
Using $T$, we may also rewrite \eqref{eq-reduction-to-center} as the two-step procedure of first taking $T$-invariants and then taking $W$-invariants:
$$
(\fr g^{\kappa})^G \xrightarrow{\sim} ((\fr g^{\kappa})^T)^W.
$$
This isomorphism again follows from the description of fibers of $\fr g^{\kappa}$ in Lemma \ref{lem-subspace-decomposition}.
\end{rem}

\subsubsection{} We will consider a slight variant of the isomorphism \eqref{eq-reduction-to-torus} which takes into account the critical shift:
\begin{equation}
\label{eq-reduction-to-torus-true}
\op{Gr}_{\op{Lag}}^G(\fr g\oplus\fr g^*) \xrightarrow{\sim} \op{Gr}_{\op{Lag}}^W(\fr t\oplus\fr t^*),\quad \fr g^{\kappa} \leadsto (\fr g^{\kappa-\op{crit}})^T.
\end{equation}
There is an isomorphism between $\op{Gr}_{\op{Lag}}^G(\fr g\oplus\fr g^*)$ and the corresponding space for $\check G$, making the following diagram commute:
$$
\xymatrix@C=1.5em@R=1.5em{
\op{Gr}_{\op{Lag}}^G(\fr g\oplus\fr g^*) \ar[r]^-{\sim}\ar[d]^{\eqref{eq-reduction-to-torus-true}} & \op{Gr}_{\op{Lag}}^{\check G}(\check{\fr g}\oplus\check{\fr g}^*) \ar[d]^{\eqref{eq-reduction-to-torus-true}\text{ for }\check G} \\
\op{Gr}_{\op{Lag}}^W(\fr t\oplus\fr t^*) \ar[r]^-{\eqref{eq-self-duality-torus-gr}}_-{\sim} & \op{Gr}_{\op{Lag}}^{\check W}(\check{\fr t}\oplus\check{\fr t}^*)
}
$$
We denote the image of $\fr g^{\kappa}$ in $\op{Gr}_{\op{Lag}}^{\check G}(\check{\fr g}\oplus\check{\fr g}^*)$ by $\check{\fr g}^{\check\kappa}$.

Since $(\fr g^{\kappa-\op{crit}})^G \cong (\fr g^{\kappa})^G$, there is an isomorphism
\begin{equation}
\label{eq-duality-for-par}
\op{Par}_G \xrightarrow{\sim} \op{Par}_{\check G},\quad (\fr g^{\kappa}, E) \leadsto (\fr g^{\check{\kappa}}, \check E)
\end{equation}
where $\check E$ is the extension of $(\check{\fr g}^{\check{\kappa}})^{\check G}\boxtimes\cal O_X$ induced from $E$ via the identification of $\cal O_{S\times X}$-modules:
$$
(\fr g^{\kappa})^G \xrightarrow{\sim} (\fr g^{\kappa-\op{crit}})^G \cong (\check{\fr g}^{\check{\kappa}-\op{crit}})^{\check G} \xleftarrow{\sim} (\check{\fr g}^{\check{\kappa}})^{\check G}
$$
where the middle isomorphism comes from the identification of $(\fr g^{\kappa-\op{crit}})^T$ and $(\check{\fr g}^{\check{\kappa}-\op{crit}})^{\check T}$ under \eqref{eq-self-duality-torus-gr}. We refer to \eqref{eq-duality-for-par} as the \emph{Langlands duality} for the parameter space $\op{Par}_G$.

\begin{eg}
\label{eg-dual-coxeter-number}
Suppose $G$ is simple, and we fix a $k$-valued parameter $(\fr g^{\kappa}, 0)$ of $\op{Par}_G$ corresponding to some bilinear form $\kappa$ on $\fr g$. Then $\kappa=\lambda\cdot\op{Kil}_G$ for some $\lambda\in k$. Write $\lambda = (c- h^{\vee})/2h^{\vee}$ for some $c\in k$, where $h^{\vee}$ denotes the dual Coxeter number of $G$.

Assume $c\neq 0$. Then under the isomorphism \eqref{eq-duality-for-par}:
$$
\op{Par}_G\xrightarrow{\sim}\op{Par}_{\check G},\quad (\fr g^{\kappa}, 0)\leadsto (\check{\fr g}^{\check{\kappa}}, 0),
$$
we \emph{claim} that $\check{\fr g}^{\check{\kappa}}$ also arises from a bilinear form $\check{\kappa}$, defined by the formulae:
$$
\check{\kappa} = \check{\lambda}\cdot\op{Kil}_{\check G}, \quad \check{\lambda} = (-\frac{1}{rc} - h)/2h,
$$
where $r=1,2$ or $3$ denotes the maximal multiplicity of arrows in the Dynkin diagram of $G$. Indeed, one sees this from the fact that $(1/2h^{\vee})\cdot\op{Kil}_G$ is the minimal bilinear form and $r$ is the ratio of the square lengths of long and short roots of $G$.
\end{eg}

\subsection{Parabolics and anomaly} We now explain how to incorporate, via an additional parameter, the anomaly term that appears in the study of constant term functors (see \cite[\S 3.3]{Ga15}). This discussion requires further fixing:
\begin{itemize}
	\item a Borel subgroup $B$ containing $T$;
	\item a \emph{theta characteristic} on the curve $X$, i.e., a line bundle $\theta$ together with an isomorphism $\theta^{\otimes 2}\xrightarrow{\sim}\omega_X$.
\end{itemize}
A \emph{standard} parabolic is a parabolic subgroup of $G$ containing $B$.

\subsubsection{} Let $P$ be a standard parabolic with Levi quotient $M$. Then we may regard $T$ as a maximal torus of $M$ via the composition $T\hookrightarrow B\hookrightarrow P \twoheadrightarrow M$. Note that the Weyl group $W_M$ of $T\hookrightarrow M$ is naturally a subgroup of $W$.

The embedding
$$
\fr z\xrightarrow{\sim}\fr t^W\hookrightarrow\fr t^{W_M}\xrightarrow{\sim}\fr z_M
$$
is canonically split; this is because the composition $Z_0(G)\hookrightarrow G\twoheadrightarrow G/[G,G]$ is an isogeny, giving rise to the projection $\fr z_M\rightarrow\fr z$. It follows that we have a canonical map:
\begin{equation}
\label{eq-reduction-on-centers}
\fr z_M\oplus\fr z_M^* \rightarrow \fr z_G\oplus\fr z_G^*
\end{equation}
from the $W_M$-invariants of $\fr t\oplus\fr t^*$ to its $W$-invariants. In particular, given any Lagrangian, $W$-invariant subbundle $\fr t^{\kappa}\hookrightarrow (\fr t\oplus\fr t^*)\otimes\cal O_S$, we have a morphism
\begin{equation}
\label{eq-reduction-subspace-on-centers}
(\fr t^{\kappa})^{W_M}\rightarrow (\fr t^{\kappa})^W
\end{equation}
compatible with \eqref{eq-reduction-on-centers}.

\subsubsection{}
There is a \emph{reduction} morphism
\begin{equation}
\label{eq-reduction-gr}
\op{Gr}_{\op{Lag}}^G(\fr g\oplus\fr g^*) \rightarrow \op{Gr}_{\op{Lag}}^M(\fr m\oplus\fr m^*)
\end{equation}
given by the composition
$$
\op{Gr}_{\op{Lag}}^G(\fr g\oplus\fr g^*)\xrightarrow{\sim} \op{Gr}_{\op{Lag}}^W(\fr t\oplus\fr t^*) \hookrightarrow \op{Gr}_{\op{Lag}}^{W_M}(\fr t\oplus\fr t^*) \xleftarrow{\sim} \op{Gr}_{\op{Lag}}^M(\fr m\oplus\fr m^*)
$$
where the isomorphisms are supplied by \eqref{eq-reduction-to-torus-true} for $G$, respectively $M$. In other words, the image of $\fr g^{\kappa}$ under \eqref{eq-reduction-gr} is an $S$-point $\fr m^{\kappa}$ such that $(\fr m^{\kappa-\op{crit}})^T$ and $(\fr g^{\kappa-\op{crit}})^T$ are canonically isomorphic as subbundles of $(\fr t\oplus\fr t^*)\otimes\cal O_S$.

\subsubsection{} Let $Z_0(M)$ denote the neutral component of the center of $M$. Write $2\check{\rho}_M$ for the character of $Z_0(M)$ determined by the representation $\det(\fr n_P)$, where $\fr n_P$ is the Lie algebra of the unipotent part of $P$.

Let $\check Z_0(M)$ denote the Langlands dual torus of $Z_0(M)$. We use $\omega_X^{\check{\rho}_M}$ to denote the $\check Z_0(M)$-bundle on $X$ induced from $\theta$ under $2\check{\rho}_M$ (regarded as a cocharacter of $\check Z_0(M)$). Then the Atiyah bundle of $\omega_X^{\check{\rho}_M}$ fits into an exact sequence:
$$
0 \rightarrow \fr z_M^*\otimes\cal O_X \rightarrow\op{At}(\omega_X^{\check{\rho}_M}) \rightarrow \cal T_X \rightarrow 0
$$
Its monoidal dual gives rise to an extension of $\cal O_{\cal X}$-modules for every $S$ (recall the notation $\cal X:=S\times X$):
\begin{equation}
\label{eq-anomaly-untwisted}
0 \rightarrow \omega_{\cal X/S} \rightarrow \cal O_S\boxtimes\op{At}(\omega_X^{\check{\rho}_M})^* \rightarrow (\fr z_M\otimes\cal O_S)\boxtimes\cal O_X \rightarrow 0
\end{equation}

For each $S$-point $\fr m^{\kappa}$ of $\op{Gr}_{\op{Lag}}^{M}(\fr m\oplus\fr m^*)$, let $E_{G\rightarrow M}^+$ denote the extension of $(\fr m^{\kappa})^M$ induced from \eqref{eq-anomaly-untwisted} along the canonical map
$$
(\fr m^{\kappa})^M \hookrightarrow (\fr z_M\oplus\fr z_M^*)\otimes\cal O_S\twoheadrightarrow \fr z_M\otimes\cal O_S.
$$
The additional parameter $E_{G\rightarrow M}^+$ is the \emph{anomaly term} at level $\fr m^{\kappa}$.

\subsubsection{} The reduction morphism for quantum parameters is defined by
\begin{equation}
\label{eq-reduction-par}
\op{Par}_G \rightarrow \op{Par}_M,\quad (\fr g^{\kappa}, E) \leadsto (\fr m^{\kappa}, E_{G\rightarrow M})
\end{equation}
where $\fr m^{\kappa}$ is the image of $\fr g^{\kappa}$ under \eqref{eq-reduction-gr}, and $E_{G\rightarrow M}$ is the Baer sum of the following two extensions of $(\fr m^{\kappa})^M$:
\begin{itemize}
	\item an extension induced from $E$ (which is an extension of $(\fr g^{\kappa})^G$) via the map:
	$$
	(\fr m^{\kappa})^M \xrightarrow{\sim} (\fr m^{\kappa-\op{crit}})^M\rightarrow(\fr g^{\kappa-\op{crit}})^G \xrightarrow{\sim}(\fr g^{\kappa})^G,
	$$
	where the map in the middle comes from \eqref{eq-reduction-subspace-on-centers} for $\fr t^{\kappa}:=(\fr m^{\kappa-\op{crit}})^T\cong(\fr g^{\kappa-\op{crit}})^T$;
	\item the anomaly term $E_{G\rightarrow M}^+$ at level $\fr m^{\kappa}$.
\end{itemize}

\begin{rem}
The image of $(\fr g^{\infty}, E)$ under \eqref{eq-reduction-par} is simply the unadjusted one $(\fr m^{\infty}, E)$. In particular, we see that \eqref{eq-reduction-par} is \emph{incompatible} with Langlands duality for quantum parameters, i.e., if we let $\check M$ be the group dual to $M$, the following diagram does \emph{not} commute:
$$
\xymatrix@C=1.5em@R=1.5em{
\op{Par}_G \ar[r]^-{\eqref{eq-duality-for-par}}\ar[d]_{\eqref{eq-reduction-par}} & \op{Par}_{\check G} \ar[d]^{\eqref{eq-reduction-par}} \\
\op{Par}_M \ar[r]^-{\eqref{eq-duality-for-par}} & \op{Par}_{\check M}.
}
$$
\end{rem}

\begin{rem}
For $P=B$ and $M=T$, the character $2\check{\rho}$ is the sum of positive roots, and splittings of \eqref{eq-anomaly-untwisted} form a $\fr t^*\otimes\omega_X$-torsor $\op{Conn}(\omega_X^{\check{\rho}})$ commonly known as the \emph{Miura opers}.
\end{rem}

\subsection{Structures on $\fr g^{\kappa}$} We now note some structures on an $S$-point $\fr g^{\kappa}$ of $\op{Gr}_{\op{Lag}}^G(\fr g\oplus\fr g^*)$ that will be used later. These structures are functorial in $S$.

\subsubsection{} There is an $\cal O_S$-bilinear Lie bracket:
\begin{equation}
\label{eq-structure-lie}
[-,-] : \fr g^{\kappa}\underset{\cal O_S}{\otimes}\fr g^{\kappa} \rightarrow \fr g^{\kappa}
\end{equation}
defined by the formula (on the ambient bundle $(\fr g\oplus\fr g^*)\otimes\cal O_S$):
$$
[(\xi\oplus\varphi)\otimes\mathbf 1, (\xi'\oplus\varphi')\otimes\mathbf 1] := ([\xi, \xi']\oplus\op{Coad}_{\xi}(\varphi'))\otimes\mathbf 1.
$$
One checks immediately that the image lies in $\fr g^{\kappa}$ and the required identities hold. Note that \eqref{eq-structure-lie} factors through the embedding $\fr g_{\op{s.s.}}^{\kappa} \hookrightarrow \fr g^{\kappa}$.

\subsubsection{} There is an $\cal O_S$-bilinear symmetric pairing:
\begin{equation}
\label{eq-structure-form}
(-,-) : \fr g^{\kappa}\underset{\cal O_S}{\otimes}\fr g^{\kappa} \rightarrow \cal O_S
\end{equation}
defined by the formula:
$$
((\xi\oplus\varphi)\otimes\mathbf 1, (\xi'\oplus\varphi')\otimes\mathbf 1) := \varphi'(\xi)\cdot\mathbf 1.
$$

\begin{rem}
The pairing \eqref{eq-structure-form} gives rise to a canonical central extension of the loop algebra $\fr g^{\kappa}(\!(t)\!)$:
$$
0 \rightarrow \cal O_S \rightarrow \widehat{\fr g}^{\kappa} \rightarrow \fr g^{\kappa}(\!(t)\!)\rightarrow 0
$$
whose cocycle is given by the residue pairing $\op{Res}(-,d-)$. This is the prototype of a \emph{generalized} Kac-Moody extension. We will return to it in \S\ref{sec-construction} (although in the $\cal D$-module setting).
\end{rem}

\subsubsection{} Fixing an $S$-point $(\fr g^{\kappa}, E)$ of $\op{Par}_G$, there is an extension of $\cal O_{\cal X}$-modules:
\begin{equation}
\label{eq-canonical-module-ext}
0 \rightarrow \omega_{\cal X/S} \rightarrow \widehat{\fr g}^{(\kappa,E)} \rightarrow\fr g^{\kappa}\boxtimes\cal O_X \rightarrow 0.
\end{equation}
induced from \eqref{eq-additional-parameter} along $\fr g^{\kappa}\otimes\cal O_S\rightarrow (\fr g^{\kappa})^G\otimes\cal O_X$. In other words, $\widehat{\fr g}^{(\kappa,E)}$ is the direct sum of $E$ and $\fr g_{\op{s.s.}}^{\kappa}\boxtimes\cal O_X$, corresponding to the decomposition $\fr g^{\kappa}\xrightarrow{\sim}\widehat{\fr g}^{\kappa}\oplus\fr g_{\op{s.s.}}^{\kappa}$.

\medskip

\part*{Quasi-twistings and their quotients}

\section{Quasi-twistings}
\label{sec-qtw}

In this section, we make sense of a central extension of Lie algebroids in the DG setting; such objects are called \emph{quasi-twistings}. A dynamic theory of Lie algebroids in such generality has been built by Gaitsgory and Rozenblyum \cite{GR16}, and our results in \S\ref{sec-qtw} and \S\ref{sec-quotient} are no more than a modest extension of their theory.

We work over a fixed base scheme $S$ locally of finite type over $k$.

\subsection{The classical notion} Let $Y\in\mathbf{Sch}_{/S}$ be a scheme over $S$.

\subsubsection{}
\label{sec-liealgd}
A \emph{Lie algebroid} over $Y$ (relative to $S$) is an $\cal O_Y$-module $\cal L$ together with an $\cal O_S$-linear Lie bracket $[-,-]$ and an $\cal O_Y$-module map $\sigma:\cal L\rightarrow\cal T_{Y/S}$ such that the following properties are satisfied:
\begin{itemize}
	\item $[l_1, f\cdot l_2] = \sigma(l_1)(f)\cdot l_2 + f[l_1,l_2]$;
	\item $\sigma$ intertwines $[-,-]$ with the canonical Lie bracket on $\cal T_{Y/S}$.
\end{itemize}
The morphism $\sigma$ is called the \emph{anchor map} of $\cal L$. The category of Lie algebroids over $Y$ is denoted by $\mathbf{LieAlgd}_{/S}(Y)$.

A \emph{Picard algebroid} is a central extension of the tangent Lie algebroid $\cal T_{Y/S}$ by $\cal O_Y$; they are equivalent to twisted differential operators (TDO) over $Y$ (see \cite{BB93}).

\subsubsection{} A \emph{classical quasi-twisting} $\cal T^{\op{cl}}$ over $Y$ (relative to $S$) is a central extension
\begin{equation}
\label{eq-classical-qtw}
0\rightarrow\cal O_Y \rightarrow \widehat{\cal L} \rightarrow \cal L \rightarrow 0
\end{equation}
of Lie algebroids. We say that $\cal T^{\op{cl}}$ is \emph{based} at the Lie algebroid $\cal L$. Classical quasi-twistings with a fixed base $\cal L$ form a $k$-linear category, denoted by $\mathbf{QTw}^{\op{cl}}_{/S}(Y/\cal L)$. The following is obvious:

\begin{lem}
A classical quasi-twisting $\cal T^{\op{cl}}$ is a Picard algebroid if and only if the anchor map of $\cal L$ is an isomorphism.\qed
\end{lem}

The \emph{(reduced) universal envelop} of $\cal T^{\op{cl}}$ is the $\cal O_Y$-algebra
$$
\op U(\cal T^{\op{cl}}) := \op U(\widehat{\cal L})/(1-\mathbf 1),
$$
where $\op U(\widehat{\cal L})$ is the universal enveloping algebra of $\widehat{\cal L}$, and $\mathbf 1$ denotes the image of the unit in $\cal O_Y$. A \emph{module} over $\cal T^{\op{cl}}$ is a $\op U(\cal T^{\op{cl}})$-module, or equivalently, a module over the Lie algebroid $\widehat{\cal L}$ on which $\mathbf 1$ acts by the identity.

\subsection{Some $\infty$-dimensional geometry}
\label{sec-elephant}
When $Y$ is not locally of finite type over $S$, the above notion of Lie algebroids is not very amenable to study. We will occasionally encounter some $\infty$-type schemes, for which we need the notion of a Lie algebroid ``on Tate module''.

\subsubsection{} Let $R$ be a (discrete) ring over $k$. The notion of Tate $R$-modules is developed in \cite{Dr06}. We briefly recall the definitions.

An \emph{elementary Tate $R$-module} is a topological $R$-module isomorphic to $P\oplus Q^*$, where $P$ and $Q$ are discrete, projective $R$-modules.\footnote{The topology on $Q^*$ is generated by opens of the form $U^{\perp}$ where $U$ is a finite generated $R$-submodule of $Q$.} A \emph{Tate $R$-module} is topological $R$-module isomorphic to a direct summand of some elementary Tate $R$-module.

There are two important types of submodules of a Tate $R$-module $M$:
\begin{itemize}
	\item a \emph{lattice} is an open submodule $L^+$ with the property that $L^+/U$ is finitely generated for all open submodule $U\hookrightarrow L^+$.
	\item a \emph{co-lattice} is a submodule $L^-$ such that for some lattice $L^+$, both $L^+\cap L^-$ and $M/(L^+ + L^-)$ are finitely generated.
\end{itemize}

\begin{eg}
Clearly, every profinite $R$-module is an elementary Tate $R$-module. The Laurent series ring $R(\!(t)\!)$ is also an elementary Tate module (but \emph{not} profinite).
\end{eg}

\subsubsection{} Given a map of (discrete) rings $R\rightarrow R'$, the \emph{pullback} of a Tate $R$-module $M$ is defined by
$$
M\underset{R}{\widehat{\otimes}} R' := \underset{\longleftarrow}{\op{lim}}\, (M/U)\underset{R}{\otimes} R'
$$
where $U$ ranges over open submodules of $M$.

Tate $R$-modules are local objects for the flat topology (\cite[Theorem 3.3]{Dr06}.) In particular, we may define a Tate $\cal O_Y$-module $\cal F$ over a scheme $Y$ (or more generally, an algebraic stack) as a compatible system of Tate $\cal O_Z$-modules $\cal F\big|_Z$ for every affine scheme $Z$ mapping to $Y$.

\subsubsection{} Let $Y$ be a scheme over $S$. Then $Y$ is \emph{placid} if locally there is a presentation $Y \xrightarrow{\sim} \underset{\longleftarrow}{\op{lim}}\,Y_i$, where each $Y_i$ is a scheme of finite type, and the connecting morphisms $Y_j\rightarrow Y_i$ are smooth surjections. We call a placid scheme $Y$ \emph{pro-smooth}, if we can furthermore choose each $Y_i$ to be smooth.

If $Y$ is a pro-smooth placid scheme, then the tangent sheaf $\cal T_{Y/S}$ is naturally a Tate $\cal O_Y$-module. Indeed, locally on $Y$ there is an isomorphism:
$$
\cal T_{Y/S} \xrightarrow{\sim} \underset{\longleftarrow}{\op{lim}}\,\pi_i^*\cal T_{Y_i/S},
$$
where $\pi_i : Y\rightarrow Y_i$ is the canonical map.

\subsubsection{}
\label{sec-tate-liealgd} Suppose $Y$ is a pro-smooth placid scheme. We define a \emph{Lie algebroid on Tate module} over $Y$ as a Tate $\cal O_Y$-module $\cal L$ together with a \emph{continuous} $\cal O_Y$-linear map $\sigma : \cal L\rightarrow\cal T_{Y/S}$, such that as a \emph{plain} $\cal O_Y$-module, $\cal L$ has the structure of a Lie algebroid with $\sigma$ as its anchor map.

\begin{eg}
The tangent sheaf $\cal T_{Y/S}$ has the structure of a Lie algebroid on Tate module.
\end{eg}

A \emph{classical quasi-twisting on Tate modules} $\cal T^{\op{cl}}$ over $Y$ is a central extension \eqref{eq-classical-qtw} of Lie algebroids on Tate modules where all the morphisms are continuous.

\begin{rem}
The above notion is very na\"ive, as it does not indicate how the Lie bracket interacts with the topology on $\cal L$. However, it suffices for our purpose since in the construction of $\cal T^{(\kappa,E)}_G$ in \S\ref{sec-construction}, the first quotient step will reduce the classical quasi-twisting on Tate modules $\tilde{\cal T}_G^{(\kappa,E)}$ into a discrete, classical quasi-twisting over $\op{Bun}_{G,nx}^{(\le\theta)}$.
\end{rem}

\begin{rem}
We will frequently refer to a classical quasi-twisting on Tate modules simply as a classical quasi-twisting, as the Tate structures should be clear from the context.
\end{rem}

\subsection{Some infinitesimal geometry}
\label{sec-ant}
All materials here are taken from \cite{GR16}. We use the language of $\infty$-categories as developed in \cite{Lu09}, \cite{Lu12}; a \emph{DG category} means a module object over the monoidal $\infty$-category $\mathbf{Vect}$ of complexes of $k$-vector spaces. We use $\mathbf{DGCat}_{\op{cont}}$ to denote the $\infty$-category of all DG categories with continuous functors between them. It admits a symmetric monoidal structure given by the Lurie tensor product.

Suppose $\cal Y\in\mathbf{PStk}_{\op{laft-def}/S}$, where the latter notation means the $\infty$-category of prestacks locally almost of finite type (\emph{laft}) over $S$ which admit deformation theory (see \cite[III.1]{GR16}).

\subsubsection{}
\label{sec-formal-groupoid}
A \emph{groupoid} over $\cal Y$ (relative to $S$) is a simplicial object $\cal R^{\bullet}\in\mathbf{PStk}_{\op{laft-def}/S}$ together with an isomorphism $\cal R^0\xrightarrow{\sim}\cal Y$ such that the following conditions are satisfied:
\begin{itemize}
	\item for every $n\ge 2$, the map
	$$
	\cal R^n \rightarrow \cal R^1\underset{\cal Y}{\times}\cdots\underset{\cal Y}{\times}\cal R^1
	$$
	induced by products of the maps $[1]\rightarrow[n]$ sending $0\leadsto i$, $1\leadsto i+1$, is an isomorphism;
	\item the map $\cal R^2\rightarrow\cal R^1\underset{\cal Y}{\times}\cal R^1$ induced by the product of the maps $[1]\rightarrow[2]$ sending
	$$
	0\leadsto 0,\; 1\leadsto 1\text{ and }0\leadsto 0,\;1\leadsto 2
	$$
	is an isomorphism.
\end{itemize}

A groupoid object $\cal R^{\bullet}$ is a \emph{formal groupoid} if all morphisms in $\cal R^{\bullet}$ are \emph{nil-isomorphisms}, i.e., inducing isomorphism on the reduced prestacks. We denote the $\infty$-category of formal groupoids (relative to $S$) by $\mathbf{FrmGpd}_{/S}$; this is a full subcategory of that of simplicial objects $\cal R^{\bullet}$ in $\mathbf{PStk}_{\op{laft-def}/S}$. We have the functor
\begin{equation}
\label{eq-frmgpd-to-pstk}
\mathbf{FrmGpd}_{/S} \rightarrow \mathbf{PStk}_{\op{laft-def}/S},\quad\cal R^{\bullet}\leadsto\cal R^0,
\end{equation}
whose fiber at $\cal Y\in\mathbf{PStk}_{\op{laft-def}/S}$ is by definition the $\infty$-category of formal groupoids \emph{acting} on $\cal Y$, and is denoted by $\mathbf{FrmGrp}_{/S}(\cal Y)$.

\begin{eg}[Infinitesimal groupoid]
\label{eg-infinitesimal-groupoid}
Completion along the main diagonals $\cal Y\rightarrow\cal Y\underset{S}{\times}\cdots\underset{S}{\times}\cal Y$ organizes into a formal groupoid $\cal R^{\bullet}:=(\cal Y^{\bullet})_{\widehat{\cal Y}}$ acting on $\cal Y$. This is the final object of $\mathbf{FrmGpd}_{/S}(\cal Y)$.
\end{eg}

Furthermore, \eqref{eq-frmgpd-to-pstk} is a Cartesian fibration of $\infty$-categories. The Cartesian arrows in $\mathbf{FrmGrp}_{/S}$ are maps $\cal R^{\bullet} \rightarrow \cal T^{\bullet}$ such that the induced morphism
$$
\cal R^{\bullet} \rightarrow \cal T^{\bullet}\underset{(\cal Z^{\bullet})_{\widehat{\cal Z}}}{\times}(\cal Y^{\bullet})_{\widehat{\cal Y}},\quad\text{where $\cal Y:=\cal R^0$ and $\cal Z:=\cal T^0$}
$$
is an isomorphism.

\subsubsection{}
\label{sec-frmmod} Let $\mathbf{FrmMod}_{/S}$ denote the $\infty$-category of nil-isomorphisms $\cal Y\rightarrow\cal Y^{\flat}$ in $\mathbf{PStk}_{\op{laft-def}/S}$. In particular, $\mathbf{FrmMod}_{/S}$ is a full subcategory of the functor category $\op{Fun}(\Delta^1, \mathbf{PStk}_{\op{laft-def}/S})$. We have a functor
\begin{equation}
\label{eq-frmmod-to-pstk}
\mathbf{FrmMod}_{/S} \rightarrow \mathbf{PStk}_{\op{laft-def}/S},\quad (\cal Y\rightarrow\cal Y^{\flat}) \leadsto \cal Y,
\end{equation}
whose fiber at $\cal Y\in\mathbf{PStk}_{\op{laft-def}/S}$ is by definition the $\infty$-category of formal moduli problems \emph{under} $\cal Y$, and is denoted by $\mathbf{FrmMod}_{/S}(\cal Y)$.

The functor \eqref{eq-frmmod-to-pstk} is also a Cartesian fibration of $\infty$-categories, whose Cartesian arrows are commutative diagrams on the left whose induced square on the right is Cartesian:
$$
\xymatrix@C=1.5em@R=1.5em{
	\cal Y \ar[r]\ar[d] & \cal Z \ar[d] \\
	\cal Y^{\flat} \ar[r] & \cal Z^{\flat}
}\quad
\xymatrix@C=1.5em@R=1.5em{
	\cal Y^{\flat} \ar[r]\ar[d] & \cal Z^{\flat} \ar[d] \\
	\cal Y_{\dR} \ar[r] & \cal Z_{\dR}
}
$$

\begin{rem}
Analogously, one may consider the Cartesian fibration $\mathbf{FrmMod}_{/S}\rightarrow\mathbf{PStk}_{\op{laft-def}/S}$ sending $\cal Y^{\sharp}\rightarrow\cal Y$ to $\cal Y$, whose fiber is called the $\infty$-category of formal moduli problems \emph{over} $\cal Y$. Such objects are also studied in \cite[\S IV.1]{GR16}, but we will not need them for this paper.
\end{rem}

Straightening \eqref{eq-frmmod-to-pstk} gives rise to a \emph{pullback} functor for every morphism $f : \cal Y\rightarrow\cal Z$ in $\mathbf{PStk}_{\op{laft-def}/S}$:
$$
f^!_{\mathbf{FrmMod}} : \mathbf{FrmMod}_{/S}(\cal Z) \rightarrow \mathbf{FrmMod}_{/S}(\cal Y),\quad f^!_{\mathbf{FrmMod}}\cal Z^{\flat}\xrightarrow{\sim} \cal Z^{\flat}\underset{\cal Z_{\dR}}{\times} \cal Y_{\dR}.
$$

\subsubsection{}
\label{sec-delooping-functor} The \v Cech nerve construction defines a functor $\Omega : \mathbf{FrmMod}_{/S} \rightarrow \mathbf{FrmGpd}_{/S}$ of $\infty$-categories over $\mathbf{PStk}_{\op{laft-def}/S}$. The main result in \cite[\S IV.1]{GR16} (which has its origin in Lurie's theory of formal moduli problems) can be summarized as follows:

\begin{thm}[Lurie-Gaitsgory-Rozenblyum]
\label{thm-frmmod-to-frmgpd}
The functor $\Omega$ is an equivalence. \qed
\end{thm}

\noindent
Indeed, \cite[\S IV.1, Theorem 2.3.2]{GR16} shows that $\Omega$ is an equivalence when restricted to the fiber at each $\cal Y\in\mathbf{PStk}_{\op{laft-def}/S}$. The above formulation follows because $\Omega$ also preserves Cartesian arrows (and we appeal to \cite[Corollary 2.4.4.4]{Lu09}).

We denote the functor inverse to $\Omega$ by $\op B : \mathbf{FrmGpd}_{/S} \rightarrow \mathbf{FrmMod}_{/S}$. Their restrictions to the fiber at $\cal Y\in\mathbf{PStk}_{\op{laft-def}/S}$ are denoted by $\Omega_{\cal Y}$ and $\op B_{\cal Y}$.

\begin{eg}[de Rham prestack]
\label{eg-de-rham-prestack}
Let $\cal Y_{\dR/S}$ denote the fiber product $\cal Y_{\dR}\underset{S_{\dR}}{\times}S$ which is the terminal object of $\mathbf{FrmMod}_{/S}(\cal Y)$. Then $\cal Y_{\dR/S}$ corresponds to the infinitesimal groupoid $(\cal Y)^{\bullet}_{\widehat{\cal Y}}$ (Example \ref{eg-infinitesimal-groupoid}) under the equivalence $\mathbf{FrmGpd}_{/S}(\cal Y) \xrightarrow{\sim} \mathbf{FrmMod}_{/S}(\cal Y)$.

In particular, given any group object $H\in\mathbf{PStk}_{\op{laft-def}/S}$, there is a canonical short exact sequence of group prestacks:
\begin{equation}
\label{eq-group-de-rham}
1 \rightarrow H_{\{\widehat{1}\}} \rightarrow H \rightarrow H_{\dR/S} \rightarrow 1
\end{equation}
\end{eg}

Here is a simple Corollary of Theorem \ref{thm-frmmod-to-frmgpd}:

\begin{cor}
The prestack $\op B_{\cal Y}(\cal R^{\bullet})$ identifies with the quotient of $\cal R^{\bullet}$ in $\mathbf{PStk}_{\op{laft-def}/S}$.
\end{cor}
\begin{proof}
We need to show that $\op B_{\cal Y}(\cal R^{\bullet})$ identifies with $\underset{\Delta^{\op{op}}}{\op{colim}}\,\cal R^{\bullet}$, where the colimit is taken in $\mathbf{PStk}_{\op{laft-def}/S}$. This follows from the fact that $\op{Maps}(\op B_{\cal Y}(\cal R^{\bullet}), \cal Z)$ identifies with the mapping space from $\cal Y\rightarrow\op B_{\cal Y}(\cal R^{\bullet})$ to $\cal Z\rightarrow\cal Z$ in $\mathbf{FrmMod}_{/S}$, which by Theorem \ref{thm-frmmod-to-frmgpd} identifies with $\op{Maps}(\cal R^{\bullet}, \cal Z)$.
\end{proof}

\subsubsection{} However, we point out that the quotient of $\cal R^{\bullet}$ in $\mathbf{PStk}_{\op{laft-def}/S}$ may not agree with that in $\mathbf{PStk}_{/S}$, which is one of the main technical complications for us.

\begin{eg}
Let $S=\op{pt}$ for simplicity. The \v Cech nerve of the object $\op{pt} \rightarrow \mathbb A^1_{\{\widehat{0}\}}$ in $\mathbf{FrmMod}$ is the formal groupoid $\cal R^{\bullet}:=\op{pt}\underset{\mathbb A^1}{\times}\cdots\underset{\mathbb A^1}{\times}\op{pt}$. The quotient $\underset{\Delta^{\op{op}}}{\op{colim}}\,\cal R^{\bullet}$ taken in $\mathbf{PStk}_{/S}$ does \emph{not} agree with $\mathbb A^1_{\{\widehat{0}\}}$, since there are equivalences:
$$
\op{Maps}(\op{Spec}(k[\varepsilon]/(\varepsilon^2)), \underset{\Delta^{\op{op}}}{\op{colim}}\,\cal R^{\bullet}) \xleftarrow{\sim}\underset{\Delta^{\op{op}}}{\op{colim}}\,\op{Maps}(\op{Spec}(k[\varepsilon]/(\varepsilon^2)), \cal R^{\bullet}) \xrightarrow{\sim}\op{pt},
$$
but $\mathbb A^1_{\{\widehat{0}\}}$ receives nontrivial maps from $\op{Spec}(k[\varepsilon]/(\varepsilon^2))$.
\end{eg}

\subsubsection{}
\label{sec-smooth-comparison-of-quotients} We comment on one case where $\op B_{\cal Y}(\cal R^{\bullet})$ agrees with the quotient in $\mathbf{PStk}_{/S}$.

\begin{lem}
\label{lem-smooth-comparison-of-quotients}
Suppose the morphisms $\xymatrix@C=1.5em{\cal R^1 \ar@<0.5ex>[r]\ar@<-0.5ex>[r] & \cal Y}$ are formally smooth. Then the canonical map $\underset{\Delta^{\op{op}}}{\op{colim}}\,\cal R^{\bullet} \rightarrow \op B_{\cal Y}(\cal R^{\bullet})$, where the colimit is taken in $\mathbf{PStk}_{/S}$, is an isomorphism.
\end{lem}

Recall that a morphism $\cal X\rightarrow\cal Y$ of prestacks is called \emph{formally smooth} if for every affine DG scheme $T$ over $\cal Y$, and a nilpotent embedding $T\hookrightarrow T'$, the map
$$
\op{Maps}(T',\cal Y) \rightarrow \op{Maps}(T, \cal Y)
$$
is surjective on $\pi_0$ (see \cite[III.1, \S7.3]{GR16}.)

Let $\cal T_{\cal X/\cal Y}^*\big|_x$ denote the cotangent complex at a $T$-point $x: T\rightarrow\cal X$. It is proved in \emph{loc.cit.~}that if $\cal X\rightarrow\cal Y$ admits (relative) deformation theory, then formal smoothness is equivalent to
\begin{equation}
\label{eq-formal-smoothness-criterion}
\op{Maps}(\cal T_{\cal X/\cal Y}^*\big|_x, \cal F)\in\mathbf{Vect}^{\le 0},
\end{equation}
where $\cal F\in\op{QCoh}(T)^{\heartsuit}$ and $T$ is any \emph{affine} DG scheme with a morphism $x : T\rightarrow\cal X$.

\begin{proof}[Proof of Lemma \ref{lem-smooth-comparison-of-quotients}]
The authors of \cite{GR16} give the following explicit description of $\op B_{\cal Y}(\cal R^{\bullet})$: let $U$ be an affine DG scheme, then $\op{Maps}(U,\op B_{\cal Y}(\cal R^{\bullet}))$ identifies with the space of:
\begin{itemize}
	\item a formal moduli problem $\widetilde U$ \emph{over} $U$;
	\item a morphism from the \v Cech nerve of $\widetilde U\rightarrow U$ to $\cal R^{\bullet}$, such that the following diagram is Cartesian for each of the vertical arrows:
	$$
	\xymatrix@C=1.5em@R=1.5em{
		\widetilde{U} \underset{U}{\times} \widetilde{U} \ar[r]\ar@<0.5ex>[d]\ar@<-0.5ex>[d] & \cal R^1 \ar@<0.5ex>[d]\ar@<-0.5ex>[d] \\
		\widetilde U \ar[r] & \cal R^0
	}
	$$
\end{itemize}
On the other hand, $\op{Maps}(U, \underset{\Delta^{\op{op}}}{\op{colim}}\,\cal R^{\bullet})$ classifies the above data satisfying the \emph{condition} that $\widetilde{U}\rightarrow U$ admits a section. Now, since $\widetilde{U}\rightarrow U$ is a nil-isomorphism, we obtain a section over $U^{\op{red}}$. A lift of this section to $U$ exists if the morphism $\widetilde U\rightarrow U$ is formally smooth.

Now, let $T$ be affine DG scheme equipped with a map $\tilde u : T\rightarrow\widetilde U$. The Cartesian diagrams:
$$
\xymatrix@C=1.5em@R=1.5em{
	\widetilde{U} \underset{U}{\times} \widetilde{U} \ar[r]\ar[d] & \widetilde U \ar[d] \\
	\widetilde U \ar[r] & U
}\quad
\xymatrix@C=1.5em@R=1.5em{
	\widetilde{U} \underset{U}{\times} \widetilde{U} \ar[r]\ar[d] & \cal R^1 \ar[d] \\
	\widetilde U \ar[r] & \cal Y
}
$$
show that $\cal T_{\widetilde U/U}^*\big|_{\tilde u}$ is isomorphic to $\cal T^*_{\widetilde{U} \underset{U}{\times} \widetilde{U}/\widetilde U}\big|_{(\tilde u,\tilde u)}$, which is in turn isomorphic to $\cal T^*_{\cal R^1/\cal Y}\big|_{r^1}$ where $r^1$ is the composition $T\xrightarrow{(\tilde u,\tilde u)}\widetilde{U} \underset{U}{\times} \widetilde{U} \rightarrow\cal R^1$. Hence the formal smoothness of $\cal R^1$ over $\cal Y$ implies that of $\widetilde U$ over $U$.
\end{proof}

In particular, let $\fr h$ be a (classical) Lie algebra over $\cal O_S$, such that $\exp(\fr h)$ acts on some $\cal Y\in\mathbf{PStk}_{\op{laft-def}/S}$. Then the groupoid $\xymatrix@C=1.5em{\cal Y\underset{S}{\times}\exp(\fr h) \ar@<0.5ex>[r]\ar@<-0.5ex>[r] & \cal Y}$ is formally smooth, so its quotient may be formed in $\mathbf{PStk}_{/S}$. We have two particular instances of this example:
\begin{itemize}
	\item Taking $\cal Y=\op{pt}$, we see that $\op B\exp(\fr h)$ is the prestack quotient $\op{pt}/\exp(\fr h)$;
	\item Let $H$ be a group scheme. Then the prestack quotient $H/\exp(\fr h)$ identifies with $H_{\dR/S}$.
\end{itemize}

\subsection{More on $\mathbf{FrmMod}_{/S}$} We now recall the notion of modules over a formal moduli problem, and the comparison between formal moduli problems and Lie algebroids. These materials are taken from \cite[\S IV.4]{GR16}.

\subsubsection{Ind-coherent sheaves}
Recall that for an affine DG scheme $Y$ almost of finite type over $S$, the DG category $\op{IndCoh}(Y)$ is the ind-completion of the full subcategory $\op{Coh}(Y)\hookrightarrow\op{QCoh}(Y)$.\footnote{We use the notation $\op{QCoh}(Y)$ to denote the DG category of complexes of $\cal O_Y$-modules. In contrast, the abelian category of $\cal O_Y$-modules is denoted by $\op{QCoh}(Y)^{\heartsuit}$, understood as the heart of a natural $t$-structure on $\op{QCoh}(Y)$. The same principle applies to variants of this notation.} There is a symmetric monoidal functor:
$$
\Upsilon_Y : \op{QCoh}(Y) \rightarrow \op{IndCoh}(Y),\quad \cal F\leadsto \cal F\otimes\omega_Y,
$$
which is an equivalence of DG categories if $Y$ is smooth over $k$ (\cite[II.3]{GR16}). More generally, we may regard $\op{IndCoh}(-)$ as a functor
$$
\op{IndCoh} : \mathbf{PStk}_{\op{laft}/S}\rightarrow\mathbf{DGCat}_{\op{cont}},
$$
where a morphism $f : \cal X\rightarrow\cal Y$ of laft prestacks gives rise to the functor of $!$-pullback: $f^! : \op{IndCoh}(\cal Y)\rightarrow\op{IndCoh}(\cal X)$. 

Note that if $f : \cal X\rightarrow\cal Y$ is an inf-schematic nil-isomorphism, then the functor $f^!$ is conservative (\cite[III.3, Proposition 3.1.2]{GR16}). It furthermore has a left adjoint $f_*^{\op{IndCoh}}$ and the pair $(f_*^{\op{IndCoh}}, f^!)$ is monadic. One deduces from this a descent property:

\begin{prop}
\label{prop-indcoh-descent}
Let $\cal X^{\bullet}_{\cal Y}$ be the \v{C}ech nerve of an inf-schematic nil-isomorphism $f : \cal X\rightarrow\cal Y$. Then the canonical functor:
\begin{equation}
\label{eq-indcoh-descent}
\op{IndCoh}(\cal Y) \rightarrow\op{Tot}(\op{IndCoh}(\cal X_{\cal Y}^{\bullet}))
\end{equation}
is an equivalence.
\end{prop}
\begin{proof}
This is Proposition 3.3.3 in \emph{loc.cit.}.
\end{proof}

\subsubsection{}
The DG category of \emph{modules} over an object $\cal Y^{\flat}\in\mathbf{FrmMod}_{/S}(\cal Y)$ is defined as $\op{IndCoh}(\cal Y^{\flat})$. Note that $\op{IndCoh}(\cal Y^{\flat})$ is tensored over $\op{QCoh}(S)$. By the above discussion, there is a conservative functor $\mathbf{oblv} : \op{IndCoh}(\cal Y^{\flat})\rightarrow\op{IndCoh}(\cal Y)$ given by $!$-pullback along $\cal Y\rightarrow\cal Y^{\flat}$. Furthermore, Proposition \ref{prop-indcoh-descent} provides an equivalence of categories:
\begin{equation}
\label{eq-module-cat-descent}
\op{IndCoh}(\cal Y^{\flat}) \xrightarrow{\sim} \op{Tot}(\op{IndCoh}(\cal R^{\bullet})).
\end{equation}
whenever $\cal Y^{\flat}=\op B_{\cal Y}(\cal R^{\bullet})$.

\subsubsection{}
\label{sec-liealgd-to-frmmod}
Given $\cal Y^{\flat} \in (\mathbf{PStk}_{\op{laft-def}})_{\cal Y//S}$, we can associate the relative tangent complex $\mathbb T_{\cal Y/\cal Y^{\flat}}$ which is in general an object of $\op{IndCoh}(\cal Y)$. The following result is \cite[IV.4, Theorem 9.1.5]{GR16}:

\begin{thm}
\label{thm-liealgd-derived}
If $Y\in\mathbf{Sch}^{\op{ft}}_{/S}$,\footnote{The notation $\Sch^{\op{ft}}_{/S}$ (resp.~$\Sch^{\op{lft}}_{/S}$) means classical scheme (locally) of finite type over $S$.} then we have a fully faithful functor:
\begin{equation}
\label{eq-liealgd-derived}
\mathbf{LieAlgd}_{/S}(Y) \hookrightarrow \mathbf{FrmGrp}_{/S}(Y)
\end{equation}
whose essential image consists of those formal groupoids $\cal R^{\bullet}$ such that $\mathbb T_{Y/\op B_{Y}(\cal R^{\bullet})}$ lies in the essential image of $\op{QCoh}(Y)^{\heartsuit}$ under $\Upsilon_Y$.\qed
\end{thm}

Composing \eqref{eq-liealgd-derived} with $\op B_Y$, we obtain a fully faithful functor
\begin{equation}
\label{eq-liealgd-to-frmmod}
\mathbf{LieAlgd}_{/S}(Y)\hookrightarrow\mathbf{FrmMod}_{/S}(Y).
\end{equation}
whose essential image consists of those formal moduli problems $\cal Y^{\flat}\in\mathbf{FrmMod}_{/S}(Y)$ such that $\mathbb T_{Y/\cal Y^{\flat}}$ lies in $\Upsilon_Y(\op{QCoh}(Y)^{\heartsuit})$. Furthermore, given a smooth morphism $\pi : Y'\rightarrow Y$ in $\mathbf{Sch}^{\op{ft}}_{/S}$, the following diagram commutes:
\begin{equation}
\label{eq-liealgd-to-frmmod-diag}
\xymatrix@C=3em{
	\mathbf{LieAlgd}_{/S}(Y) \ar[r]^-{\pi^!_{\mathbf{LieAlgd}}} \ar[d]^{\eqref{eq-liealgd-to-frmmod}} & \mathbf{LieAlgd}_{/S}(Y') \ar[d]^{\eqref{eq-liealgd-to-frmmod}} \\
	\mathbf{FrmMod}_{/S}(Y) \ar[r]^-{\pi^!_{\mathbf{FrmMod}}} & \mathbf{FrmMod}_{/S}(Y')
}
\end{equation}
where $\pi^!_{\mathbf{LieAlgd}}$ is the pullback of Lie algebroids (as described in \cite{BB93}), and $\pi^!_{\mathbf{FrmMod}}$ is the functor described in \S\ref{sec-frmmod}.

\begin{rem}
In what follows, we will frequently use the fact that $\pi^!_{\mathbf{LieAlgd}}(\cal L)$ has underlying $\cal O_{Y'}$-module given by $\pi^*\cal L\underset{\pi^*\cal T_{Y/S}}{\times}\cal T_{Y'/S}$.
\end{rem}

\begin{nota}
We shall refer to the image $\cal Y^{\flat}$ of a Lie algebroid $\cal L$ under \eqref{eq-liealgd-to-frmmod} as the formal moduli problem \emph{associated} to $\cal L$, and denote it by $\cal Y^{\flat}:=\cal L_{\mathbf F}$.
\end{nota}

Note that when $Y$ is smooth, $\op{IndCoh}(\cal Y^{\flat})$ identifies with the DG category of complexes of (quasi-coherent) $\cal L$-modules.

\subsection{Quasi-twistings} Let $\cal Y\in\mathbf{PStk}_{\op{laft-def}/S}$. We use $\widehat{\mathbb G}_m$ to denote the formal completion of $\mathbb G_m$ at identity. It is a group formal scheme.

\subsubsection{} A \emph{quasi-twisting} $\cal T$ over $\cal Y$ consists of the following data:
\begin{itemize}
	\item an object $\cal Y^{\flat}\in\mathbf{FrmMod}_{/S}(\cal Y)$;
	\item a $\widehat{\mathbb G}_m$-gerbe $\widehat{\cal Y}^{\flat}$ over $\cal Y^{\flat}$;\footnote{Abuse of notation: we should really be thinking about $S\times\widehat{\mathbb G}_m$ as a group formal scheme over $S$.}
	\item a trivialization of the pullback of $\widehat{\cal Y}^{\flat}$ along $\cal Y\rightarrow\cal Y^{\flat}$.
\end{itemize}
We say that $\cal T$ is \emph{based} at the formal moduli problem $\cal Y^{\flat}$.

\begin{rem}
For an abelian group prestack $A$ over $S$, the notion of an $A$-gerbe here is taken in the na\"ive sense: the prestack $\op B^2A$ classifies $A$-gerbes (on an affine $S$-scheme) that are globally nonempty, and an $A$-gerbe on a prestack $\cal Y$ is an object of
$$
\op{Ge}_{A}(\cal Y):=\underset{T\rightarrow\cal Y}{\op{lim}}\,\op{Maps}(T, \op B^2A),
$$
where $T$ ranges through affine $S$-schemes mapping to $\cal Y$. We will later show that using \'etale locally trivial $\widehat{\mathbb G}_m$-gerbes in the definition of a quasi-twisting produces the same class of objects.
\end{rem}

\begin{rem}
Alternatively, one can think of a quasi-twisting $\cal T$ as consisting of two formal moduli problems $\widehat{\cal Y}^{\flat}\rightarrow\cal Y^{\flat}$ under $\cal Y$, equipped with the structure of a $\widehat{\mathbb G}_m$-gerbe.
\end{rem}

The $\infty$-groupoid of quasi-twistings $\cal T$ based at $\cal Y^{\flat}$ can be defined as a fiber of $\infty$-groupoids:
$$
\mathbf{QTw}_{/S}(\cal Y/\cal Y^{\flat}) := \op{Fib}(\op{Ge}_{\widehat{\mathbb G}_m}(\cal Y^{\flat}) \rightarrow \op{Ge}_{\widehat{\mathbb G}_m}(\cal Y)).
$$
More generally, we use $\mathbf{QTw}_{/S}^A(\cal Y/\cal Y^{\flat})$ to denote an analogously defined category, with the abelian group prestack $A$ acting as the structure group instead of $\widehat{\mathbb G}_m$.

\subsubsection{} We now show that quasi-twistings can be defined using different structure groups. The same results about twistings are obtained in \cite{GR14}.

\begin{lem}
\label{lem-qtw-group-change}
The functor of inducing an $A$-gerbe from an $A_{\{\widehat 1\}}$-gerbe gives rise to an equivalence of categories $\mathbf{QTw}_{/S}^{A_{\{\widehat 1\}}}(\cal Y/\cal Y^{\flat})\xrightarrow{\sim} \mathbf{QTw}_{/S}^A(\cal Y/\cal Y^{\flat})$. 
\end{lem}
\begin{proof}
In light of the exact sequence \eqref{eq-group-de-rham}, an inverse functor exists if the induced $A_{\dR/S}$-gerbe of any object in $\mathbf{QTw}_{/S}^A(\cal Y/\cal Y^{\flat})$ is canonically trivialized. Indeed, let $\widehat{\cal Y}^{\flat}_{A_{\dR/S}}$ be the $A_{\dR/S}$-gerbe over $\cal Y^{\flat}$ induced from some $A$-gerbe $\widehat{\cal Y}^{\flat}_A$. Clearly, there is an identification between $\widehat{\cal Y}^{\flat}_{A_{\dR/S}}$ and the formal completion of $\widehat{\cal Y}^{\flat}_A$ inside $\cal Y^{\flat}$, i.e., $\widehat{\cal Y}^{\flat}_{A_{\dR}}\xrightarrow{\sim} (\widehat{\cal Y}^{\flat}_A)_{\dR/S}\underset{\cal Y_{\dR/S}}{\times}\cal Y^{\flat}$ (c.f.~Example \ref{eg-de-rham-prestack}).

Therefore, a section of the $A_{\dR/S}$-gerbe $\widehat{\cal Y}^{\flat}_{A_{\dR}/S}$ amounts to filling in the dotted arrow
$$
\xymatrix@R=1.5em@C=1.5em{
 \widehat{\cal Y}^{\flat}_{A_{\dR}/S} \ar[r]\ar[d] & (\widehat{\cal Y}^{\flat}_A)_{\dR/S} \ar[d] \\
 \cal Y^{\flat} \ar[r]\ar@{.>}[ur] & \cal Y_{\dR/S}
}
$$
making the lower-right triangle commute. However, the structure of a quasi-twisting on $\widehat{\cal Y}^{\flat}_A$ supplies a section $\cal Y\rightarrow\widehat{\cal Y}^{\flat}_A$ over $\cal Y^{\flat}$. Hence we obtain a map $\cal Y^{\flat}\rightarrow\cal Y_{\dR/S}\rightarrow (\widehat{\cal Y}^{\flat}_A)_{\dR/S}$ over $\cal Y_{\dR/S}$.
\end{proof}

It follows from Lemma \ref{lem-qtw-group-change} that the following functors are equivalences:
\begin{equation}
\label{eq-qtw-group-change}
\mathbf{QTw}^{\mathbb G_m}_{/S}(\cal Y/\cal Y^{\flat}) \xleftarrow{\sim} \mathbf{QTw}_{/S}(\cal Y/\cal Y^{\flat}) \xrightarrow{\sim} \mathbf{QTw}^{\widehat{\mathbb G}_a}_{/S}(\cal Y/\cal Y^{\flat}) \xrightarrow{\sim} \mathbf{QTw}^{\mathbb G_a}_{/S}(\cal Y/\cal Y^{\flat}).
\end{equation}

\noindent
Suppose we let $\mathbf{QTw}_{/S}^{\et}(\cal Y/\cal Y^{\flat})$ denote the $\infty$-groupoid of \'etale locally trivial $\widehat{\mathbb G}_m$-gerbes over $\cal Y^{\flat}$, equipped with a section over $\cal Y$.

\begin{cor}
\label{cor-qtw-local-triviality}
The tautological functor $\mathbf{QTw}_{/S}(\cal Y/\cal Y^{\flat})\rightarrow\mathbf{QTw}^{\et}_{/S}(\cal Y/\cal Y^{\flat})$ is an equivalence.
\end{cor}
\begin{proof}
We use the $\mathbb G_a$-incarnation of quasi-twistings, as well as their counterparts defined by \'etale locally trivial gerbes. For an affine $S$-scheme $T$, there holds
$$
\op H^1_{\et}(T, \mathbb G_a)=0,\quad\op H^2_{\et}(T, \mathbb G_a) = 0,
$$
so the prestacks $\op B^2\mathbb G_a$ and $\op B^2_{\et}\mathbb G_a$ (classifying \'etale locally trivial $\mathbb G_a$-gerbes on an affine $S$-scheme) are equivalent. It follows that the corresponding notions of quasi-twistings are also equivalent.
\end{proof}

\subsection{Modules over a quasi-twisting} We continue to assume $\cal Y\in\mathbf{PStk}_{\op{laft-def}/S}$ and $\cal T$ is a quasi-twisting over $\cal Y$. Our goal now is to define $\cal T\Mod$ as a DG category tensored over $\op{QCoh}(S)$.

\subsubsection{} We first proceed more generally and define ind-coherent sheaves ``twisted'' by a $\widehat{\mathbb G}_m$-gerbe.

Let $\cal Z\in\mathbf{PStk}_{\op{laft-def}/S}$, and $\widehat{\cal Z}$ be a $\widehat{\mathbb G}_m$-gerbe over $\cal Z$. Consider the canonical action of $\op B\mathbb G_m$ on $\mathbf{Vect}$, which induces an action of $\op B\widehat{\mathbb G}_m$. More formally, $\mathbf{Vect}$ can be regarded as a co-module object in $\mathbf{DGCat}_{\op{cont}}$ over the co-algebra $(\op{IndCoh}(\op B\widehat{\mathbb G}_m), m^!)$, where $m$ is the multiplication map on $\op B\widehat{\mathbb G}_m$.  The co-action
$$
\mathbf{Vect} \rightarrow \mathbf{Vect}\otimes\op{IndCoh}(\op B\widehat{\mathbb G}_m) \xrightarrow{\sim} \op{IndCoh}(\op B\widehat{\mathbb G}_m)
$$
is specified by $\chi\in\op{IndCoh}(\op B\widehat{\mathbb G}_m)$, the character sheaf induced from the map $\op B\widehat{\mathbb G}_m\rightarrow\op B\mathbb G_m$. (See \cite[\S1-2]{Be13} for notions pertaining to group actions on DG categories.)

Note that $\op{IndCoh}(\widehat{\cal Z})$ admits a $\op B\widehat{\mathbb G}_m$-action, so the product $\op{IndCoh}(\widehat{\cal Z})\otimes\mathbf{Vect}$ is again acted on by $\op B\widehat{\mathbb G}_m$. The corresponding co-simplicial system $\{\op{IndCoh}(\widehat{\cal Z}\times\op B\widehat{\mathbb G}_m^{\times n})\}_{[n]\in\Delta}$ has the following first few terms:
\begin{equation}
\label{eq-twisted-indcoh-system}
\xymatrix{
	\cdots  & \op{IndCoh}(\widehat{\cal Z}\times\op B\widehat{\mathbb G}_m^{\times 2}) \ar@<1.5ex>[l]\ar@<0.5ex>[l]\ar@<-0.5ex>[l]\ar@<-1.5ex>[l] & \op{IndCoh}(\widehat{\cal Z}\times\op B\widehat{\mathbb G}_m) \ar@<-2.5ex>[l]_-{(\op{act}\times 1)^!}\ar@<-1.5ex>[l]^-{(1\times m)^!}\ar@<1ex>[l]^-{\op{pr}_{12}^!\otimes\chi} & \op{IndCoh}(\widehat{\cal Z}). \ar@<-0.5ex>[l]_-{\op{act}^!}\ar@<0.5ex>[l]^-{\op{pr}_1^!\otimes\chi}
}
\end{equation}
We define the DG category $\op{IndCoh}(\cal Z)_{\widehat{\cal Z}}$ of \emph{$\widehat{\cal Z}$-twisted} ind-coherent sheaves on $\cal Z$ by the totalization of the above co-simplicial system. One sees immediately that $\op{IndCoh}(\cal Z)_{\widehat{\cal Z}}$ is tensored over $\op{QCoh}(S)$.

Since the functors associated to each face map $[n]\rightarrow[m]$ all admit left adjoints, we obtain:
$$
\op{IndCoh}(\cal Z)_{\widehat{\cal Z}} = \underset{[n]\in\Delta}{\lim}\,\op{IndCoh}(\widehat{\cal Z}\times\op B\widehat{\mathbb G}_m^{\times n}) \xrightarrow{\sim} \underset{[n]\in\Delta^{\op{op}}}{\op{colim}}\,\op{IndCoh}(\widehat{\cal Z}\times\op B\widehat{\mathbb G}_m^{\times n}),
$$
where we use the left adjoints to form the colimit.

\begin{rem}
The above colimit is taken in $\mathbf{DGCat}_{\op{cont}}$, and the forgetful functor from $\mathbf{DGCat}_{\op{cont}}$ to plain $\infty$-categories does not commute with colimits.
\end{rem}

\begin{rem}
\label{rem-trivial-gerbe-equiv}
Note that any (global) trivialization of the gerbe $\widehat{\cal Z}\rightarrow\cal Z$ gives rise to an equivalence $\op{IndCoh}(\cal Z)_{\widehat{\cal Z}} \xrightarrow{\sim} \op{IndCoh}(\cal Z)$.
\end{rem}

\begin{rem}
In \cite[\S1.7]{GL16}, a definition of a twisted presheaf of DG categories is given. We relate their definition to ours. For the presheaf over $\cal Z$:
$$
\op{IndCoh}_{/\cal Z} : (\mathbf{DGSch}^{\op{aff}}_{/\cal Z})^{\op{op}} \ni S \leadsto \op{IndCoh}(S)
$$
and a $\widehat{\mathbb G}_m$-gerbe $\widehat{\cal Z}$, the \emph{twisted sheaf of DG categories} $(\op{IndCoh}_{/\cal Z})_{\widehat{\cal Z}}$ is defined by
\begin{itemize}
	\item specifying its values on the category $\op{Split}(\widehat{\cal Z})$ of affine DG schemes $S\rightarrow\cal Z$ equipped with a lift to $\widehat{\cal Z}$, using the canonical $\op{Maps}(S, \op B\widehat{\mathbb G}_m)$-action on $\op{IndCoh}(S)$; and then
	\item applying $h$-descent\footnote{The authors of \cite{GL16} work with the \'etale topology instead.} along the basis $\op{Split}(\widehat{\cal Z})\rightarrow \mathbf{DGSch}^{\op{aff}}_{/\cal Z}$ to obtain a sheaf (in the $h$-topology) over $\mathbf{DGSch}^{\op{aff}}_{/\cal Z}$, denoted by $(\op{IndCoh}_{/\cal Z})_{\widehat{\cal Z}}$.
\end{itemize}
Thus we may calculate the global section $\mathbf{\Gamma}(\cal Z, (\op{IndCoh}_{/\cal Z})_{\widehat{\cal Z}})$ by the covering $\widehat{\cal Z}\rightarrow\cal Z$. The resulting co-simplicial system identifies with \eqref{eq-twisted-indcoh-system}. Hence the definition of $\widehat{\cal Z}$-twisted ind-coherent sheaves in \cite[\S1.7]{GL16} (adjusted to the $h$-topology) agrees with ours.
\end{rem}

\subsubsection{} Let $\cal T$ be a quasi-twisting over $\cal Y$, represented by the $\widehat{\mathbb G}_m$-gerbe $\widehat{\cal Y}^{\flat}\rightarrow\cal Y^{\flat}$. We denote by $\widehat{\cal Y}$ the $\widehat{\mathbb G}_m$-gerbe over $\cal Y$ pulled back along $\cal Y\rightarrow\cal Y^{\flat}$; it is equipped with a canonical trivialization.

We define the DG category of $\cal T$-modules by: $\cal T\Mod:= \op{IndCoh}(\cal Y^{\flat})_{\widehat{\cal Y}^{\flat}}$. There is a canonical functor:
$$
\mathbf{oblv}_{\cal T} : \cal T\Mod \rightarrow \op{IndCoh}(\cal Y)_{\widehat{\cal Y}} \xrightarrow{\sim} \op{IndCoh}(\cal Y),
$$
since $\widehat{\cal Y}^{\flat}$ is trivialized over $\cal Y$, and Remark \ref{rem-trivial-gerbe-equiv} identifies the corresponding twisted category with $\op{IndCoh}(\cal Y)$.

\begin{prop}
The functor $\mathbf{oblv}_{\cal T}$ admits a left adjoint $\mathbf{ind}_{\cal T}$, and the pair $(\mathbf{ind}_{\cal T}, \mathbf{oblv}_{\cal T})$ is monadic.
\end{prop}
\begin{proof}
The functor $\mathbf{oblv}_{\cal T}$ is by definition the totalization of the $!$-pullback functors:
$$
(\pi^{(n)})^!: \op{IndCoh}(\widehat{\cal Y}^{\flat} \times \op B\widehat{\mathbb G}_m^{\times n}) \rightarrow \op{IndCoh}(\widehat{\cal Y} \times \op B\widehat{\mathbb G}_m^{\times n}),\text{ where }\pi^{(n)} : \widehat{\cal Y}\times\op B\widehat{\mathbb G}_m^{\times n} \rightarrow \widehat{\cal Y}^{\flat}\times\op B\widehat{\mathbb G}_m^{\times n}.
$$
Each $(\pi^{(n)})^!$ admits a left adjoint $\pi^{(n)}_{*,\op{IndCoh}}$. Furthermore, the diagram induced from an arbitrary face map:
$$
\xymatrix@C=1.5em@R=1.5em{
\op{IndCoh}(\widehat{\cal Y}\times\op B\widehat{\mathbb G}_m^{\times n}) \ar[r]\ar[d]^{\pi_{*,\op{IndCoh}}^{(n)}} & \op{IndCoh}(\widehat{\cal Y}\times\op B\widehat{\mathbb G}_m^{\times m}) \ar[d]^{\pi_{*,\op{IndCoh}}^{(m)}} \\
\op{IndCoh}(\widehat{\cal Y}^{\flat}\times\op B\widehat{\mathbb G}_m^{\times n}) \ar[r] & \op{IndCoh}(\widehat{\cal Y}^{\flat}\times\op B\widehat{\mathbb G}_m^{\times m})
}
$$
which \emph{a priori} commutes up to a natural transformation, actually commutes. Hence $\mathbf{oblv}_{\cal T}$ admits a left adjoint $\mathbf{ind}_{\cal T}:=\op{Tot}(\pi_{*,\op{IndCoh}}^{(n)})$. We now prove:
\begin{itemize}
	\item $\mathbf{oblv}_{\cal T}$ is conservative; this is because all other arrows in the following commutative diagram:
	$$
	\xymatrix@C=1.5em@R=1.5em{
	\op{IndCoh}(\cal Y^{\flat})_{\widehat{\cal Y}^{\flat}} \ar[d]^{\op{ev}^0} \ar[r]^-{\mathbf{oblv}_{\cal T}} & \op{IndCoh}(\cal Y)_{\widehat{\cal Y}} \ar[d]^{\op{ev}^0} \\
	\op{IndCoh}(\cal Y^{\flat}) \ar[r]^{(\pi^{(0)})^!} & \op{IndCoh}(\cal Y)
	}
	$$
	are conservative, hence so is $\mathbf{oblv}_{\cal T}$.
	
	\item $\mathbf{oblv}_{\cal T}$ preserves colimits; this is obvious as we work in $\mathbf{DGCat}_{\op{cont}}$.
\end{itemize}
It follows that that the pair $(\mathbf{ind}_{\cal T}, \mathbf{oblv}_{\cal T})$ is monadic, by the Barr-Beck-Lurie theorem.
\end{proof}

We may regard $\op U(\cal T):=\mathbf{oblv}_{\cal T}\circ\mathbf{ind}_{\cal T}$ as an algebra object in $\op{End}(\op{IndCoh}(\cal Y))$, and the DG category $\cal T\Mod$ identifies with that of $\op U(\cal T)$-module objects in $\op{IndCoh}(\cal Y)$. We call $\op U(\cal T)$ the \emph{universal envelope} of $\cal T$.

\subsection{Comparison with the classical notion} Suppose $Y\in\mathbf{Sch}_{/S}^{\op{ft}}$ is classical. Let $\cal L$ be a classical Lie algebroid over $Y$ and $\cal Y^{\flat}\in\mathbf{FrmMod}_{/S}(Y)$ be the formal moduli problem associated to $\cal L$, under the embedding \eqref{eq-liealgd-to-frmmod}.

\subsubsection{} Given a formal moduli problem $\widehat{\cal Y}^{\flat} \rightarrow \cal Y^{\flat}$ such that $\mathbb T_{Y/\widehat{\cal Y}^{\flat}} \in \Upsilon_Y(\op{QCoh}(Y)^{\heartsuit})$, one can functorially assign a classical Lie algebroid $\widehat{\cal L}$ equipped with a map $\widehat{\cal L}\rightarrow\cal L$. Furthermore, a morphism $\widehat{\cal Y}^{\flat}\times\op B\widehat{\mathbb G}_m \rightarrow \widehat{\cal Y}^{\flat}$ in $\mathbf{FrmMod}_{/S}(Y)$ induces a map
\begin{equation}
\label{eq-qtw-induce-cl}
\widehat{\cal L}\oplus\cal O_Y \rightarrow \widehat{\cal L},\quad (l,f)\leadsto l+f\mathbf 1
\end{equation}
where $\mathbf 1$ is the image of $(0,f)$ in $\widehat{\cal L}$. If the morphism $\widehat{\cal Y}^{\flat}\times\op B\widehat{\mathbb G}_m \rightarrow \widehat{\cal Y}^{\flat}$ realizes $\widehat{\cal Y}^{\flat}$ as a $\widehat{\mathbb G}_m$-gerbe over $\cal Y^{\flat}$, then we see that $\cal O_Y\rightarrow\widehat{\cal L}$, $f\leadsto f\mathbf 1$ is the kernel of the canonical map $\widehat{\cal L}\rightarrow\cal L$. The fact that \eqref{eq-qtw-induce-cl} preserves Lie bracket then implies $\cal O_Y$ is central inside $\widehat{\cal L}$. In other words, the map $\widehat{\cal L}\rightarrow\cal L$ is a central extension of classical Lie algebroids.

\subsubsection{} Now, given any object in $\mathbf{QTw}_{/S}(Y/\cal Y^{\flat})$, we claim that the corresponding formal moduli problem $\widehat{\cal Y}^{\flat}$ satisfies the property that $\mathbb T_{Y/\widehat{\cal Y}^{\flat}}$ lies in $\Upsilon_Y(\op{QCoh}(Y)^{\heartsuit})$. Indeed, we have a canonical triangle in $\op{IndCoh}(Y)$:
$$
\omega_Y\cong\mathbb T_{\widehat{\cal Y}^{\flat}/\cal Y^{\flat}}\big|_Y \rightarrow \mathbb T_{Y/\widehat{\cal Y}^{\flat}} \rightarrow \mathbb T_{Y/\cal Y^{\flat}}
$$
and the outer terms lie in the essential image of $\op{QCoh}(Y)^{\heartsuit}$. Hence the previous discussion shows that we have a functor:
\begin{equation}
\label{eq-qtw-comparison}
\mathbf{QTw}_{/S}(Y/\cal Y^{\flat}) \rightarrow \mathbf{QTw}_{/S}^{\op{cl}}(Y/\cal L).
\end{equation}

\begin{prop}
\label{prop-qtw-comparison}
The functor \eqref{eq-qtw-comparison} is an equivalence of categories.
\end{prop}
\begin{proof}
We explicitly construct the functor inverse to \eqref{eq-qtw-comparison}. Namely, given a central extension $\widehat{\cal L}$ of $\cal L$, we need to equip its corresponding formal moduli problem $\widehat{\cal Y}^{\flat}$ with the structure of a $\widehat{\mathbb G}_m$-gerbe over $\cal Y^{\flat}$. As before, the action map $\widehat{\cal Y}^{\flat}\times\op B\widehat{\mathbb G}_m\rightarrow\widehat{\cal Y}^{\flat}$ arises from the morphism of classical Lie algebroids over $Y$:
$$
\widehat{\cal L} \oplus \cal O_Y \rightarrow \widehat{\cal L} ,\quad (l,f)\leadsto l + f\mathbf 1.
$$
The morphism induced by action and projection $\widehat{\cal Y}^{\flat}\times\op B\widehat{\mathbb G}_m\rightarrow\widehat{\cal Y}^{\flat}\underset{\cal Y^{\flat}}{\times}\widehat{\cal Y}^{\flat}$ is an isomorphism since the same holds for the corresponding map of classical Lie algebroids:
$$
\widehat{\cal L}\oplus\cal O_Y \rightarrow \widehat{\cal L}\underset{\cal L}{\times}\widehat{\cal L},\quad (l,f)\leadsto (l+f\mathbf 1, l).
$$

It remains to show that $\widehat{\cal Y}^{\flat}\rightarrow\cal Y^{\flat}$ admits a section over any affine DG scheme $T$ mapping to $\cal Y^{\flat}$. We shall deduce the existence of this section from the following claim:

\begin{claim}
The morphism $\widehat{\cal Y}^{\flat}\rightarrow\cal Y^{\flat}$ is formally smooth.
\end{claim}

\noindent
Indeed, let $T$ be any affine DG scheme with a morphism $\widehat y : T\rightarrow\widehat{\cal Y}^{\flat}$. By the criterion of formal smoothness \eqref{eq-formal-smoothness-criterion}, we ought to show $\op{Maps}(\cal T^*_{\widehat{\cal Y}^{\flat}/\cal Y^{\flat}}\big|_{\widehat y}, \cal F)\in\mathbf{Vect}^{\le 0}$ for all $\cal F\in\op{QCoh}(T)^{\heartsuit}$. The Cartesian square:
$$
\xymatrix@C=1.5em@R=1.5em{
T \ar[r]^-{(\widehat y, \widehat y)} & \widehat{\cal Y}^{\flat}\underset{\cal Y^{\flat}}{\times}\widehat{\cal Y}^{\flat} \ar[d]\ar[r] & \widehat{\cal Y}^{\flat} \ar[d] \\
& \widehat{\cal Y}^{\flat} \ar[r] & \cal Y^{\flat}
}
$$
together with the isomorphism above gives: 
$$
\cal T^*_{\widehat{\cal Y}^{\flat}/\cal Y^{\flat}}\big|_{\widehat y}\xrightarrow{\sim} \cal T^*_{\widehat{\cal Y}^{\flat}\underset{\cal Y^{\flat}}{\times}\widehat{\cal Y}^{\flat}/\widehat{\cal Y}^{\flat}}\big|_{(\widehat y,\widehat y)} \xrightarrow{\sim} \cal T_{\widehat{\cal Y}^{\flat}\times\op B\widehat{\mathbb G}_m/\widehat{\cal Y^{\flat}}}^*\big|_{(\widehat y, 1)} \xrightarrow{\sim} \cal O_T[-1].
$$
One deduces from this the required degree estimate.

Using the claim, we will construct a section of $\widehat{\cal Y}^{\flat} \rightarrow\cal Y^{\flat}$ over $T\rightarrow\cal Y^{\flat}$ as follows. First consider the fiber product $T\underset{\cal Y^{\flat}}{\times}\cal Y$, which is equipped with a nil-isomorphism to $T$. We obtain a solid commutative diagram:
$$
\xymatrix@C=1.5em@R=1.5em{
T^{\op{red}} \ar[r]\ar[dr] & T\underset{\cal Y^{\flat}}{\times}\cal Y \ar[d] \ar[r] & \cal Y \ar[r]\ar[d] & \widehat{\cal Y}^{\flat} \ar[dl] \\
& T \ar[r]\ar@{.>}[urr] & \cal Y^{\flat}
}
$$
Formal smoothness now implies the existence of the dotted arrow.
\end{proof}

In particular, the $\infty$-category $\mathbf{QTw}_{/S}(Y/\cal Y^{\flat})$ is an ordinary category.

\begin{rem}
By letting $\cal L = \cal T_{Y/S}$ be the tangent Lie algebroid, we obtain from Proposition \ref{prop-qtw-comparison} the fact that Picard algebroids identify with twistings on classical schemes locally of finite type. The same result is established in \cite[\S6.5]{GR14} using a computation involving de Rham cohomology.
\end{rem}

\medskip

\section{How to take quotient of a Lie algebroid?}
\label{sec-quotient}

This section is devoted to the study of quotients of Lie algebroids, in both classical and DG settings. The set-up involves an $H$-torsor $Y\rightarrow Z$ and a Lie algebroid $\cal L$ over $Y$. With additional data on $\cal L$, there exists a \emph{quotient Lie algebroid} over $Z$. The quotient procedure we shall describe take as input a map $\eta : \fr k\otimes\cal O_Y\rightarrow\cal L$, where $\fr k$ is an arbitrary Lie algebra. It generalizes two existing notions---\emph{weak} and \emph{strong} quotients---both considered by Beilinson and Bernstein \cite{BB93}.

For technical reasons involving $\infty$-type schemes, we shall construct two quotient functors:
\begin{itemize}
	\item $\mathbf Q^{(\fr k, H)}_{\op{inj}}$, which is a classical procedure that works in the case where $\eta$ is injective;
	\item $\mathbf Q^{(H,H^{\flat})}$, which is its geometric counterpart for $Y$ locally of finite type,
\end{itemize}
and we check that they agree in overlapping cases. A geometric procedure that works in full generality should exist as soon as the theory in \cite{GR16} is extended to $\infty$-type situations.

\subsection{$(\fr k,H)$-Lie algebroids} We continue to work over a base (classical) scheme $S\in\Sch^{\op{ft}}_{/k}$.

\subsubsection{}
\label{sec-action-pair}
A \emph{classical action pair} $(\fr k, H)$ consists of an affine group scheme $H$ over $S$, an $\cal O_S$-linear Lie algebra $\fr k$ acted on by $H$, as well as a morphism of Lie algebras:
\begin{equation}
\label{eq-map-in-pair}
\fr k\rightarrow\fr h:=\op{Lie}(H)
\end{equation}
with the following properties:
\begin{itemize}
	\item \eqref{eq-map-in-pair} is $H$-equivariant, where $\fr h$ is equipped with the adjoint $H$-action;
	\item the $\fr k$-action on itself induced from \eqref{eq-map-in-pair} is the adjoint action.
\end{itemize}

\begin{rem}
This datum is superficially similar to that of a Harish-Chandra pair, but they serve very different purposes.
\end{rem}

\begin{eg}
Fix an $S$-point $\fr g^{\kappa}$ of $\op{Gr}_{\op{Lag}}^G(\fr g\oplus\fr g^*)$ (see \S\ref{sec-quantum-parameter}). Then we have a classical action pair $(\fr g^{\kappa}[\![t]\!], S\times G[\![t]\!])$, where the morphism \eqref{eq-map-in-pair} is induced from the projection $\fr g^{\kappa}\rightarrow\fr g\otimes\cal O_S$. All classical action pairs considered in this paper are variants of $(\fr g^{\kappa}[\![t]\!], S\times G[\![t]\!])$. Note that the group scheme $S\times G[\![t]\!]$ is \emph{not} of finite type.
\end{eg}

\subsubsection{}
\label{sec-normal-subpair}
The notion of a morphism $(\fr k^0, H^0)\rightarrow(\fr k, H)$ of classical action pairs is obvious. We say that $(\fr k^0, H^0)$ is a \emph{normal} subpair if $\fr k^0\hookrightarrow\fr k$ is an ideal, $H^0\hookrightarrow H$ is a normal subgroup, the $H$-action stabilizes $\fr k^0$, and $H^0$ acts trivially on $\fr k/\fr k^0$. This definition means precisely that a normal subpair fits into an \emph{exact sequence} (in the obvious sense):
\begin{equation}
\label{eq-normal-subpair}
1 \rightarrow (\fr k^0, H^0)\rightarrow(\fr k, H) \rightarrow (\fr k_0, H_0) \rightarrow 1.
\end{equation}

\subsubsection{} Let $Y\in\mathbf{Sch}_{/S}$ be acted on by $H$. Recall that every $H$-equivariant $\cal O_Y$-module $\cal F$ admits an $\fr h$-action by derivations. Specializing to $\cal O_Y$ itself, we obtain a canonical map:
\begin{equation}
\label{eq-lie-algebra-maps-to-tangent}
\fr h\otimes\cal O_Y \rightarrow \cal T_{Y/S}.
\end{equation}
On the other hand, the $\cal O_Y$-module $\cal T_{Y/S}$ admits a canonical $H$-equivariance structure, given by pushforward of tangent vectors.

\subsubsection{}
\label{sec-kh-lie-algebroid} A \emph{$(\fr k, H)$-Lie algebroid} on $Y$ consists of a Lie algebroid $\cal L\in\mathbf{LieAlgd}_{/S}(Y)$, an $H$-equivariance structure on the underlying $\cal O_Y$-module of $\cal L$, and a morphism $\eta : \fr k\otimes\cal O_Y\rightarrow\cal L$ of $H$-equivariant $\cal O_Y$-modules, subject to the following conditions:
\begin{itemize}
	\item the $H$-equivariance structure on $\cal L$ is compatible with its Lie bracket;
	\item the anchor map $\sigma$ of $\cal L$ intertwines the $H$-equivariance structures on $\cal L$ and $\cal T_{Y/S}$;
	\item the following diagram is commutative:
	\begin{equation}
	\label{eq-algebroid-diagram}
	\xymatrix@C=1em@R=1em{
		& \cal L \ar[dr]^{\sigma} & \\
	\fr k\otimes\cal O_Y \ar[ur]^{\eta}\ar[dr]_{\eqref{eq-map-in-pair}} & & \cal T_{Y/S} \\
		& \fr h\otimes\cal O_Y \ar[ur]_{\eqref{eq-lie-algebra-maps-to-tangent}}
	}
	\end{equation}
	\item $\eta$ is compatible with the Lie bracket on $\cal L$ in the following sense: given $\xi\in\fr k\otimes\cal O_Y$ and $l\in \cal L$, there holds:
	\begin{equation}
	\label{eq-algebroid-identity}
	[\eta(\xi), l] = \xi_{\fr h}\cdot l \in \cal L
	\end{equation}
	where $\xi_{\fr h}$ is the image of $\xi$ in $\fr h\otimes\cal O_Y$ along \eqref{eq-map-in-pair}, and $\xi_{\fr h}\cdot l$ denotes the action of $\xi_{\fr h}$ on $l$ coming from the equivariance structure.
\end{itemize}

\noindent
We will frequently write a $(\fr k, H)$-Lie algebroid as $(\cal L,\eta)$, in order to emphasize the dependence on $\eta$. The category of $(\fr k, H)$-Lie algebroids on $Y$ is denoted by $\mathbf{LieAlgd}^{(\fr k, H)}_{/S}(Y)$.

Given another scheme $Y'\in\mathbf{Sch}_{/S}$ acted on by $H$ and an $H$-equivariant morphism $Y'\rightarrow Y$, one can form the pullback of a $(\fr k, H)$-Lie algebroid in a way compatible with the forgetful functor to plain Lie algebroids.

\subsection{Quotient I} We describe how to form the \emph{quotient} of a $(\fr k, H)$-Lie algebroid when the morphism $\eta$ is \emph{injective}. Denote the category of such $(\fr k, H)$-Lie algebroid by $\mathbf{LieAlgd}^{(\fr k, H)}_{\op{inj}/S}(Y)$.

\subsubsection{} Suppose $Z\in\mathbf{Sch}_{/S}$ and $Y$ is an $H$-torsor over $Z$. Since $H$ is affine, the projection $\pi:Y\rightarrow Z$ is an affine, faithfully flat cover (in particular, fpqc). We will define a \emph{quotient} functor:
\begin{equation}
\label{eq-quotient-inj}
\mathbf{Q}_{\op{inj}}^{(\fr k, H)} : \mathbf{LieAlgd}_{\op{inj}/S}^{(\fr k, H)}(Y) \rightarrow\mathbf{LieAlgd}_{/S}(Z) 
\end{equation}
on each $(\cal L, \eta) \in \mathbf{LieAlgd}_{\op{inj}}^{(\fr k, H)}(Y/S)$ by the following procedure:
\begin{itemize}
	\item (\emph{$\cal O_Z$-module and anchor map}) We have a morphism of $H$-equivariant $\cal O_Y$-modules:
	$$
	\cal L/(\fr k\otimes\cal O_Y) \rightarrow \cal T_{Y/S}/(\fr h\otimes\cal O_Y)\xrightarrow{\sim}\pi^*\cal T_{Z/S}
	$$
	by \eqref{eq-algebroid-diagram}. Let $\cal L_0$ denote the fpqc descent of $\cal L/(\fr k\otimes\cal O_Y)$ to $Z$, so we obtain a map of $\cal O_Z$-modules $\sigma_0 : \cal L_0 \rightarrow \cal T_{Z/S}$. The image of $(\cal L, \eta)$ under $\mathbf{Q}_{\op{inj}}^{(\fr k, H)}$ is supposed to have underlying $\cal O_Z$-module $\cal L_0$ and anchor map $\sigma_0$.
	\item (\emph{Lie bracket}) Since $\pi$ is affine, it suffices to define an $\cal O_S$-linear Lie bracket on $\pi^{-1}\cal L_0$. Consider the embedding:
	$$
	\pi^{-1}\cal L_0\hookrightarrow \pi^*\cal L_0 \xrightarrow{\sim} \cal L/(\fr k\otimes\cal O_Y).
	$$
	The Lie bracket on $\cal L$ will induce one on $\pi^{-1}\cal L_0$ if $[\fr k\otimes\cal O_Y, \pi^{-1}\cal L_0]=0$ in $\cal L$. The latter identity is guaranteed by \eqref{eq-algebroid-identity}.
\end{itemize}
We omit checking that this procedure gives rise to a well-defined functor $\mathbf{Q}_{\op{inj}}^{(\fr k, H)}$.

Given a \emph{flat} morphism of schemes $f:Z'\rightarrow Z$, we set $Y':=Z'\underset{Z}{\times}Y$ which is an $H$-torsor over $Z'$. The map $\tilde f : Y'\rightarrow Y$ is $H$-equivariant, and the pullback of $(\cal L,\eta)\in\mathbf{LieAlgd}_{\op{inj}/S}^{(\fr k, H)}(Y)$ along $\tilde f$ lies in $\mathbf{LieAlgd}_{\op{inj}/S}^{(\fr k, H)}(Y')$. Furthermore, $\mathbf{Q}_{\op{inj}}^{(\fr k, H)}$ is compatible with pullbacks along $f$ and $\tilde f$.

\begin{rem}
Since Lie algebroids are smooth local objects (see \cite{BB93}) and $\mathbf{Q}_{\op{inj}}^{(\fr k, H)}$ is compatible with flat pullbacks, we may generalize $\mathbf{Q}_{\op{inj}}^{(\fr k, H)}$ to the case where $\cal Z:=Y/H$ is representable by an algebraic stack (i.e., smooth locally a scheme).
\end{rem}

\begin{rem}
The special case where the classical action pair is given by $(\fr h, H)$ with \eqref{eq-map-in-pair} being the identity map, has been studied in \cite{BB93} under the name \emph{strong quotient}. Note that when $H$ acts freely on $Y$, the map $\eta$ is automatically injective.
\end{rem}

\begin{eg}
\label{eg-weak-quotient-1}
Another instance of the functor \eqref{eq-quotient-inj} is the \emph{weak quotient}. This is the case where $\fr k=0$. The only data needed in defining a $(0, H)$-Lie algebroid are a Lie algebroid $\cal L\in\mathbf{LieAlgd}_{/S}(Y)$, together with an $H$-equivariance structure on the underlying $\cal O_Y$-module of $\cal L$, subject to the first two conditions in \S\ref{sec-kh-lie-algebroid}.

Suppose $Y/H$ is representable by an algebraic stack. Then the resulting quotient $\mathbf Q_{\op{inj}}^{(0,H)}(\cal L)$ has underlying $\cal O_{Y/H}$-module the descent of (the $\cal O_Y$-module) $\cal L$ along $Y\rightarrow Y/H$.
\end{eg}

\subsubsection{}
\label{sec-quotient-classical-univ} We now characterize the object $\mathbf Q_{\op{inj}}^{(\fr k, H)}(\cal L) \in \mathbf{LieAlgd}_{/S}(Z)$ by a universal property. Consider an arbitrary Lie algebroid $\cal M\in\mathbf{LieAlgd}_{/S}(Z)$. We can equip $\pi^!_{\mathbf{LieAlgd}}\cal M$ with the structure of a $(\fr k, H)$-Lie algebroid as follows:
\begin{itemize}
	\item regarding $\pi^!_{\mathbf{LieAlgd}}\cal M$ as the $\cal O_Y$-module $\pi^*\cal M\underset{\pi^*\cal T_{Z/S}}{\times}\cal T_{Y/S}$, the $H$-equivariance structure is a combination of the natural $H$-equivariance structures on $\pi^*\cal M$ and $\cal T_{Y/S}$;
	\item the morphism $\eta : \fr k\otimes\cal O_Y \rightarrow \pi^!_{\mathbf{LieAlgd}}(\cal M)$ is a combination of the \emph{zero} map $\fr k\otimes\cal O_Y\rightarrow\pi^*\cal M$ and the composition $\fr k\otimes\cal O_Y \rightarrow \fr h\otimes\cal O_Y \rightarrow\cal T_{Y/S}$.
\end{itemize}
Note that $\pi^!_{\mathbf{LieAlgd}}\cal M \in \mathbf{LieAlgd}_{/S}^{(\fr k, H)}(Y)$ does \emph{not} belong to $\mathbf{LieAlgd}_{\op{inj}/S}^{(\fr k,H)}(Y)$ in general.

\begin{prop}
\label{prop-quotient-classical-univ}
There is a natural bijection:
\begin{equation}
\label{eq-quotient-classical-univ}
\op{Maps}_{\mathbf{LieAlgd}_{/S}(Z)}(\mathbf Q_{\op{inj}}^{(\fr k, H)}(\cal L), \cal M) \xrightarrow{\sim} \op{Maps}_{\mathbf{LieAlgd}_{/S}^{(\fr k,H)}(Y)}(\cal L, \pi^!_{\mathbf{LieAlgd}}\cal M)
\end{equation}
\end{prop}
\begin{proof}
A morphism $\mathbf Q^{(\fr k, H)}_{\op{inj}}(\cal L) \rightarrow \cal M$ is equivalent to an $H$-equivariant map $\phi : \cal L/\fr k\otimes\cal O_Y \rightarrow \pi^*\cal M$ preserving the Lie bracket on $H$-invariant sections. We \emph{claim} that such datum is equivalent to a morphism $\widetilde{\phi} : \cal L\rightarrow\pi^!_{\mathbf{LieAlgd}}\cal M$ of $(\fr k, H)$-Lie algebroids.

Indeed, given $\phi$, the map $\widetilde{\phi}$ is uniquely determined by the properties that the following diagrams commute:
$$
\xymatrix@C=1.5em@R=1.5em{
	\cal L \ar[r]^-{\widetilde{\phi}}\ar[d] & \pi^!_{\mathbf{LieAlgd}}\cal M \ar[d] \\
	\cal L/\fr k\otimes\cal O_Y \ar[r]^-{\phi} & \pi^*\cal M
}
\quad\quad
\xymatrix@C=1.5em@R=1.5em{
	\cal L \ar[r]^-{\widetilde{\phi}}\ar[rd]_-{\sigma} & \pi^!_{\mathbf{LieAlgd}}\cal M \ar[d] \\
	 & \cal T_{Y/S}.
}
$$
Furthermore, $\widetilde{\phi}$ preserves the Lie bracket on $\cal L$, because $\cal L$ is generated over $\cal O_Y$ by $H$-invariant sections and on such sections, the Lie bracket factors through $\cal L/\fr k\otimes\cal O_Y$ and is preserved by $\phi$. Conversely, given $\widetilde{\phi}$, the map $\phi$ is uniquely determined by the first commutative diagram above.
\end{proof}

\subsubsection{} Suppose we are given an exact sequence \eqref{eq-normal-subpair} of classical action pairs, and an object $(\cal L,\eta)\in\mathbf{LieAlgd}_{\op{inj}/S}^{(\fr k, H)}(Y)$. Assume also that $Y/H$ is representable by an algebraic stack. Note that:
\begin{itemize}
	\item $Y/H^0$ admits an $H_0$-action, realizing it as an $H_0$-torsor over $Y/H$ (in particular, $Y/H^0$ is also representable by an algebraic stack);
	\item there is an induced $(\fr k_0, H_0)$-Lie algebroid structure on $\mathbf{Q}_{\op{inj}}^{(\fr k^0, H^0)}(\cal L)$, for which the structure map
	$$
	\eta_0 : \fr k_0\otimes\cal O_{Y/H^0} \rightarrow \mathbf{Q}_{\op{inj}}^{(\fr k^0, H^0)}(\cal L)
	$$
	is again injective, i.e., $(\mathbf{Q}_{\op{inj}}^{(\fr k^0, H^0)}(\cal L), \eta_0)\in \mathbf{LieAlgd}_{\op{inj}/S}^{(\fr k_0, H_0)}(Y/H^0)$.
\end{itemize}
We have a version of the second isomorphism theorem:

\begin{prop}
\label{prop-normal-subpair-cl}
There is a natural isomorphism:
$$
\mathbf Q_{\op{inj}}^{(\fr k_0, H_0)}\circ\mathbf Q_{\op{inj}}^{(\fr k^0, H^0)}(\cal L) \xrightarrow{\sim} \mathbf Q_{\op{inj}}^{(\fr k, H)}(\cal L).
$$
\end{prop}
\begin{proof}
As $\cal O_{Y/H^0}$-modules, the cokernel of $\eta_0$ identifies with the descent of $\cal L/\fr k\otimes\cal O_Y$ along $Y\rightarrow Y/H^0$ since the latter map is faithfully flat. Hence the underlying $\cal O_{Y/H}$-module of $\mathbf Q_{\op{inj}}^{(\fr k_0, H_0)}\circ\mathbf Q_{\op{inj}}^{(\fr k^0, H^0)}(\cal L)$ agrees with that of $\mathbf Q_{\op{inj}}^{(\fr k, H)}(\cal L)$. Identifications of the anchor maps and the Lie brackets are immediate.
\end{proof}

\subsubsection{} Suppose we have a classical quasi-twisting \eqref{eq-classical-qtw} over $Y$, where both Lie algebroids $\widehat{\cal L}$ and $\cal L$ have the structure of $(\fr k, H)$-algebroids, and $\widehat{\cal L}\rightarrow\cal L$ is a morphism of such. In particular, the structure map $\widehat{\eta} : \fr k\otimes\cal O_Y\rightarrow\widehat{\cal L}$ is a lift of $\eta$. Hence, if $(\cal L, \eta)\in\mathbf{LieAlgd}^{(\fr k, H)}_{\op{inj}/S}(Y)$, then so does $(\widehat{\cal L}, \widehat{\eta})$. For fixed $(\cal L,\eta)$, we denote the category of classical quasi-twistings with this additional structure by $\mathbf{QTw}^{(\fr k,H)}_{/S}(Y/\cal L)$.

Assuming that $\cal Z:=Y/H$ is represented by an algebraic stack. Then the quotient Lie algebroids again form a central extension:
$$
0 \rightarrow \cal O_{Y/H} \rightarrow \mathbf Q_{\op{inj}}^{(\fr k, H)}(\widehat{\cal L}) \rightarrow \mathbf Q_{\op{inj}}^{(\fr k, H)}(\cal L) \rightarrow 0.
$$
Therefore, we may regard $\mathbf Q_{\op{inj}}^{(\fr k, H)}$ as a functor from $\mathbf{QTw}_{/S}^{(\fr k,H)}(Y/\cal L)$ to $\mathbf{QTw}_{/S}(\cal Z/\mathbf Q_{\op{inj}}^{(\fr k, H)}(\cal L))$.

\begin{rem}
When $Y$ is placid and $\fr k$ is a topological Lie algebra over $\cal O_S$, we can adapt the above definitions to make sense of a Tate $(\fr k, H)$-Lie algebroid $\cal L$ (c.f.~\S\ref{sec-tate-liealgd}). In particular, $\eta$ will be a map out of the completed tensor product $\fr k\widehat{\otimes}\cal O_Y\rightarrow\cal L$.

We do not discuss how to keep track of the topology in the (analogously defined) quotient $\mathbf Q_{\op{inj}}^{(\fr k,H)}(\cal L)$, since all quotients considered in this paper have the properties that $Y/H$ is locally of finite type and $\mathbf Q_{\op{inj}}^{(\fr k, H)}(\cal L)$ should be discrete.
\end{rem}

\subsection{$(H,H^{\flat})$-formal moduli problems}
We now study the geometric version of quotient of Lie algebroids. Recall the $\infty$-category $\mathbf{FrmMod}_{/S}$ of \S\ref{sec-ant}.

\subsubsection{}
\label{sec-geom-action-pair} We call a group object $(H, H^{\flat})$ in $\mathbf{FrmMod}_{/S}$ a \emph{geometric action pair} if $H$ is a group \emph{scheme} locally of finite type. Explicitly, a geometric action pair consists of a group scheme $H$, a group prestack $H^{\flat}\in\mathbf{PStk}_{\op{laft-def}/S}$, and a nil-isomorphism $H\rightarrow H^{\flat}$ that is a group homomorphism.

\subsubsection{}
\label{sec-classical-to-geom-action-pair}
We will functorially construct a geometric action pair from any classical action pair $(\fr k, H)$, where $H$ is locally of finite type.

Indeed, there is a morphism $\exp(\fr k)\rightarrow H$ coming from the composition $\exp(\fr k)\rightarrow\exp(\fr h)\rightarrow H$. Furthermore, the $H$-action on $\exp(\fr k)$ equips the prestack quotient $H^{\flat} := H/\exp(\fr k)$ with a group structure, such that $H\rightarrow H^{\flat}$ is a group morphism. Note that Lemma \ref{lem-smooth-comparison-of-quotients} identifies $H^{\flat}$ with $\op B_H(H\times\exp(\fr k)^{\bullet})$; in particular, $H^{\flat}\in\mathbf{PStk}_{\op{laft-def}/S}$, so $(H,H^{\flat})$ is a geometric action pair.

\begin{lem}
The category of classical action pairs is the full subcategory of geometric action pairs $(H,H^{\flat})$, for which $\mathbb T_{H/H^{\flat}}$ belongs to $\Upsilon_H(\op{QCoh}(H)^{\heartsuit})$.
\end{lem}
\begin{proof}
We explicitly construct the inverse functor. Given a geometric action pair $(H,H^{\flat})$ for which $\mathbb T_{H/H^{\flat}}\in\Upsilon_H(\op{QCoh}(H)^{\heartsuit})$, we can functorially associate a classical Lie algebroid $\cal L$ over $H$. The following Cartesian diagrams:
$$
\xymatrix@C=1.5em@R=1.5em{
H \underset{S}{\times} H \ar[r]\ar[d]^m & H^{\flat}\underset{S}{\times} H \ar[d]^{\op{act}} \\
H \ar[r] & H^{\flat}
}\quad
\xymatrix@C=1.5em@R=1.5em{
H \underset{S}{\times} H \ar[r]\ar[d]^m & H\underset{S}{\times} H^{\flat} \ar[d]^{\op{act}} \\
H \ar[r] & H^{\flat}
}
$$
equip the underlying $\cal O_H$-module of $\cal L$ with right, respectively left, $H$-equivariance structures. Hence we may realize $\cal L$ as $\fr k\otimes\cal O_H$ where $\fr k$ is an $\cal O_S$-module equipped with an $H$-action. The Lie bracket on $\fr k$ comes from the Lie algebroid bracket on $\cal L$. We omit checking that these data make $(\fr k, H)$ into a classical action pair.
\end{proof}

\subsubsection{}
For a geometric action pair $(H,H^{\flat})$, we define $\mathbf{FrmMod}_{/S}^{(H,H^{\flat})}$ to be the $\infty$-category of objects in $\mathbf{FrmMod}_{/S}$ equipped with an $(H,H^{\flat})$-action. Explicitly, an object of $\mathbf{FrmMod}_{/S}^{(H,H^{\flat})}$ consists of the following data:
\begin{itemize}
	\item $\cal Y,\cal Y^{\flat}\in\mathbf{PStk}_{\op{laft-def}/S}$ together with a nil-isomorphism $\cal Y\rightarrow\cal Y^{\flat}$;
	\item an $H$-action on $\cal Y$, and an $H^{\flat}$-action on $\cal Y^{\flat}$, such that the morphism $\cal Y\rightarrow\cal Y^{\flat}$ intertwines them.
\end{itemize}
Note that there is a functor
\begin{equation}
\label{eq-frmmod-action-to-pstk-action}
\mathbf{FrmMod}_{/S}^{(H,H^{\flat})}\rightarrow\mathbf{PStk}_{\op{laft-def}/S}^H,\quad (\cal Y,\cal Y^{\flat})\leadsto\cal Y
\end{equation}
where $\mathbf{PStk}_{\op{laft-def}/S}^H$ denotes the $\infty$-category of objects in $\mathbf{PStk}_{\op{laft-def}/S}$ equipped with an $H$-action. The fiber of \eqref{eq-frmmod-action-to-pstk-action} at $\cal Y$ is denoted by $\mathbf{FrmMod}_{/S}^{(H,H^{\flat})}(\cal Y)$. Informally, $\mathbf{FrmMod}_{/S}^{(H,H^{\flat})}(\cal Y)$ is the $\infty$-category of formal moduli problems $\cal Y^{\flat}$ equipped with an $H^{\flat}$-action that extends the $H$-action on $\cal Y$.

\subsubsection{}
\label{sec-classical-to-geom-action} Suppose $(\fr k,H)$ and $(H,H^{\flat})$ are as in \S\ref{sec-classical-to-geom-action-pair}, and let $Y\in\mathbf{Sch}^{\op{lft}}_{/S}$ be acted on by $H$.

We will now construct a functor:
\begin{equation}
\label{eq-classical-to-geom-action}
\mathbf{LieAlgd}_{/S}^{(\fr k,H)}(Y) \rightarrow \mathbf{FrmMod}_{/S}^{(H,H^{\flat})}(Y)
\end{equation}
which enhances the association of formal moduli problems to Lie algebroids, in the sense that the following diagram commutes:
$$
\xymatrix@C=1.5em@R=1.5em{
	\mathbf{LieAlgd}_{/S}^{(\fr k,H)}(Y) \ar[r]^-{\eqref{eq-classical-to-geom-action}}\ar[d]^-{\op{oblv}} & \mathbf{FrmMod}_{/S}^{(H,H^{\flat})}(Y)\ar[d]^-{\op{oblv}} \\
	\mathbf{LieAlgd}_{/S}(Y) \ar[r]^-{\eqref{eq-liealgd-to-frmmod}} & \mathbf{FrmMod}_{/S}(Y)
}
$$
To proceed, let us be given $(\cal L,\eta)\in\mathbf{LieAlgd}_{/S}^{(\fr k,H)}(Y)$. We need to construct an $H^{\flat}$-action on the formal moduli problem $\cal Y^{\flat}$ corresponding to $\cal L$, expressed by some groupoid $\xymatrix{\cal Y^{\flat}\underset{S}{\times} H^{\flat} \ar@<0.5ex>[r]^-{\op{act}^{\flat}}\ar@<-0.5ex>[r]_-{\op{pr}_1} & \cal Y^{\flat}}$, together with a map of simplicial prestacks:
\begin{equation}
\label{eq-enhance-action-map}
\xymatrix{
	\cdots \ar@<1.5ex>[r]\ar@<0.5ex>[r]\ar@<-0.5ex>[r]\ar@<-1.5ex>[r] &  Y\underset{S}{\times} H\underset{S}{\times} H \ar@<2.5ex>[r]^-{\op{act}\times 1}\ar@<1.5ex>[r]_-{1\times m}\ar@<-1ex>[r]_-{\op{pr}_{12}} \ar[d] &  Y\underset{S}{\times} H \ar@<0.5ex>[r]^-{\op{act}}\ar@<-0.5ex>[r]_-{\op{pr}_1}\ar[d] &  Y \ar[d]. \\
	\cdots \ar@<1.5ex>[r]\ar@<0.5ex>[r]\ar@<-0.5ex>[r]\ar@<-1.5ex>[r] & \cal Y^{\flat}\underset{S}{\times} H^{\flat}\underset{S}{\times} H^{\flat} \ar@<2.5ex>[r]^-{\op{act}^{\flat}\times 1}\ar@<1.5ex>[r]_-{1\times m}\ar@<-1ex>[r]_-{\op{pr}_{12}} & \cal Y^{\flat}\underset{S}{\times} H^{\flat} \ar@<0.5ex>[r]^-{\op{act}^{\flat}}\ar@<-0.5ex>[r]_-{\op{pr}_1} & \cal Y^{\flat}.
}
\end{equation}
Since each formal moduli problem $\cal Y^{\flat}\underset{S}{\times} (H^{\flat})^{\bullet}$ arises from the Lie algebroid $\op{pr}_Y^*\cal L \oplus \op{pr}_H^*(\fr k\otimes\cal O_H)^{\oplus\bullet}$ over $Y\underset{S}{\times} H^{\bullet}$, we only need to
\begin{itemize}
	\item produce a morphism
	\begin{equation}
	\label{eq-liealgd-action-map}
	\alpha : \op{pr}_Y^*\cal L \oplus \op{pr}_H^*(\fr k\otimes\cal O_H) \rightarrow \op{act}^!_{\mathbf{LieAlgd}}\cal L
	\end{equation}
	between Lie algebroids over $Y\underset{S}{\times}H$ (which would rise to $\op{act}^{\flat}$, in a way compatible with the morphism $\op{act}$)
	\item check that the following diagram:
	\begin{equation}
\label{eq-liealgd-action-cocycle}
\xymatrix@R=1.5em@C=5em{
	\op{pr}_Y^*\cal L \oplus \op{pr}_H^*(\fr k\otimes\cal O_H)^{\oplus 2} \ar[r]^-{\op{can}}\ar[d]^{\op{act}^!_{\mathbf{LieAlgd}}(\alpha)\times 1} & (1\times m)^!_{\mathbf{LieAlgd}}(\op{pr}_Y^*\cal L \oplus \op{pr}_H^*(\fr k\otimes\cal O_H)) \ar[d]^{(1\times m)^*_{\mathbf{LieAlgd}}(\alpha)} \\
	\op{act}^!_{\mathbf{LieAlgd}}(\cal L)\oplus \op{pr}_H^*(\fr k\otimes\cal O_H) \ar[d]^{\rotatebox{90}{$\sim$}} & (1\times m)_{\mathbf{LieAlgd}}^!\op{act}^!_{\mathbf{LieAlgd}}(\cal L)\ar[d]^{\rotatebox{90}{$\sim$}} \\
	(\op{act}\times 1)^!_{\mathbf{LieAlgd}}(\op{pr}_Y^*\cal L\oplus \op{pr}_H^*(\fr k\otimes\cal O_H)) \ar[r]^-{(\op{act}\times 1)^!_{\mathbf{LieAlgd}}(\alpha)} & (\op{act}\times 1)^!_{\mathbf{LieAlgd}}\op{act}^!_{\mathbf{LieAlgd}}(\cal L)
}
\end{equation}
	of Lie algebroids over $Y\underset{S}{\times}H\underset{S}{\times}H$ is commutative (which would affirm the commutativity of \eqref{eq-enhance-action-map} up to $2$-simplices, but the higher commutativity constraints are satisfied automatically since the corresponding $\infty$-categories are classical.)
\end{itemize}

\subsubsection{} Note that as an $\cal O_{Y\underset{S}{\times}H}$-module, we have an isomorphism $\op{act}^!_{\mathbf{LieAlgd}}(\cal L)\xrightarrow{\sim}\op{act}^*\cal L\underset{\op{act}^*\cal T_{Y/S}}{\times}\cal T_{Y\underset{S}{\times} H/S}$ by definition. The required map $\alpha$ is the sum of the following components:
\begin{itemize}
	\item the map $\op{pr}_Y^*\cal L \rightarrow \op{act}^!_{\mathbf{LieAlgd}}(\cal L)$ induced from the $H$-equivariance structure on $\cal L$ and the composition
	$$
	\op{pr}_Y^*\cal L \xrightarrow{\op{pr}_Y^*\sigma} \op{pr}_Y^*\cal T_{Y/S}\hookrightarrow \cal T_{Y\underset{S}{\times} H/S},
	$$
	where $\sigma$ is the anchor map of $\cal L$;
	\item the map $\fr k\otimes\cal O_H\rightarrow\op{act}^!_{\mathbf{LieAlgd}}(\cal L)$ induced from
	$$
	\fr k \xrightarrow{\eta} \op H^0(Y, \cal L) \xrightarrow{\op{act}^*} \op H^0(Y\underset{S}{\times} H, \op{act}^*\cal L),
	$$
	and the composition
	\begin{equation}
	\label{eq-anchor-for-k}
	\fr k\otimes\cal O_H \rightarrow \fr h\otimes\cal O_H \hookrightarrow \cal T_{Y\underset{S}{\times} H/S}.
	\end{equation}
\end{itemize}

The following Lemma shows that the functor \eqref{eq-classical-to-geom-action} is well-defined.

\begin{lem}
The map $\alpha$ is a morphism of Lie algebroids, and the diagram \eqref{eq-liealgd-action-cocycle} commutes.
\end{lem}
\begin{proof}
It is obvious that $\alpha$ is compatible with the anchor maps. To show that $\alpha$ preserves the Lie bracket, we check it for sections of $\op{pr}_Y^*\cal L \oplus \op{pr}_H^*(\fr k\otimes\cal O_H)$ of the following types:
\begin{itemize}
	\item $l_1, l_2\in\op{pr}_Y^{-1}\cal L$; this follows from the assumptions that the equivariance structure $\theta : \op{pr}_Y^*\cal L\rightarrow\op{act}^*\cal L$ is compatible with the Lie bracket, and $\sigma$ is a map of $H$-equivariant sheaves;
	\item $\xi_1,\xi_2\in\fr k$; this is clear;
	\item $l\in\op{pr}_Y^{-1}\cal L$ and $\xi\in\fr k$; this is a slightly more involved calculation, which we now perform.
\end{itemize}

Write $\theta(l) = \sum_i f_i\otimes l_i$, where $f_i\in\cal O_{Y\times H}$ and $l_i\in\op{act}^{-1}\cal L$. We need to show the vanishing of the following element in $\op{act}^*\cal L\underset{\op{act}^*\cal T_{Y/S}}{\times}\cal T_{Y\underset{S}{\times} H/S}$:
	\begin{equation}
	\label{eq-vanishing-mixed-terms}
		[\alpha(l), \alpha(\xi)] = [\sum_i(f_i\otimes l_i)\times\sigma(l), (1\otimes\eta(\xi))\times \sigma'(\xi)]
	\end{equation}
	where $\sigma'$ denotes the composition \eqref{eq-anchor-for-k}. Note that the $\cal T_{Y\underset{S}{\times} H/S}$-component of \eqref{eq-vanishing-mixed-terms} vanishes tautologically, so we just need to show the vanishing of its $\op{act}^*\cal L$-component. The latter is given (using \eqref{eq-algebroid-identity}) by
	\begin{equation}
	\label{eq-vanishing-mixed-terms-2}
	\sum_i f_i\otimes [l_i, \eta(\xi)] - \sum_i \sigma'(\xi)(f_i)\otimes l_i = -\sum_i (f_i\otimes (\xi_{\fr h}\cdot l_i) + (\xi_{\fr h}\cdot f_i)\otimes l_i)
	\end{equation}
	where in the second summand, $\xi_{\fr h}$ acts on $f_i\in\cal O_{{Y\underset{S}{\times} H/S}}$ by derivation on the \emph{$\cal O_H$-component}.
	
	Consider the right $H$-action on $Y\underset{S}{\times} H$, given by $(y,h),h'\leadsto (y,hh')$; if we equip $\op{act}^*\cal L$ with the following $H$-equivariance structure:
	$$
	\op{act}^*\cal L\big|_{(y,h)} \xrightarrow{\sim} \cal L\big|_{yh} \xrightarrow{\theta_{(yh, h')}} \cal L\big|_{yhh'} \xrightarrow{\sim} \op{act}^*\cal L\big|_{(y,hh')},
	$$
	then \eqref{eq-vanishing-mixed-terms-2} is the (negative of the) induced action of $\xi_{\fr h}$ on the section $\sum_i f_i\otimes l_i = \theta(l)$ in $\op{act}^*\cal L$. Note that $\op{pr}_Y^*\cal L$ can also be endowed with an $H$-equivariance structure:
	$$
	\op{pr}^*\cal L\big|_{(y,h)} \xrightarrow{\sim} \cal L\big|_y \xrightarrow{\sim} \op{pr}^*\cal L\big|_{(y,hh')}
	$$
	such that $\theta$ is a map of $H$-equivariant $\cal O_{Y\underset{S}{\times} H}$-modules. Hence the element $\xi_{\fr h}\cdot\theta(l)$ identifies with $\theta(\xi_{\fr h}\cdot l)$. On the other hand, $l\in\op{pr}^{-1}\cal L$ so $\xi_{\fr h}\cdot l=0$, from which we deduce the required vanishing of \eqref{eq-vanishing-mixed-terms-2}.
	
	Checking the commutativity of \eqref{eq-liealgd-action-cocycle} is not difficult, and we leave it to the reader.
\end{proof}

\subsubsection{} We now characterize the image of the functor \eqref{eq-classical-to-geom-action}.

\begin{prop}
\label{prop-classical-to-geom-action-image}
The functor \eqref{eq-classical-to-geom-action} is an equivalence onto the full subcategory:
$$
\mathbf{FrmMod}_{/S}^{(H,H^{\flat})}(Y)^{\op{cl}} \hookrightarrow \mathbf{FrmMod}_{/S}^{(H,H^{\flat})}(Y)
$$
that consists of objects $\cal Y^{\flat}$ such that $\mathbb T_{Y/\cal Y^{\flat}}$ lies in $\Upsilon_Y(\op{QCoh}(Y)^{\heartsuit})$.
\end{prop}
\begin{proof}
Indeed, such a formal moduli problems $\cal Y^{\flat}$ arises from some Lie algebroid $\cal L$ via the functor \eqref{eq-liealgd-to-frmmod}. Given the additional data of an $(H,H^{\flat})$-action, we consider the following commutative diagrams:
\begin{equation}
\label{eq-recover-eta-diag}
\xymatrix@C=1.5em@R=1.5em{
	Y\underset{S}{\times} H \ar[rr]^-{\op{act}}\ar[d] & & Y \ar[d] \\
	\cal Y^{\flat}\underset{S}{\times} H \ar[r]^-{i} & \cal Y^{\flat}\underset{S}{\times} H^{\flat} \ar[r]^-{\op{act}^{\flat}} & \cal Y^{\flat}
}
\quad\quad
\xymatrix@C=1.5em@R=1.5em{
	Y\underset{S}{\times} H \ar[rr]^-{\op{act}}\ar[d] & & Y \ar[d] \\
	Y\underset{S}{\times} H^{\flat} \ar[r]^-{j} & \cal Y^{\flat}\underset{S}{\times} H^{\flat} \ar[r]^-{\op{act}^{\flat}} & \cal Y^{\flat}
}
\end{equation}
From these diagrams, we obtain two maps between tangent complexes:
$$
\mathbb T_{Y\underset{S}{\times} H/\cal Y^{\flat}\underset{S}{\times} H} \xrightarrow{\op{act}^{\flat}_*\circ i_*} \mathbb T_{Y\underset{S}{\times} H/\cal Y^{\flat}} \rightarrow\mathbb T_{Y/\cal Y^{\flat}}\big|_{Y\underset{S}{\times} H},
$$
which gives rise to a morphism $\theta : \op{pr}_Y^*\cal L \rightarrow \op{act}^*\cal L$; and
\begin{equation}
\label{eq-recover-eta}
\mathbb T_{Y\underset{S}{\times} H/Y\underset{S}{\times} H^{\flat}} \xrightarrow{\op{act}_*^{\flat}\circ j_*}\mathbb T_{Y\underset{S}{\times} H/\cal Y^{\flat}}\rightarrow\mathbb T_{Y/\cal Y^{\flat}}\big|_{Y\underset{S}{\times} H},
\end{equation}
which gives rise to a map $\widetilde{\eta} : \op{pr}_H^*(\fr k\otimes\cal O_H) \rightarrow \op{act}^*\cal L$; restricting to $Y\underset{S}{\times}\{1\}$, we obtain a map $\eta: \fr k\otimes\cal O_Y \rightarrow \cal L$.

The functor $\mathbf{FrmMod}_{/S}^{(H,H^{\flat})}(Y)^{\op{cl}} \rightarrow \mathbf{LieAlgd}^{(\fr k,H)}_{/S}(Y)$ inverse to \eqref{eq-classical-to-geom-action} is defined by sending $\cal Y^{\flat}$ to the Lie algebroid $\cal L$, equipped with the $(\fr k,H)$-structure specified by the above maps $\theta$ and $\eta$.
\end{proof}

\subsubsection{} We give an alternative description of the map $\alpha$ that will be used in the proof of Proposition \ref{prop-quotient-compare}. Consider the commutative diagram:
\begin{equation}
\label{eq-recover-eta-diag-alt}
\xymatrix@C=1.5em@R=1.5em{
	Y \ar[rr]^-{\op{can}}\ar[d] & & Y/H \ar[d] \\
	Y\underset{S}{\times} (H^{\flat}/H) \ar[r]^-{\widetilde j} & \cal Y^{\flat}\underset{S}{\times} (H^{\flat}/H) \ar[r]^-{\widetilde{\op{act}}^{\flat}} & \cal Y^{\flat}/H
}
\end{equation}
which is the ``quotient'' by $H$ of the right diagram in \eqref{eq-recover-eta-diag}. It produces the following map between tangent complexes:
\begin{equation}
\label{eq-recover-eta-alt}
\mathbb T_{Y/(Y\underset{S}{\times}(H^{\flat}/H))} \xrightarrow{\widetilde{\op{act}}_*^{\flat} \circ\widetilde{j}_*} \mathbb T_{Y/(\cal Y^{\flat}/H)} \rightarrow \mathbb T_{(Y/H)/(\cal Y^{\flat}/H)}\big|_Y \xrightarrow{\sim} \mathbb T_{Y/\cal Y^{\flat}}.
\end{equation}
We \emph{claim} that \eqref{eq-recover-eta-alt} identifies with the restriction of \eqref{eq-recover-eta} to $Y\underset{S}{\times}\{1\}$. Indeed, this follows from the fact that \eqref{eq-recover-eta} is the pullback of \eqref{eq-recover-eta-alt} along $\op{pr}_Y : Y\underset{S}{\times} H\rightarrow Y$, and the composition $Y\underset{S}{\times}\{1\} \hookrightarrow Y\underset{S}{\times} H \xrightarrow{\op{pr}_Y} Y$ is the identity.

\subsection{Quotient II} Suppose $(H,H^{\flat})$ is a geometric action pair (see \S\ref{sec-geom-action-pair} for the definition). Let $(\cal Y, \cal Y^{\flat})\in\mathbf{FrmMod}_{/S}^{(H,H^{\flat})}$.

\subsubsection{}
\label{sec-quotient-2-exists} The \emph{quotient} of $(\cal Y, \cal Y^{\flat})$ by $(H,H^{\flat})$ is defined as the quotient in the $\infty$-category $\mathbf{FrmMod}_{/S}$. In other words, it is the geometric realization of the simplicial object $(\cal Y, \cal Y^{\flat})\times(H,H^{\flat})^{\bullet}$ in $\mathbf{FrmMod}_{/S}^{(H,H^{\flat})}$ characterizing the $(H,H^{\flat})$-action on $(\cal Y, \cal Y^{\flat})$.

\begin{prop}
\label{prop-quotient-2}
The quotient of $(\cal Y,\cal Y^{\flat})$ by $(H, H^{\flat})$ exists.
\end{prop}
\begin{proof}
We construct the quotient in the $\infty$-category $\op{Fun}(\Delta^1, \mathbf{PStk}_{\op{laft-def}/S})$, and then check that the result belongs to the full subcategory $\mathbf{FrmMod}_{/S}$. Quotient in the above functor category is computed pointwise as follows:
\begin{itemize}
	\item at the vertex $[0]$, we have the prestack quotient $\cal Y/H$; it is an object of $\mathbf{PStk}_{\op{laft-def}/S}$ because $H$ is a group \emph{scheme} locally of finite type;
	\item at the vertex $[1]$, we assert that the quotient of $\cal Y^{\flat}$ by $H^{\flat}$ exists in $\mathbf{PStk}_{\op{laft-def}/S}$; indeed, it is given by $\op B_{\cal Y^{\flat}/H}(\cal Y^{\flat}\underset{S}{\overset{H}{\times}} H^{\flat}/H)$ where $\cal Y^{\flat}\underset{S}{\overset{H}{\times}} H^{\flat}/H$ denotes the Hecke groupoid\footnote{Suppose $\cal C$ is an $\infty$-category with finite products. Let $H\rightarrow K$ be a map of group objects in $\cal C$. Suppose any object in $\cal C$ with an $H$-action admits a quotient. Then given an object $Y\in\cal C$ with a $K$-action, there exists a \emph{Hecke groupoid} $Y\overset{H}{\times}K/H$ acting on $Y/H$ whose quotient, if exists, agrees with $Y/K$.} acting on the prestack quotient $\cal Y^{\flat}/H$:
	$$
	\xymatrix{
	\cdots \ar@<1.5ex>[r]\ar@<0.5ex>[r]\ar@<-0.5ex>[r]\ar@<-1.5ex>[r] & \cal Y^{\flat}\underset{S}{\overset{H}{\times}} H^{\flat}\underset{S}{\overset{H}{\times}} H^{\flat}/H \ar@<2.5ex>[r]^-{\op{act}^{\flat}\times 1}\ar@<1.5ex>[r]_-{1\times m}\ar@<-1ex>[r]_-{\op{pr}_{12}} & \cal Y^{\flat}\underset{S}{\overset{H}{\times}} H^{\flat}/H \ar@<0.5ex>[r]^-{\op{act}^{\flat}}\ar@<-0.5ex>[r]_-{\op{pr}_1} & \cal Y^{\flat}/H,
	}
	$$
	and $\op B_{\cal Y^{\flat}/H}$ is the functor from \S\ref{sec-delooping-functor}.
\end{itemize}
Finally, the morphism $\cal Y/H \rightarrow \op B_{\cal Y^{\flat}/H}(\cal Y^{\flat}\underset{S}{\overset{H}{\times}} H^{\flat}/H)$ is a nil-isomorphism since it is the composition of nil-isomorphisms $\cal Y/H \rightarrow \cal Y^{\flat}/H \rightarrow\op B_{\cal Y^{\flat}/H}(\cal Y^{\flat}\underset{S}{\overset{H}{\times}} H^{\flat}/H)$.
\end{proof}

Regarding $\cal Y$ as a fixed prestack acted on by $H$, we denote the resulting quotient functor by
\begin{equation}
\label{eq-quotient-geom}
\mathbf Q^{(H,H^{\flat})} : \mathbf{FrmMod}_{/S}^{(H,H^{\flat})}(\cal Y) \rightarrow \mathbf{FrmMod}_{/S}(\cal Y/H),\quad \cal Y^{\flat}\leadsto \op B_{\cal Y^{\flat}/H}(\cal Y^{\flat}\underset{S}{\overset{H}{\times}} H^{\flat}/H).
\end{equation}

\subsubsection{} Tautologically, the quotient $(\cal Y/H, \op B_{\cal Y^{\flat}/H}(\cal Y^{\flat}\underset{S}{\overset{H}{\times}} H^{\flat}/H))$, equipped with the map from $(\cal Y, \cal Y^{\flat})$, satisfies the universal property:
$$
\op{Maps}_{\mathbf{FrmMod}_{/S}}((\cal Y/H, \op B_{\cal Y^{\flat}/H}(\cal Y^{\flat}\underset{S}{\overset{H}{\times}} H^{\flat}/H)), (\cal Z,\cal Z^{\flat}))\xrightarrow{\sim}
\op{Maps}_{\mathbf{FrmMod}_{/S}^{(H,H^{\flat})}}((\cal Y, \cal Y^{\flat}), (\cal Z,\cal Z^{\flat}))
$$
where in the second expression, $(\cal Z, \cal Z^{\flat})$ is equipped with the trivial $(H,H^{\flat})$-action.

Specializing to $\cal Z=\cal Y/H$, we see that the object $\mathbf Q^{(H,H^{\flat})}(\cal Y^{\flat}) \in \mathbf{FrmMod}_{/S}(\cal Y/H)$ is characterized by the universal property:
\begin{equation}
\label{eq-quotient-geom-univ}
\op{Maps}_{\mathbf{FrmMod}_{/S}(\cal Y/H)}(\mathbf Q^{(H,H^{\flat})}(\cal Y^{\flat}), \cal Z^{\flat}) \xrightarrow{\sim} \op{Maps}_{\mathbf{FrmMod}_{/S}^{(H,H^{\flat})}(\cal Y)}(\cal Y^{\flat}, \pi^!_{\mathbf{FrmMod}}(\cal Z^{\flat}))
\end{equation}
where in the second expression, $\pi^!_{\mathbf{FrmMod}}\cal Z^{\flat}\cong \cal Z^{\flat}\underset{(\cal Y/H)_{\dR}}{\times} \cal Y_{\dR}$ is acted on by $H^{\flat}$ through the canonical homomorphism $H^{\flat}\rightarrow H_{\dR}$ on the $\cal Y_{\dR}$ factor.

\begin{rem}
Recall from \S\ref{sec-quotient-classical-univ} the $(\fr k, H)$-Lie algebroid structure on $\pi^!_{\mathbf{LieAlgd}}(\cal M)$, where $(\fr k,H)$ is any classical action pair and $\cal M$ is a Lie algebroid on the quotient $Y/H$. If $H^{\flat}=H/\exp(\fr k)$ as in \S\ref{sec-classical-to-geom-action-pair}, then the $(H,H^{\flat})$-formal moduli problem $\pi^!_{\mathbf{FrmMod}}(\cal Z^{\flat})$ is precisely the one associated to $\pi^!_{\mathbf{LieAlgd}}(\cal M)$ under the functor \eqref{eq-classical-to-geom-action}.
\end{rem}

\subsubsection{}
\label{sec-normal-subpair-geom} Let $(H^0, (H^0)^{\flat})\rightarrow(H, H^{\flat})$ be a morphism of geometric action pairs. We say that $(H^0, (H^0)^{\flat})$ is a \emph{normal subpair} of $(H, H^{\flat})$ if there is a morphism $(H, H^{\flat}) \rightarrow (H_0, (H_0)^{\flat})$ of geometric action pairs whose kernel identifies with $(H^0, (H^0)^{\flat})$. In particular, the $(H,H^{\flat})$-action on itself extends to $(H^0, (H^0)^{\flat})$.

Given a normal subpair $(H^0, (H^0)^{\flat})$ of $(H, H^{\flat})$, we recover $(H_0, (H_0)^{\flat})$ by the isomorphisms:
$$
H_0 \xrightarrow{\sim} H/H^0, \quad H_0^{\flat} \xrightarrow{\sim} \mathbf Q^{(H^0, (H^0)^{\flat})}(H^{\flat}).
$$
Let $\cal Y^{\flat}\in\mathbf{FrmMod}_{/S}^{(H, H^{\flat})}(\cal Y)$. Then the prestack $\mathbf Q^{(H^0, (H^0)^{\flat})}(\cal Y^{\flat})$ is naturally an object of $\mathbf{FrmMod}_{/S}^{(H_0, H_0^{\flat})}(\cal Y/H^0)$, and we have a second isomorphism theorem:

\begin{prop}
\label{prop-normal-subpair-geom}
There is a natural isomorphism:
$$
\mathbf Q^{(H_0, H_0^{\flat})}\circ\mathbf Q^{(H^0, (H^0)^{\flat})}(\cal Y^{\flat}) \xrightarrow{\sim} \mathbf Q^{(H, H^{\flat})}(\cal Y^{\flat}).
$$
\end{prop}
\begin{proof}
Both sides are the quotient of $(\cal Y, \cal Y^{\flat})$ by $(H,H^{\flat})$ in the $\infty$-category $\mathbf{FrmMod}_{/S}$.
\end{proof}

\subsubsection{}
Suppose we have a quasi-twisting $\widehat{\cal Y}^{\flat} \in \mathbf{QTw}_{/S}(\cal Y/\cal Y^{\flat})$, such that $(\cal Y, \widehat{\cal Y}^{\flat})$ is also an $(H,H^{\flat})$-formal moduli problem, and the morphism $\widehat{\cal Y}^{\flat}\rightarrow\cal Y^{\flat}$ preserves this structure. We call quasi-twistings with these additional data \emph{$(H,H^{\flat})$-quasi-twistings} (based at $\cal Y^{\flat}$) and denote the category of them by $\mathbf{QTw}_{/S}^{(H,H^{\flat})}(\cal Y/\cal Y^{\flat})$. The quotient $\mathbf Q^{(H,H^{\flat})}(\widehat{\cal Y}^{\flat})$ inherits the structure of a quasi-twisting on $\cal Y/H$ based at $\mathbf Q^{(H,H^{\flat})}(\cal Y^{\flat})$. Indeed,
\begin{itemize}
	\item applying $\mathbf Q^{(H,H^{\flat})}$ to the action groupoid $\widehat{\cal Y}^{\flat}\times\op B\widehat{\mathbb G}_m^{\bullet}$, we obtain a $\op B\widehat{\mathbb G}_m$-action on $\mathbf Q^{(H,H^{\flat})}(\widehat{\cal Y}^{\flat})$, which gives rise to a $\widehat{\mathbb G}_m$-gerbe structure;
	\item the section $\cal Y/H \rightarrow \mathbf Q^{(H,H^{\flat})}(\widehat{\cal Y}^{\flat})$ is given by the composition:
	$$
	\cal Y/H \rightarrow \widehat{\cal Y}^{\flat}/H \rightarrow \mathbf Q^{(H,H^{\flat})}(\widehat{\cal Y}^{\flat}).
	$$
\end{itemize}
Therefore, we may view $\mathbf Q^{(H,H^{\flat})}$ as a functor $\mathbf{QTw}_{/\cal Y^{\flat}}^{(H,H^{\flat})}(Y/S)\rightarrow \mathbf{QTw}_{/\mathbf Q^{(H,H^{\flat})}(\cal Y^{\flat})}((Y/H)/S)$.

\subsection{Comparison of $\mathbf Q_{\op{inj}}^{(\fr k,H)}$ and $\mathbf Q^{(H,H^{\flat})}$} Suppose $(\fr k,H)$ and $(H,H^{\flat})$ are as in \S\ref{sec-classical-to-geom-action-pair}, and let $Y\in\mathbf{Sch}^{\op{lft}}_{/S}$ be acted on by $H$.

\subsubsection{} We now show that the two quotient functors constructed above are compatible.

\begin{prop}
\label{prop-quotient-compare}
The following diagram is commutative:
$$
\xymatrix{
\mathbf{LieAlgd}_{\op{inj}/S}^{(\fr k,H)}(Y) \ar@{^{(}->}[r]^-{\eqref{eq-classical-to-geom-action}} \ar[d]^{\mathbf Q_{\op{inj}}^{(\fr k, H)}} & \mathbf{FrmMod}_{/S}^{(H,H^{\flat})}(Y) \ar[d]^{\mathbf Q^{(H,H^{\flat})}} \\
\mathbf{LieAlgd}_{/S}(Y/H) \ar@{^{(}->}[r]^-{\eqref{eq-liealgd-to-frmmod}} & \mathbf{FrmMod}_{/S}(Y/H).
}
$$
\end{prop}
\begin{proof}
Suppose $(\cal L, \eta)\in\mathbf{LieAlgd}_{\op{inj}/S}^{(\fr k, H)}(Y)$, i.e., $\cal L$ is a $(\fr k, H)$-Lie algebroid over $Y$ such that the map $\eta: \fr k\otimes\cal O_Y\rightarrow\cal L$ is injective. Let $\cal Y^{\flat}$ be the corresponding formal moduli problem under $Y$, equipped with the $H^{\flat}$-action defined by the functor \eqref{eq-classical-to-geom-action}. Thus $\mathbf Q^{(H,H^{\flat})}(\cal Y^{\flat})$ satisfies the universal property \eqref{eq-quotient-geom-univ} for $\cal Z^{\flat}\in\mathbf{FrmMod}_{/S}(\cal Y/H)$.

On the other hand, $\mathbf Q_{\op{inj}}^{(\fr k, H)}(\cal L)$ satisfies the universal property \eqref{eq-quotient-classical-univ}. Since the essential image of \eqref{eq-liealgd-to-frmmod} consists of objects $\cal Z^{\flat}\in\mathbf{FrmMod}_{/S}(Y/H)$ such that $\mathbb T_{(Y/H)/\cal Z^{\flat}}\in\Upsilon_{Y/H}(\op{QCoh}(Y/H)^{\heartsuit})$, it suffices to show that $\mathbf Q^{(H,H^{\flat})}(\cal Y^{\flat})$ has this property. The result thus follows from the lemma below and the fact that $Y\rightarrow Y/H$ is faithfully flat.
\end{proof}

\begin{lem}
Suppose $(Y,\cal Y^{\flat})$ is the $(H,H^{\flat})$-formal moduli problem corresponding to the $(\fr k,H)$-Lie algebroid $(\cal L,\eta)$ under the functor \eqref{eq-classical-to-geom-action}. Then there is a canonical isomorphism
$$
\cal T_{(Y/H)/\mathbf Q^{(H,H^{\flat})}(\cal Y^{\flat})}\big|_Y \xrightarrow{\sim} \op{Cofib}(\eta).
$$
\end{lem}
\begin{proof}
We will use the expression of $\mathbf Q^{(H,H^{\flat})}(\cal Y^{\flat})$ as quotient of the Hecke groupoid $\cal Y^{\flat}\underset{S}{\overset{H}{\times}}H^{\flat}/H$ (see \eqref{eq-quotient-geom}). Consider the following commutative diagram extending the diagram \eqref{eq-recover-eta-diag-alt}:
$$
\xymatrix@C=1.5em@R=1.5em{
	Y \ar@[red]@{.>}[rr] \ar[d]_{\op{id}\times\{1\}} \ar@{.>}[ddr] & & Y/H \ar@{.>}@[red][d] \\
	Y\underset{S}{\times} H^{\flat}/H \ar@[blue][r]^{\widetilde j}\ar@[blue][d]_{\op{pr}} & \cal Y^{\flat}\underset{S}{\times} H^{\flat}/H \ar@[blue][r]^-{\widetilde{\op{act}}^{\flat}}\ar[d]_{\op{pr}} & \cal Y^{\flat}/H \ar@[purple][d] \\
	Y \ar@[blue][r] & \cal Y^{\flat} \ar@{.>}[ur]\ar@[blue][r] & \mathbf Q^{(H,H^{\flat})}(\cal Y^{\flat})
}
$$
where the two lower squares, as well as the dotted quadrilateral, are Cartesian. Thus, we obtain the following commutative diagram of objects in $\op{QCoh}(Y)$, where commutativity of the red (resp.~blue) squares is derived from the red (resp.~blue) arrows in the above diagram\footnote{Recall: purple = red + blue.}:
$$
\xymatrix@C=1.5em@R=1.5em{
	\cal T_{(Y\underset{S}{\times} H^{\flat}/H)/Y}\big|_Y[-1] \ar@[blue][r]^-{\sim}\ar@[blue][d]^-{\sim} & \cal T_{(\cal Y^{\flat}/H)/\mathbf Q^{(H,H^{\flat})}(\cal Y^{\flat})}\big|_Y[-1] \ar@[red][r]\ar@[purple][d] & \cal T_{(Y/H)/(\cal Y^{\flat}/H)}\big|_Y \ar@[red][r]\ar@[red][d]^-{\sim} & \cal T_{(Y/H)/\mathbf Q^{(H,H^{\flat})}(\cal Y^{\flat})}\big|_Y \ar@[red][d] \\
	\cal T_{Y/(Y\underset{S}{\times} H^{\flat}/H)} \ar@/_1pc/[drr]_-{\eqref{eq-recover-eta-alt}} \ar@[blue][r]^-{\widetilde{\op{act}}^{\flat}_*\circ\widetilde j_*} & \cal T_{Y/(\cal Y^{\flat}/H)} \ar@[red][r] & \cal T_{(Y/H)/(\cal Y^{\flat}/H)}\big|_Y\ar@[red][r]\ar[d]^{\sim} & \cal T_{Y/(Y/H)}[1] \\
	& & \cal T_{Y/\cal Y^{\flat}} &
}
$$
Furthermore, the two horizontal red triangles are exact. Note that the composition \eqref{eq-recover-eta-alt} identifies with $\eta$, so the upper horizontal triangle identifies $\cal T_{(Y/H)/\mathbf Q^{(H,H^{\flat})}(\cal Y^{\flat})}\big|_Y$ with $\op{Cofib}(\eta)$.
\end{proof}

\subsection{Example: inert quasi-twistings}
\label{sec-inert-qtw}
We now specialize to Lie algebroids arising from abelian Lie algebras. They give rise to what we call ``inert quasi-twistings.'' In the geometric Langlands theory, they arise naturally as degeneration of (non-inert) quasi-twistings as the quantum parameter $\kappa$ tends to $\infty$ (see \S\ref{sec-infty}).

\subsubsection{} Recall that over any $\cal Y\in\mathbf{PStk}_{\op{laft-def}/S}$, there is a functor
$$
\op{triv} : \op{IndCoh}(\cal Y)\rightarrow\mathbf{Lie}(\op{IndCoh}(\cal Y))
$$
that associates to an ind-coherent sheaf $\cal F$ the abelian Lie algebra on $\cal F$. More precisely, $\op{triv}$ is the right inverse to the forgetful functor; because the latter is conservative and preserves limits, $\op{triv}$ also preserves limits.

We also have a pair of adjunction:
$$
\xymatrix{
	\op{diag}_{\cal Y} : \mathbf{Lie}(\op{IndCoh}(\cal Y)) \ar@<0.5ex>[r] & \mathbf{FrmMod}(\cal Y) : \op{ker-anch} \ar@<0.5ex>[l]
}
$$
where $\op{diag}_{\cal Y}$ preserves fiber products.\footnote{One sees this by identifying $\mathbf{Lie}(\op{IndCoh}(\cal Y))$ with $\mathbf{FrmMod}(\cal Y)_{/\cal Y}$, where $\cal Y$ is regarded as a formal moduli problem under itself by the identity map. Under this identification, $\op{diag}_{\cal Y}$ becomes the tautological forgetful functor; see \cite[IV.4]{GR16}.} It follows that the composition $\op{diag}_{\cal Y}\circ\op{triv}$ preserves fiber products. We call $\cal Y^{\flat}:=\op{diag}_{\cal Y}\circ\op{triv}(\cal F)$ the \emph{inert} formal moduli problem on $\cal F$.

\begin{rem}
Let $Y$ be a scheme (not necessarily locally of finite type) over $S$. The classical analogue of the above construction associates to an $\cal O_Y$-module $\cal F$ the Lie algebroid on $\cal F$ with \emph{zero} Lie bracket and anchor map. If $Y\in\mathbf{Sch}^{\op{lft}}_{/S}$, then the image of $\cal F$ under \eqref{eq-liealgd-to-frmmod} agrees with $\op{diag}_{\cal Y}\circ\op{triv}(\Upsilon_Y(\cal F))$.
\end{rem}

\subsubsection{} For the remainder of this section, we suppose $Y\in\mathbf{Sch}_{/S}^{\op{lft}}$ is \emph{smooth}. Then the identification $\Upsilon_Y : \op{QCoh}(Y)\xrightarrow{\sim}\op{IndCoh}(Y)$ allows us to view the universal enveloping algebra\footnote{This is defined as a monad on $\op{IndCoh}(Y)$ in \cite[IV.4.4]{GR16}.} of an object $\cal Y^{\flat}\in\mathbf{FrmMod}_{/S}(Y)$ as an algebra in $\op{QCoh}(Y)$. If $\cal Y^{\flat}=\op{diag}_{Y}\circ\op{triv}(\Upsilon_Y(\cal F))$, then it is given by $\op{Sym}_{\cal O_Y}(\cal F)$.

Let $\mathbb V(\cal F):=\underline{\op{Spec}}_Y\op{Sym}_{\cal O_Y}(\cal F)$; it is a stack over $Y$ fibered in linear DG schemes. We have an equivalence of DG categories:
\begin{equation}
\label{eq-module-cat-abelian}
\op{IndCoh}(\cal Y^{\flat}) \xrightarrow{\sim} \op{QCoh}(\mathbb V(\cal F)),
\end{equation}
where $\mathbf{oblv} : \op{IndCoh}(\cal Y^{\flat})\rightarrow\op{IndCoh}(Y)$ passes to the pushforward functor on $\op{QCoh}$ (see \cite[IV.4 \S4.1.3, IV.2 (7.12), and IV.3 Proposition 5.1.2]{GR16}).

\subsubsection{}
\label{sec-inert-qtw-on-triangle} Suppose, furthermore, that we have a quasi-twisting $\widehat{\cal Y}^{\flat}\in\mathbf{QTw}_{/S}(Y/\cal Y^{\flat})$ that arises from a triangle $\cal O_Y \rightarrow \widehat{\cal F} \rightarrow \cal F$ in $\op{QCoh}(Y)$ under the composition of functors $\op{diag}_{Y}\circ\op{triv}\circ\Upsilon_Y$. We call $\widehat{\cal Y}^{\flat}$ the \emph{inert quasi-twisting} on the triangle $\cal O_Y\rightarrow\widehat{\cal F} \rightarrow\cal F$.

The map $\cal O_Y\rightarrow\widehat{\cal F}$ gives rise to a morphism of DG schemes:
\begin{equation}
\label{eq-embedding-into-a1}
\underline{\op{Spec}}_Y\op{Sym}_{\cal O_Y}(\widehat{\cal F}) \rightarrow Y\times\mathbb A^1
\end{equation}
We let $\mathbb V(\widehat{\cal F})_{\lambda=1}$ be the fiber of \eqref{eq-embedding-into-a1} at $\{ 1\}\hookrightarrow\mathbb A^1$. Note that the analogously defined fiber $\mathbb V(\widehat{\cal F})_{\lambda=0}$ identifies with $\mathbb V(\cal F)$. There is a canonical equivalence of DG categories:
\begin{equation}
\label{eq-module-cat-abelian-qtw}
\widehat{\cal Y}^{\flat}\Mod \xrightarrow{\sim} \op{QCoh}(\mathbb V(\widehat{\cal F})_{\lambda = 1}).
\end{equation}

\begin{rem}
From our point of view, the DG category $\op{QCoh}(\op{LocSys}_G)$ is realized by modules over some quasi-twisting on $\op{Bun}_G$. The DG stack $\op{LocSys}_G$ only appears \emph{a posteriori} through \eqref{eq-module-cat-abelian-qtw}. Thus, one can say that the origin of $\op{QCoh}(\op{LocSys}_G)$ is non-geometric.
\end{rem}

\subsubsection{} We now discuss how quotient interacts with inert quasi-twistings. Denote by $\op{pt}$ the $S$-scheme $S$ itself. Suppose $(\fr k, H)$ is a classical action pair with \emph{zero} map $\fr k\rightarrow\fr h$. Then we have
$$
H^{\flat} := H/\exp(\fr k) \xrightarrow{\sim} H\ltimes (\op{pt}/\exp(\fr k)),
$$
where the formation of the semidirect product is formed by the $H$-action on $\op{pt}/\exp(\fr k)$. Note that the normal subpair $(\op{pt}, \op{pt}/\exp(\fr k))$ of $(H, H^{\flat})$ has quotient $(H, H)$, since
$$
\mathbf Q^{(\op{pt}, \op{pt}/\exp(\fr k))}(H^{\flat})\xrightarrow{\sim} \op B_{H^{\flat}}(H^{\flat}\times (\op{pt}/\exp(\fr k))^{\bullet}) \xrightarrow{\sim} H;
$$
see \S\ref{sec-normal-subpair-geom}.

\subsubsection{}
\label{sec-inert-qtw-geom} We now assume that $\fr k$ is also abelian. Suppose the smooth scheme $Y$ admits an $H$-action, and $\cal Y^{\flat}$ is the inert formal moduli problem on some $H$-equivariant sheaf $\cal F\in\op{QCoh}(Y)^{\heartsuit}$.

Suppose we have an $H$-equivariant map $\eta : \fr k\otimes\cal O_Y\rightarrow\cal F$, giving rise to an $H^{\flat}$-action on $\cal Y^{\flat}$ (see \S\ref{sec-classical-to-geom-action}). Let $\cal Q:=\op{Cofib}(\eta)$; it is an $H$-equivariant complex of $\cal O_Y$-modules, hence descends to an object $\cal Q^{\op{desc}} \in \op{QCoh}(Y/H)$.

\begin{prop}
\label{prop-quotient-abelian-frmmod}
The quotient $\mathbf Q^{(H,H^{\flat})}(\cal Y^{\flat})$ identifies with the inert formal moduli problem on $\cal Q^{\op{desc}}\in\op{QCoh}(Y/H)$.
\end{prop}
\begin{proof}
By Proposition \ref{prop-normal-subpair-geom}, we have
$$
\mathbf Q^{(H, H^{\flat})}(\cal Y^{\flat}) \xrightarrow{\sim} \mathbf Q^{(H,H)}\circ\mathbf Q^{(\op{pt}, \op{pt}/\exp(\fr k))}(\cal Y^{\flat})\xrightarrow{\sim} \mathbf Q^{(\op{pt}, \op{pt}/\exp(\fr k))}(\cal Y^{\flat})/H.
$$
Note that descent of $\cal O_Y$-modules corresponds to quotient by $H$ on the inert formal moduli problem. Hence we only need to identify $\mathbf Q^{(\op{pt}, \op{pt}/\exp(\fr k))}(\cal Y^{\flat})$ as the inert formal moduli problem on $\cal Q$.

Consider the \v Cech nerve of $\cal F\rightarrow\cal Q$ in $\op{QCoh}(Y)$, which identifies with the groupoid $\cal F\oplus(\fr k\otimes\cal O_Y)^{\oplus\bullet}$. Since the composition $\op{diag}_Y\circ\op{triv}$ preserves fiber products, we see that 
$$
\op{diag}_Y\circ\op{triv}(\cal F\oplus(\fr k\otimes\cal O_Y)^{\oplus\bullet}) \xrightarrow{\sim} \cal Y^{\flat}\times(\op{pt}/\exp(\fr k))^{\bullet}
$$
identifies with the \v Cech nerve of the map $\cal Y^{\flat} \rightarrow \op{diag}_Y\circ\op{triv}(\cal Q)$. The result follows since this is also the \v Cech nerve of $\cal Y^{\flat}\rightarrow\mathbf Q^{(\op{pt}, \op{pt}/\exp(\fr k))}(\cal Y^{\flat})$.
\end{proof}

\begin{rem}
When $Y$ is any scheme over $S$ (\emph{not} necessarily locally of finite type) but $\eta$ is injective, we also have an identification of $\mathbf Q^{(\fr k, H)}_{\op{inj}}(\cal F)$ with the Lie algebroid on $\cal Q^{\op{desc}}$ with zero Lie bracket and anchor map. This follows immediately from the definition of $\mathbf Q_{\op{inj}}^{(\fr k, H)}(\cal F)$.
\end{rem}

Geometrically, the datum of $\eta$ gives rise to a map $\phi : \mathbb V(\cal F) \rightarrow Y\underset{S}{\times}\fr k^*$, and $\mathbb V(\cal Q)$ identifies with its fiber at $\{0\} \hookrightarrow \fr k^*$. Hence we have isomorphisms of DG stacks:
\begin{equation}
\label{eq-quotient-abelian-frmmod}
\mathbb V(\cal Q^{\op{desc}})\xrightarrow{\sim} \mathbb V(\cal Q)/H \xrightarrow{\sim} \phi^{-1}(0)/H.
\end{equation}

\subsubsection{} Suppose we have an exact sequence of $H$-equivariant $\cal O_Y$-modules:
$$
0\rightarrow \cal O_Y \rightarrow\widehat{\cal F} \rightarrow\cal F\rightarrow 0.
$$
Let $\widehat{\cal Y}^{\flat} \in \mathbf{QTw}_{/\cal Y^{\flat}}(Y/S)$ be the corresponding inert quasi-twisting. Assume that $\eta$ lifts to an $H$-equivariant map $\widehat{\eta} : \fr k\otimes\cal O_Y\rightarrow\widehat{\cal F}$. Then Proposition \ref{prop-quotient-abelian-frmmod} shows that the quotient quasi-twisting arises from a triangle in $\op{QCoh}(Y/H)$:
$$
\cal O_{Y/H} \rightarrow \widehat{\cal Q}^{\op{desc}} \rightarrow \cal Q^{\op{desc}}
$$
where $\widehat{\cal Q}^{\op{desc}}$ is the descent of $\widehat{\cal Q}:=\op{Cofib}(\widehat{\eta})$ to $Y/H$.

In particular, we have isomorphisms of DG stacks:
\begin{equation}
\label{eq-quotient-abelian-qtw}
\mathbb V(\widehat{\cal Q}^{\op{desc}})_{\lambda =1}\xrightarrow{\sim} \mathbb V(\widehat{\cal Q})_{\lambda =1}/H \xrightarrow{\sim} \widehat{\phi}^{-1}_{\lambda=1}(0)/H
\end{equation}
where $\widehat{\phi}_{\lambda=1}$ is the composition
$$
\mathbb V(\widehat{\cal F})_{\lambda=1}\hookrightarrow\mathbb V(\widehat{\cal F})\xrightarrow{\mathbb V(\widehat{\eta})} Y\underset{S}{\times}\fr k^*.
$$

\begin{rem}
In light of \eqref{eq-quotient-abelian-frmmod} and \eqref{eq-quotient-abelian-qtw}, we would like to think of $\mathbf Q^{(H,H^{\flat})}$ on inert quasi-twistings as an analogue of symplectic reduction where $\phi$ and $\widehat{\phi}_{\lambda=1}$ play the role of the moment map (but of course, with no symplectic structures involved \emph{a priori}.)
\end{rem}

\medskip

\part*{The universal quasi-twisting}

\section{Construction of $\cal T_G^{(\kappa, E)}$}
\label{sec-construction}

In this section, we construct a quasi-twisting $\cal T_G^{(\kappa, E)}$ over $S\times\op{Bun}_G$ (relative to $S$), which depends functorially on the parameter $(\fr g^{\kappa}, E) : S\rightarrow\op{Par}_G$. We proceed by first constructing a Lie-$*$ algebra $\widehat{\fr g}_{\cal D}^{(\kappa, E)}$ over $S\times X$, then twist its pullback to $S\times\op{Bun}_{G,\infty x}\times X$ by the tautological $G$-bundle $\tilde{\cal P}_G$. Via taking sections over $\overset{\circ}{D}_x$ and using the residue theorem, we produce a classical quasi-twisting $\tilde{\cal T}_G^{(\kappa ,E)}$ over $S\times\op{Bun}_{G,\infty x}$. Then we show that $\tilde{\cal T}_G^{(\kappa, E)}$ admits an action by the pair $(\fr g^{\kappa}(\cal O_x), \cal L_x^+G)$, so we may form the quotient $\cal T_G^{(\kappa, E)}:=\mathbf Q^{(\fr g^{\kappa}(\cal O_x), \cal L_x^+G)}(\tilde{\cal T}_G^{(\kappa, E)})$. As we shall see, this step requires both quotient functors constructed in \S\ref{sec-quotient} and their compatibility.

We then verify that for a simple group $G$ and $\fr g^{\kappa}$ arising from the bilinear form $\kappa=\lambda\cdot\op{Kil}$, the quasi-twisting $\cal T_G^{(\kappa,0)}$ identifies with the twisting given by $\lambda$-power of the determinant line bundle $\cal L_{G,\det}$ over $\op{Bun}_G$.

\subsection{Recollection on Lie-$*$ algebras} Fix a base scheme $S\in\mathbf{Sch}_{/k}$. Let $\cal X\rightarrow S$ be a smooth curve relative to $S$ with connected fibers.\footnote{For our applications, we will take $\cal X:=S\times X$.} In particular, the diagonal morphism $\Delta : \cal X\rightarrow \cal X\underset{S}{\times}\cal X$ is a closed immersion. Denote by $\cal D_{\cal X/S}\Mod^r$ the category of $\cal O_{\cal X}$-modules equipped with a right action of the relative differential operators $\cal D_{\cal X/S}$.

\subsubsection{}
A \emph{Lie-$*$ algebra} on $\cal X$ (relative to $S$) is an object $\cal B \in \cal D_{\cal X/S}\Mod^r$, equipped with a $\cal D_{\cal X\underset{S}{\times}\cal X/S}$-linear morphism\footnote{We use $\boxtimes$ to denote tensoring over $\cal O_S$.} $[-,-]:\cal B^{\boxtimes 2} \rightarrow \Delta_!(\cal B)$ such that the following properties are satisfied:
\begin{enumerate}[(a)]
	\item (\emph{anti-symmetry}) for all sections $a, b$ of $\cal B$, there holds
	$$
	\tilde\sigma_{12}([a\boxtimes b]) = - [b\boxtimes a],
	$$
	where $\tilde\sigma_{12}$ is the transposition morphism over $\cal X\underset{S}{\times}\cal X$ given by:
	$$
	\sigma_{12}^{-1}\Delta_!(\cal B)\rightarrow\Delta_!(\cal B);\quad\text{where}\quad\sigma_{12}(x,y)=(y,x).
	$$
	
	\item (\emph{Jacobi identity}) for all sections $a$, $b$, and $c$ of $\cal B$, there holds
	$$
	[[a\boxtimes b]\boxtimes c] + \tilde\sigma_{123}([[b\boxtimes c]\boxtimes a]) + \tilde\sigma_{123}^2([[c\boxtimes a]\boxtimes b]) = 0,
	$$
	where $\tilde\sigma_{123}$ denotes the morphism over $\cal X\underset{S}{\times}\cal X\underset{S}{\times}\cal X$ given by:
	$$
	\sigma_{123}^{-1}(\Delta_{x=y=z})_!(\cal B) \rightarrow (\Delta_{x=y=z})_!(\cal B);\quad\text{where}\quad \sigma_{123}(x,y,z)=(y,z,x).
	$$
\end{enumerate}

Denote by $\mathbf{Lie}^*(\cal X/S)$ the category of Lie-$*$ algebras on $\cal X$ relative to $S$. Clearly, for any morphism $S'\rightarrow S$ with $\cal X':=\cal X\underset{S}{\times}S'$, we have a functor $\mathbf{Lie}^*(\cal X/S)\rightarrow\mathbf{Lie}^*(\cal X'/S')$ acting as pulling back a $\cal D_{\cal X/S}$-module, and equipping it with the induced Lie-$*$ algebra structure.

\subsubsection{} Lie-$*$ algebras are \'etale local objects. More precisely, let $\textbf{\'Et}_{/\cal X}$ be the small \'etale site of $\cal X$. Given $\cal B\in\mathbf{Lie}^*(\cal U/S)$ where $\cal U\in\textbf{\'Et}_{/\cal X}$ and a morphism $\tilde{\cal U}\rightarrow\cal U$, we may associate an object $\cal B\big|_{\tilde{\cal U}}\in\mathbf{Lie}^*(\tilde{\cal U}/S)$. This procedure defines a functor in groupoids:
\begin{equation}
\label{eq-lie*-etale-local}
\textbf{\'Et}_{/\cal X}^{\op{op}} \rightarrow \op{Gpd},\quad \cal U\leadsto\mathbf{Lie}^*(\cal U/S).
\end{equation}
The \'etale local nature of Lie-$*$ algebras refers to the fact that \eqref{eq-lie*-etale-local} satisfies descent.

\subsubsection{} Let $\cal G$ be a presheaf of group schemes on $\textbf{\'Et}_{/\cal X}$, and $\cal B\in\mathbf{Lie}^*(\cal X/S)$. A \emph{$\cal G$-action} on $\cal L$ consists of the following data:
\begin{itemize}
	\item for each $\cal U\in\textbf{\'Et}_{/\cal X}$, an action of $\cal G_{\cal U}$ as endomorphisms of $\cal B\big|_{\cal U}\in\mathbf{Lie}^*(\cal U/S)$;
\end{itemize}
furthermore, this action is required to be functorial in $\cal U$.

Suppose $\cal P$ is an \'etale $\cal G$-torsor over $\cal X$, and $\cal B\in\mathbf{Lie}^*(\cal X/S)$ admits a $\cal G$-action. Then we can form the \emph{$\cal P$-twisted Lie-$*$ algebra} $\cal B_{\cal P} \in \mathbf{Lie}^*(\cal X/S)$ using the descent property of \eqref{eq-lie*-etale-local}.

\subsubsection{} Suppose we have a section $\underline x : S\hookrightarrow\cal X$. Let $D_{\underline x}$ be the completion of $\underline x$ and $\overset{\circ}{D}_{\underline x} := D_{\underline x}-\{\underline x\}$ be its localization.

\begin{eg}
For $\cal X=S\times X$, any closed point $x\in X$ determines a section as above. Affine locally on $S$, we have $D_{\underline x}\xrightarrow{\sim}\op{Spec}(\cal O_S\widehat{\otimes}\cal O_x)$ and $\overset{\circ}{D}_{\underline x}\xrightarrow{\sim}\op{Spec}(\cal O_S\widehat{\otimes}\cal K_x)$, where $\cal O_x$ denotes the completed local ring at $x$, and $\cal K_x$ the localization of $\cal O_x$ at its uniformizer.
\end{eg}

The functor $\Gamma_{\dR}$ of \emph{zero}th de Rham cohomology gives rise to functors:
$$
\Gamma_{\dR}(D_{\underline x}, -),\;\Gamma_{\dR}(\overset{\circ}{D}_{\underline x}, -) : \cal D_{\cal X/S}\Mod^r_{\op{Coh}} \rightarrow \op{QCoh}^{\op{Tate}}(S).
$$
where $\cal D_{\cal X/S}\Mod^r_{\op{Coh}}$ denotes the category of $\cal D_{\cal X/S}$-coherent modules (see \cite[\S2.1.13-16]{BD04}). Furthermore, given a Lie-$*$ algebra $\cal B$, the object $\Gamma_{\dR}(\overset{\circ}{D}_{\underline x}, \cal B)$ acquires the structure of a Lie algebra in $\op{QCoh}^{\op{Tate}}(S)$, whose (continuous) Lie bracket is given by the composition:
$$
[-,-] : \Gamma_{\dR}(\overset{\circ}{D}_{\underline x}, \cal B)^{\boxtimes 2} \xrightarrow{\sim} \Gamma_{\dR}(\overset{\circ}{D}_{\underline x}\underset{S}{\times}\overset{\circ}{D}_{\underline x}, \cal B^{\boxtimes 2}) \rightarrow \Gamma_{\dR}(\overset{\circ}{D}_{\underline x}\underset{S}{\times}\overset{\circ}{D}_{\underline x}, \Delta_!(\cal B))\xrightarrow{\sim} \Gamma_{\dR}(\overset{\circ}{D}_{\underline x}, \cal B).
$$
The map $\Gamma_{\dR}(D_{\underline x}, \cal B)\rightarrow\Gamma_{\dR}(\overset{\circ}{D}_{\underline x}, \cal B)$ realizes $\Gamma_{\dR}(D_{\underline x}, \cal B)$ as a Lie subalgebra if $\cal B$ is $\cal O_{\cal X}$-flat.

\subsection{The Kac-Moody Lie-$*$ algebra}
\label{sec-kac-moody}
Suppose now that $S\in\mathbf{Sch}_{/k}$ is equipped with a morphism $S\rightarrow\op{Par}_G$, represented by $(\fr g^{\kappa}, E)$ (see \S\ref{sec-quantum-parameter}). We will construct a central extension of Lie-$*$ algebras over $\cal X:=S\times X$:
\begin{equation}
\label{eq-kac-moody-extension}
0 \rightarrow \omega_{\cal X/S} \rightarrow \widehat{\fr g}_{\cal D}^{(\kappa,E)} \rightarrow \fr g_{\cal D}^{\kappa} \rightarrow 0,
\end{equation}
together with $\cal G$-actions on $\widehat{\fr g}_{\cal D}^{(\kappa,E)}$ and $\fr g_{\cal D}^{\kappa}$, where $\cal G$ is the presheaf of group schemes $\cal G_{\cal U}:=\underline{\op{Maps}}(\cal U,G)$ on $\textbf{\'Et}_{/\cal X}$. The construction will be functorial in $S$.

\begin{rem}
The central extension \eqref{eq-kac-moody-extension}, together with the $\cal G$-action, is called the \emph{(generalized) Kac-Moody} central extension of Lie-$*$ algebras, and we refer to $\widehat{\fr g}_{\cal D}^{(\kappa,E)}$ as the \emph{(generalized) Kac-Moody} Lie-$*$ algebra.
\end{rem}

\subsubsection{} The Lie-$*$ algebra $\fr g_{\cal D}^{\kappa}$ has underlying $\cal D_{\cal X/S}$-module $\fr g^{\kappa}\boxtimes\cal D_{\cal X/S}$. Its Lie-$*$ algebra structure is defined using the Lie bracket \eqref{eq-structure-lie} on $\fr g^{\kappa}$:
$$
[-,-] : (\fr g_{\cal D}^{\kappa})^{\boxtimes 2}\rightarrow\Delta_!(\fr g_{\cal D}^{\kappa}),\quad (\mu\otimes\mathbf 1)\boxtimes(\mu'\otimes\mathbf 1) \leadsto [\mu,\mu']\otimes\mathbf 1_{\cal D},
$$
where $\mathbf 1_{\cal D}$ is the canonical symmetric section of $\Delta_!(\cal D_{\cal X/S})$. Note that the Lie-$*$ bracket $[-,-]$ factors through the embedding $\fr g_{\op{s.s.}}^{\kappa}\boxtimes\cal D_{\cal X/S}\hookrightarrow\fr g_{\cal D}^{\kappa}$.\footnote{See \S\ref{sec-reduction-to-center} for the notation $\fr g_{\op{s.s.}}^{\kappa}$.}

We construct a $\cal G$-action on $\fr g^{\kappa}_{\cal D}$ as follows: for every $\cal U\in\textbf{\'Et}_{/\cal X}$, there is an \emph{adjoint-coadjoint} action of the group scheme $\underline{\op{Maps}}(\cal U, G)$ on $\fr g^{\kappa}\otimes\cal O_{\cal U}$:
\begin{equation}
\label{group-action-on-last-term}
g_{\cal U} \cdot (\xi\oplus\varphi) = \op{Ad}_{g_{\cal U}}(\xi) \oplus \op{Coad}_{g_{\cal U}}(\varphi).
\end{equation}
where $\xi\oplus\varphi$ denotes a section of $\fr g^{\kappa}\otimes\cal O_{\cal U}$, regarded as a subbundle of $(\fr g\otimes\cal O_{\cal U})\oplus(\fr g^*\otimes\cal O_{\cal U})$. The action \eqref{group-action-on-last-term} extends to an action of $\underline{\op{Maps}}(\cal U, G)$ on $\fr g^{\kappa}\underset{\cal O_{\cal U}}{\otimes}\cal D_{\cal U/S}$ by Lie-$*$ algebra endomorphisms.

\subsubsection{} The underlying $\cal D_{\cal X/S}$-modules of \eqref{eq-kac-moody-extension} are defined by first inducing a sequence of $\cal D_{\cal X/S}$-modules from \eqref{eq-canonical-module-ext}:
\begin{equation}
\label{eq-canonical-d-module-ext}
0 \rightarrow \omega_{\cal X/S}\underset{\cal O_{\cal X}}{\otimes}\cal D_{\cal X/S} \rightarrow \widehat{\fr g}^{\kappa}\underset{\cal O_{\cal X}}{\otimes}\cal D_{\cal X/S} \rightarrow \fr g^{\kappa}\boxtimes\cal D_{\cal X/S} \rightarrow 0
\end{equation}
and then taking the push-out along the action map $\omega_{\cal X/S}\underset{\cal O_{\cal X}}{\otimes}\cal D_{\cal X/S} \rightarrow \omega_{\cal X/S}$.

In particular, the extension $\widehat{\fr g}_{\cal D}^{(\kappa,E)}\rightarrow\fr g_{\cal D}^{\kappa}$ splits over $\fr g_{\op{s.s.}}^{\kappa}\boxtimes\cal D_{\cal X/S}$, and we have a decomposition
\begin{equation}
\label{eq-d-module-splitting}
\widehat{\fr g}_{\cal D}^{(\kappa,E)} \xrightarrow{\sim} E_{\cal D} \oplus (\fr g_{\op{s.s.}}^{\kappa}\boxtimes\cal D_{\cal X/S}).
\end{equation}
where $E_{\cal D}$ is the push-out of $E\underset{\cal O_{\cal X}}{\otimes}\cal D_{\cal X/S}$ along $\omega_{\cal X/S}\underset{\cal O_{\cal X}}{\otimes}\cal D_{\cal X/S} \rightarrow \omega_{\cal X/S}$.

\subsubsection{}
The Lie-$*$ algebra structure on $\widehat{\fr g}_{\cal D}^{(\kappa,E)}$ is defined by the composition:
$$
(\widehat{\fr g}_{\cal D}^{(\kappa,E)})^{\boxtimes 2} \rightarrow (\fr g^{\kappa}_{\cal D})^{\boxtimes 2} \rightarrow \Delta_!(\omega_{\cal X/S})\oplus \Delta_!(\fr g_{\op{s.s.}}^{\kappa}\boxtimes\cal D_{\cal X/S}) \rightarrow \Delta_!(\widehat{\fr g}_{\cal D}^{(\kappa,E)})
$$
where the middle map is defined using the bilinear form \eqref{eq-structure-form} and the Lie bracket \eqref{eq-structure-lie} on $\fr g^{\kappa}$:
$$
(\mu\otimes\mathbf 1)\boxtimes(\mu'\otimes\mathbf 1) \leadsto (\mu, \mu')\mathbf 1'_{\omega} + [\mu,\mu']\otimes\mathbf 1_{\cal D};
$$
the notation $\mathbf 1'_{\omega}$ denotes the canonical anti-symmetric section of $\Delta_!(\omega_{\cal X/S})$.

\subsubsection{}
\label{sec-g-action-on-kac-moody} We now construct the $\cal G$-action on $\widehat{\fr g}_{\cal D}^{(\kappa,E)}$. Let $\cal U\in\textbf{\'Et}_{/\cal X}$ and $g_{\cal U}$ be a point of $\underline{\op{Maps}}(\cal U, G)$. The corresponding endomorphism $g_{\cal U} : \widehat{\fr g}_{\cal D}^{(\kappa,E)} \rightarrow \widehat{\fr g}_{\cal D}^{(\kappa,E)}$ is defined by the sum of the following maps (using the decomposition \eqref{eq-d-module-splitting}):
\begin{itemize}
	\item identity on $E_{\cal D}$;
	\item adjoint-coadjoint action on $\fr g_{\op{s.s.}}^{\kappa}\boxtimes\cal D_{\cal U/S}$ by formula \eqref{group-action-on-last-term};
	\item the composition:
	\begin{equation}
	\label{eq-g-action-on-kac-moody}
	\widehat{\fr g}_{\cal D}^{(\kappa,E)}\big|_{\cal U} \rightarrow \fr g_{\cal D}^{\kappa}\big|_{\cal U} \xrightarrow{\sim} (\fr g^{\kappa}\boxtimes\cal O_{\cal U}) \underset{\cal O_{\cal U}}{\otimes}\cal D_{\cal X/S} \xrightarrow{\op{res}(g_{\cal U})} \omega_{\cal U/S} \hookrightarrow \widehat{\fr g}_{\cal D}^{(\kappa,E)}\big|_{\cal U}
	\end{equation}
	where the map $\op{res}(g_{\cal U})$ is defined by the formula:
	$$
	(\xi\oplus\varphi)\otimes\mathbf 1 \leadsto \varphi(g_{\cal U}^{-1} dg_{\cal U}),\quad \xi\oplus\varphi \in \fr g^{\kappa}\boxtimes\cal O_{\cal U}.
	$$
	Here, $d : \cal O_{\cal U} \rightarrow \omega_{\cal U/S}$ is the exterior derivative, so $g_{\cal U}^{-1}dg_{\cal U}$ is a section of $\fr g\boxtimes\cal O_{\cal U}$, on which $\varphi$ rightfully acts.
\end{itemize}
It is clear from the construction that $\widehat{\fr g}_{\cal D}^{(\kappa,E)}\rightarrow\fr g_{\cal D}^{\kappa}$ is $\cal G$-equivariant.

\begin{rem}
If $\fr g^{\kappa}$ arises from a symmetric bilinear form $\kappa$ (see \S\ref{sec-quantum-parameter}), then we have an isomorphism $\widehat{\fr g}_{\cal D}^{(\kappa, 0)} \xrightarrow{\sim} \cal B(\fr g, \kappa)$ where $\cal B(\fr g,\kappa)$ is the Kac-Moody Lie-$*$ algebra at level $\kappa$ in the ordinary sense (see \cite{Ga98}).
\end{rem}

\subsubsection{}
\label{sec-kac-moody-bracket} Let us bring in the closed point $x\in X$, which induces a section $\underline x : S\rightarrow\cal X$. Applying $\Gamma_{\dR}(\overset{\circ}{D}_{\underline x}, -)$ to the sequence \eqref{eq-kac-moody-extension}, we obtain a central extension of Lie algebras in $\op{QCoh}^{\op{Tate}}(S)$:
\begin{equation}
\label{eq-kac-moody-extension-tate}
0 \rightarrow \cal O_S \rightarrow \widehat{\fr g}^{(\kappa,E)} \rightarrow \fr g^{\kappa}(\cal K_x) \rightarrow 0,
\end{equation}
where the notation $\fr g^{\kappa}(\cal O_x)$ (resp.~$\fr g^{\kappa}(\cal K_x)$) denotes the Tate $\cal O_S$-module $\fr g^{\kappa}\widehat{\otimes}\cal O_x$ (resp.~localization at the uniformizer of $\cal O_x$.)

The Lie bracket on $\widehat{\fr g}^{(\kappa,E)}$ is given by the composition:
$$
(\widehat{\fr g}^{(\kappa,E)})^{\boxtimes 2} \rightarrow (\fr g^{\kappa}(\cal K_x))^{\boxtimes 2} \rightarrow \cal O_S \oplus \fr g_{\op{s.s.}}^{\kappa}(\cal K_x) \rightarrow \widehat{\fr g}^{(\kappa, E)},
$$
where the middle map is defined by
$$
(\mu\otimes f)\boxtimes(\mu'\otimes f') \leadsto (\mu,\mu')\cdot\op{Res}((df)f') + [\mu,\mu']\otimes ff'.
$$

\begin{lem}
The central extension \eqref{eq-kac-moody-extension-tate} canonically splits over $\fr g^{\kappa}(\cal O_x)$.
\end{lem}
\begin{proof}
The result follows from applying $\Gamma_{\dR}(D_{\underline x}, -)$ to the sequence \eqref{eq-kac-moody-extension} and observing that $\Gamma_{\dR}(D_{\underline x}, \omega_{\cal X/S})$ vanishes.
\end{proof}

Let $\cal L_xG$ (resp.~$\cal L^+_xG$) denote the loop (resp.~arc) group of $G$ at $x$. There is an action of $\cal L_xG$ on $\widehat{\fr g}^{(\kappa, E)}$ defined analogously to \S\ref{sec-g-action-on-kac-moody}, with the composition \eqref{eq-g-action-on-kac-moody} replaced by:
$$
\widehat{\fr g}^{(\kappa,E)} \rightarrow \fr g^{\kappa}(\cal K_x) \xrightarrow{\op{res}(g)} \cal O_S \hookrightarrow\widehat{\fr g}^{(\kappa,E)}
$$
where the map $\op{res}(g)$ ($g$ is a point of $\cal L_xG$) is defined by the formula:
$$
(\xi\oplus\varphi)\otimes f \leadsto \op{Res}(f\cdot\varphi(g^{-1}dg)).
$$
Since the Lie algebra of $\cal L_xG$ identifies with $\fr g(\cal K_x)$, this $\cal L_xG$-action induces a $\fr g(\cal K_x)$-action on $\widehat{\fr g}^{(\kappa,E)}$ by $\cal O_S$-linear endomorphisms.

\begin{lem}
\label{lem-alternative-lie-bracket}
The Lie bracket on $\widehat{\fr g}^{(\kappa,E)}$ agrees with the composition:
$$
(\widehat{\fr g}^{(\kappa,E)})^{\boxtimes 2} \xrightarrow{(\op{pr},\op{id})} \fr g(\cal K_x) \boxtimes \widehat{\fr g}^{(\kappa,E)}\xrightarrow{\op{act}} \widehat{\fr g}^{(\kappa,E)}.
$$
\end{lem}
\begin{proof}
This is a straightforward computation.
\end{proof}

\subsection{The classical quasi-twisting $\tilde{\cal T}_G^{(\kappa,E)}$ over $\op{Bun}_{G,\infty x}$} Let $\op{Bun}_{G,\infty x}$ denote the stack classifying pairs $(\cal P_G, \alpha)$ where $\cal P_G$ is a $G$-bundle on $X$ and $\alpha : \cal P_G\big|_{D_x}\xrightarrow{\sim}\cal P_G^0$ is a trivialization over $D_x$. The (right) $\cal L_x^+G$-action on $\op{Bun}_{G,\infty x}$ by changing $\alpha$ realizes $\op{Bun}_{G,\infty x}$ as a $\cal L_x^+G$-bundle over $\op{Bun}_G$, locally trivial in the \'etale topology. In particular, $\op{Bun}_{G,\infty x}$ is placid; see \S\ref{sec-elephant}.

\subsubsection{} The Beauville-Laszlo theorem shows that $\op{Bun}_{G,\infty x}$ also classifies pairs $(\cal P_{G,\Sigma}, \alpha)$, where $\cal P_{G,\Sigma}$ is a $G$-bundle on $\Sigma:=X-\{x\}$ and $\alpha : \cal P_{G,\Sigma}\big|_{\overset{\circ}{D}_x}\xrightarrow{\sim}\cal P_G^0$ is a trivialization over $\overset{\circ}{D}_x$. This alternative description shows that the $\cal L_x^+G$-action on $\op{Bun}_{G,\infty x}$ extends to an $\cal L_xG$-action.

\subsubsection{} Fix an $S$-point $(\fr g^{\kappa}, E)$ of $\op{Par}_G$. We apply the construction of \S\ref{sec-kac-moody} to the relative curve
$$
\tilde{\cal X}:= S\times\op{Bun}_{G,\infty x}\times X\quad\text{over}\quad \tilde{\cal S}:= S\times\op{Bun}_{G,\infty x},
$$
and obtain a central extension in $\mathbf{Lie}^*(\tilde{\cal X}/\tilde{\cal S})$:
\begin{equation}
\label{eq-kac-moody-extension-univ}
0 \rightarrow \omega_{\tilde{\cal X}/\tilde{\cal S}} \rightarrow \widehat{\fr g}_{\cal D}^{(\kappa,E)} \rightarrow \fr g_{\cal D}^{\kappa} \rightarrow 0.
\end{equation}
In other words, \eqref{eq-kac-moody-extension-univ} is the image of Kac-Moody extension \eqref{eq-kac-moody-extension} under the base change functor $-\boxtimes\cal O_{\op{Bun}_{G,\infty x}} : \mathbf{Lie}^*(\cal X/S)\rightarrow\mathbf{Lie}^*(\tilde{\cal X}/\tilde{\cal S})$.

Let $\underline{\tilde x} : \tilde{\cal S}\hookrightarrow \tilde{\cal X}$ (resp.~$\underline x:S\hookrightarrow \cal X$) denote the section given by $x\in X$. Let $\tilde{\cal P}_G$ be the tautological $G$-bundle over $\tilde{\cal X}$ equipped with the trivialization $\alpha$ over $D_{\underline{\tilde x}}$. Since $\widehat{\fr g}_{\cal D}^{(\kappa, E)}$ and $\fr g_{\cal D}^{\kappa}$ are equipped with $\cal G$-actions, we can form the $\tilde{\cal P}_G$-twist of \eqref{eq-kac-moody-extension-univ}:
\begin{equation}
\label{eq-kac-moody-extension-twist}
0 \rightarrow \omega_{\tilde{\cal X}/\tilde{\cal S}} \rightarrow (\widehat{\fr g}_{\cal D}^{(\kappa,E)})_{\tilde{\cal P}_G} \rightarrow (\fr g_{\cal D}^{\kappa})_{\tilde{\cal P}_G} \rightarrow 0.
\end{equation}

\begin{rem}
\begin{itemize}
	\item Since $\fr g_{\cal D}^{\kappa}$ is the $\cal D_{\tilde{\cal X}/\tilde{\cal S}}$-module induced from $\fr g^{\kappa}\boxtimes\cal O_{\op{Bun}_{G,\infty x}\times X}$ and the $\cal G$-action comes from one on $\fr g^{\kappa}\boxtimes\cal O_{\op{Bun}_{G,\infty x}\times X}$, we see that $(\fr g_{\cal D}^{\kappa})_{\tilde{\cal P}_G}$ is the $\cal D_{\tilde{\cal X}/\tilde{\cal S}}$-module induced from $\fr g^{\kappa}_{\tilde{\cal P}_G}$.
	
	\item the datum of $\alpha$ gives an isomorphism between \eqref{eq-kac-moody-extension-univ} and \eqref{eq-kac-moody-extension-twist} when restricted to $D_{\tilde x}$.
\end{itemize}
\end{rem}

We apply the functors $\Gamma_{\dR}(\Sigma,-)$ and $\Gamma_{\dR}(\overset{\circ}{D}_{\underline{\tilde x}}, -)$ to \eqref{eq-kac-moody-extension-twist}. Using the two observations above, we obtain a morphism between two triangles in $\op{QCoh}^{\op{Tate}}(\tilde{\cal S})$:
\begin{equation}
\label{eq-main-splitting}
\xymatrix@C=1.5em@R=1.5em{
	\Gamma_{\dR}(\Sigma, \omega_{\tilde{\cal X}/\tilde{\cal S}}) \ar[r]\ar[d] & \Gamma_{\dR}(\Sigma, (\widehat{\fr g}_{\cal D}^{(\kappa,E)})_{\tilde{\cal P}_G}) \ar[r]\ar[d] & \Gamma(\Sigma, \fr g^{\kappa}_{\tilde{\cal P}_G}) \ar[d]^{\gamma} \ar@{.>}[dl]_{\widehat{\gamma}} \\
	\Gamma_{\dR}(\overset{\circ}{D}_{\underline{\tilde x}}, \omega_{\tilde{\cal X}/\tilde{\cal S}}) \ar[r] & \Gamma_{\dR}(\overset{\circ}{D}_{\underline{\tilde x}}, \widehat{\fr g}_{\cal D}^{(\kappa,E)}) \ar[r] & \fr g^{\kappa}(\cal K_x)\widehat{\boxtimes}\cal O_{\op{Bun}_{G,\infty x}}
}
\end{equation}
where $\fr g^{\kappa}(\cal K_x)$ is (as before) an object of $\op{QCoh}^{\op{Tate}}(S)$.

By residue theory, the first vertical map in \eqref{eq-main-splitting} vanishes. Hence we obtain a splitting $\widehat{\gamma}$ as depicted. Note that $\gamma$ (hence $\widehat{\gamma}$) is injective, so we may define two Tate $\cal O_{\tilde{\cal S}}$-modules by cokernels without running into DG issues:
$$
\widehat{\cal L}^{(\kappa, E)}:=\op{Coker}(\widehat{\gamma}),\quad \cal L^{\kappa}:=\op{Coker}(\gamma).
$$
Since $\Gamma_{\dR}(\overset{\circ}{D}_{\underline{\tilde x}}, \omega_{\tilde{\cal X}/\tilde{\cal S}})$ is canonically isomorphic to $\cal O_{\tilde{\cal S}}$, we arrive at an exact sequence of Tate $\cal O_{\tilde{\cal S}}$-modules:
\begin{equation}
\label{eq-qtw-full-level}
0 \rightarrow \cal O_{\tilde{\cal S}} \rightarrow \widehat{\cal L}^{(\kappa, E)} \rightarrow \cal L^{\kappa} \rightarrow 0
\end{equation}

\begin{nota}
In what follows, we will show that \eqref{eq-qtw-full-level} has the structure of a classical quasi-twisting (on Tate modules) over $\tilde{\cal S}$ (relative to $S$; see \S\ref{sec-tate-liealgd}), denoted by $\tilde{\cal T}_G^{(\kappa, E)}$.
\end{nota}

\subsubsection{} We (temporarily) use the notation $\widehat{\fr g}_{\cal D, \cal X}^{(\kappa,E)}$ to denote the Kac-Moody Lie-$*$ algebra over $\cal X$, constructed using the recipe in \S\ref{sec-kac-moody} for the relative curve $\cal X\rightarrow S$.

The isomorphism $\widehat{\fr g}_{\cal D}^{(\kappa,E)}\xrightarrow{\sim} \widehat{\fr g}_{\cal D, \cal X}^{(\kappa, E)}\boxtimes\cal O_{\op{Bun}_{G,\infty x}}$ gives rise to an isomorphism in $\op{QCoh}^{\op{Tate}}(\tilde{\cal S})$:
\begin{equation}
\label{eq-decompose-loop-sections}
\Gamma_{\dR}(\overset{\circ}{D}_{\underline{\tilde x}}, \widehat{\fr g}_{\cal D}^{(\kappa,E)}) \xrightarrow{\sim} \Gamma_{\dR}(\overset{\circ}{D}_{\underline x}, \widehat{\fr g}_{\cal D, \cal X}^{(\kappa, E)})\widehat{\boxtimes} \cal O_{\op{Bun}_{G,\infty x}} \cong \widehat{\fr g}^{(\kappa,E)}\widehat{\boxtimes}\cal O_{\op{Bun}_{G,\infty x}}
\end{equation}
Observe that the $G(\cal K_x)$-action on $\op{Bun}_{G,\infty x}$ gives rise to a $\fr g(\cal K_x)$-action\footnote{Unlike the Tate $\cal O_S$-module $\fr g^{\kappa}(\cal K_x)$, the notation $\fr g(\cal K_x)$ is reserved for the Tate vector space $\fr g\otimes\cal K_x$ (similar for the notation $\fr g(\cal O_x)$.)} on $\cal O_{\op{Bun}_{G,\infty x}}$ by derivations. Hence, the Lie (algebroid) bracket on $\Gamma_{\dR}(\overset{\circ}{D}_{\underline {\tilde x}}, \widehat{\fr g}_{\cal D}^{(\kappa,E)})$ can be defined using the $\cal O_S$-linear Lie bracket on $\widehat{\fr g}^{(\kappa,E)}$ (see \S\ref{sec-kac-moody-bracket}):
$$
[\mu\boxtimes f, \mu'\boxtimes f'] := [\mu, \mu'] + \overline{\mu}(f')\cdot\mu' - \overline{\mu}'(f)\cdot \mu.
$$
where $\overline{\mu}$ denotes the image of $\mu\in \widehat{\fr g}^{(\kappa,E)}$ along $\widehat{\fr g}^{(\kappa,E)} \rightarrow \fr g^{\kappa}(\cal K_x) \rightarrow \fr g(\cal K_x)\widehat{\boxtimes}\cal O_S$, which acts on $\cal O_{\tilde{\cal S}}$ by $\cal O_S$-linear derivations. The anchor map $\widehat{\sigma}$ of $\Gamma_{\dR}(\overset{\circ}{D}_{\underline{\tilde x}}, \widehat{\fr g}_{\cal D}^{(\kappa,E)})$ is defined by the composition:
\begin{equation}
\label{eq-anchor-map-full-level}
\Gamma_{\dR}(\overset{\circ}{D}_{\underline {\tilde x}}, \widehat{\fr g}_{\cal D}^{(\kappa,E)}) \xrightarrow{\eqref{eq-decompose-loop-sections}} \widehat{\fr g}^{(\kappa,E)}\widehat{\boxtimes}\cal O_{\op{Bun}_{G,\infty x}} \rightarrow \fr g(\cal K_x)\widehat{\boxtimes}\cal O_{\tilde{\cal S}} \rightarrow \cal T_{\tilde{\cal S}/S}.
\end{equation}

We have thus equipped $\Gamma_{\dR}(\overset{\circ}{D}_{\underline{\tilde x}}, \widehat{\fr g}_{\cal D}^{(\kappa,E)})$ with the structure of a Lie algebroid. The following lemma, whose proof is deferred to \S\ref{lem-liealgd-ideal}, extends this Lie algebroid structure to its quotient $\widehat{\cal L}^{(\kappa,E)}$:

\begin{lem}
\label{lem-liealgd-ideal}
The morphism $\widehat{\gamma}$ realizes $\Gamma(\Sigma, \fr g_{\tilde{\cal P}_G}^{\kappa})$ as an ideal of $\Gamma_{\dR}(\overset{\circ}{D}_{\underline{\tilde x}}, \widehat{\fr g}_{\cal D}^{(\kappa,E)})$.
\end{lem}

In an analogous way, we turn $\fr g^{\kappa}(\cal K_x)\widehat{\boxtimes}\cal O_{\op{Bun}_{G,\infty x}}$ into an object of $\mathbf{LieAlgd}(\tilde{\cal S}/S)$, and the map $\Gamma_{\dR}(\overset{\circ}{D}_{\underline{\tilde x}}, \widehat{\fr g}_{\cal D}^{(\kappa,E)}) \rightarrow \fr g^{\kappa}(\cal K_x)\widehat{\boxtimes}\cal O_{\op{Bun}_{G,\infty x}}$ in \eqref{eq-main-splitting} is a morphism of such. Lemma \ref{lem-liealgd-ideal} shows that $\gamma$ also realizes $\Gamma(\Sigma, \fr g_{\tilde{\cal P}_G}^{\kappa})$ as an ideal of $\fr g^{\kappa}(\cal K_x)\widehat{\boxtimes}\cal O_{\op{Bun}_{G,\infty x}}$. Hence the cokernels \eqref{eq-qtw-full-level} is a central extension of Lie algebroids.

\subsubsection{Proof of Lemma \ref{lem-liealgd-ideal}}
\label{sec-liealgd-ideal}
We first give an alternative description of the Lie bracket on $\Gamma_{\dR}(\overset{\circ}{D}_{\underline{\tilde x}}, \widehat{\fr g}_{\cal D}^{(\kappa,E)})$. Indeed, from the identification in \eqref{eq-decompose-loop-sections} and the $\fr g(\cal K_x)$-action on $\widehat{\fr g}^{(\kappa,E)}$ (see \S\ref{sec-kac-moody-bracket}), we obtain an action of $\fr g(\cal K_x)\widehat{\boxtimes}\cal O_{\tilde{\cal S}}$ on $\Gamma_{\dR}(\overset{\circ}{D}_{\underline {\tilde x}}, \widehat{\fr g}_{\cal D}^{(\kappa,E)})$ by $\cal O_S$-linear derivations. It follows from Lemma \ref{lem-alternative-lie-bracket} that the Lie bracket on $\Gamma_{\dR}(\overset{\circ}{D}_{\underline {\tilde x}}, \widehat{\fr g}_{\cal D}^{(\kappa,E)})$ agrees with the composition:
\begin{equation}
\label{eq-alternative-lie-bracket}
\Gamma_{\dR}(\overset{\circ}{D}_{\underline {\tilde x}}, \widehat{\fr g}_{\cal D}^{(\kappa,E)})^{\boxtimes 2} \xrightarrow{(\op{pr}, \op{id})} (\fr g(\cal K_x)\widehat{\boxtimes}\cal O_{\tilde{\cal S}})\boxtimes \Gamma_{\dR}(\overset{\circ}{D}_{\underline {\tilde x}}, \widehat{\fr g}_{\cal D}^{(\kappa,E)}) \xrightarrow{\op{act}} \Gamma_{\dR}(\overset{\circ}{D}_{\underline {\tilde x}}, \widehat{\fr g}_{\cal D}^{(\kappa,E)}),
\end{equation}
where $\op{pr}$ denotes the composition of the first two maps in \eqref{eq-anchor-map-full-level}.

Therefore, it suffices to show that the Tate $\cal O_{\tilde{\cal S}}$-submodule:
\begin{equation}
\label{eq-equivariant-submodule}
\Gamma_{\dR}(\Sigma, (\widehat{\fr g}_{\cal D}^{(\kappa, E)})_{\tilde{\cal P}_G}) \hookrightarrow \Gamma_{\dR}(\overset{\circ}{D}_{\underline {\tilde x}}, \widehat{\fr g}_{\cal D}^{(\kappa, E)})
\end{equation}
is invariant under the aforementioned $\fr g(\cal K_x)\widehat{\boxtimes}\cal O_{\tilde{\cal S}}$-action. Note that by construction, this action arises from the $S\times\cal L_xG$-equivariance structure on $\Gamma_{\dR}(\overset{\circ}{D}_{\underline {\tilde x}}, \widehat{\fr g}_{\cal D}^{(\kappa, E)})$. The following claim is immediate:

\begin{claim}
There is also an $S\times\cal L_xG$-equivariance structure on $\Gamma_{\dR}(\Sigma, (\widehat{\fr g}_{\cal D}^{(\kappa, E)})_{\tilde{\cal P}_G})$, defined at every $T$-point $(s,\cal P_{G,\Sigma},\alpha, g)$ of $S\times\op{Bun}_{G,\infty x}\times\cal L_xG$ (for $T\in\Sch_{/k}^{\op{aff}}$) by:
\begin{itemize}
	\item first identifying the fiber of $\Gamma_{\dR}(\Sigma, (\widehat{\fr g}_{\cal D}^{(\kappa, E)})_{\tilde{\cal P}_G})$ at both of the $T$-points
	$$
	(s,\cal P_{G,\Sigma}, \alpha),\;\text{and}\;(s,\cal P_{G,\Sigma}, g\cdot\alpha),\quad g\in\op{Maps}(T,\cal L_xG),
	$$
	with $\Gamma_{\dR}(\Sigma, (\widehat{\fr g}_{\cal D}^{(\kappa,E)})_{\cal P_{G,\Sigma}})$;\footnote{We are slightly abusing the notation $(\widehat{\fr g}_{\cal D}^{(\kappa, E)})_{\cal P_{G,\Sigma}}$, since this is now the Kac-Moody extension associated to the parameter $T\xrightarrow{s} S \xrightarrow{(\fr g^{\kappa},E)}\op{Par}_G$, twisted by $\cal P_{G,\Sigma}$ on the open curve $T\times\Sigma$.}
	\item relating the above two fibers via the identity map on $\Gamma_{\dR}(\Sigma, (\widehat{\fr g}_{\cal D}^{(\kappa,E)})_{\cal P_{G,\Sigma}})$. \qed
\end{itemize}
\end{claim}

So we have reduced the problem to showing that \eqref{eq-equivariant-submodule} preserves the $S\times\cal L_xG$-equivariance structure. In other words, the following diagram in $\op{QCoh}^{\op{Tate}}(T)$ needs to commute:
\begin{equation}
\label{eq-equivariant-submodule-diagram}
\xymatrix@C=1.5em@R=1.5em{
	\Gamma_{\dR}(\Sigma, (\widehat{\fr g}_{\cal D}^{(\kappa,E)})_{\cal P_{G,\Sigma}}) \ar[r]^-{\sim}\ar[d]^{\op{id}} & \Gamma_{\dR}(\Sigma, (\widehat{\fr g}_{\cal D}^{(\kappa, E)})_{\cal P_G})\big|_{(s,\cal P_{G,\Sigma},\alpha)} \ar[r]^-{\eqref{eq-equivariant-submodule}} & \Gamma_{\dR}(\overset{\circ}{D}_{\underline {\tilde x}}, \widehat{\fr g}_{\cal D}^{(\kappa, E)}) \ar[d]^{g\cdot} \\
	\Gamma_{\dR}(\Sigma, (\widehat{\fr g}_{\cal D}^{(\kappa,E)})_{\cal P_{G,\Sigma}}) \ar[r]^-{\sim} & \Gamma_{\dR}(\Sigma, (\widehat{\fr g}_{\cal D}^{(\kappa, E)})_{\cal P_G})\big|_{(s,\cal P_{G,\Sigma},g\cdot\alpha)} \ar[r]^-{\eqref{eq-equivariant-submodule}} & \Gamma_{\dR}(\overset{\circ}{D}_{\underline {\tilde x}}, \widehat{\fr g}_{\cal D}^{(\kappa, E)}).
}
\end{equation}
Here, the two horizontal compositions express the procedure of
\begin{itemize}
	\item first restricting a flat section of $(\widehat{\fr g}_{\cal D}^{(\kappa,E)})_{\cal P_{G,\Sigma}}$ to $\overset{\circ}{D}_{\underline {\tilde x}} \hookrightarrow T\times\Sigma$;
	\item then using the trivialization $\alpha$ (respectively, $g\cdot\alpha$) to identify it with a section of $\widehat{\fr g}_{\cal D}^{(\kappa, E)}$.
\end{itemize}
However, the following diagram is tautologically commutative:
$$
\xymatrix@R=1.5em{
	\Gamma_{\dR}(\overset{\circ}{D}_{\underline {\tilde x}}, (\widehat{\fr g}_{\cal D}^{(\kappa, E)})_{\cal P_{G,\Sigma}}) \ar[d]^{\op{id}}\ar[r]^-{\alpha_*} & \Gamma_{\dR}(\overset{\circ}{D}_{\underline {\tilde x}}, \widehat{\fr g}_{\cal D}^{(\kappa, E)}) \ar[d]^{g\cdot} \\
	\Gamma_{\dR}(\overset{\circ}{D}_{\underline {\tilde x}}, (\widehat{\fr g}_{\cal D}^{(\kappa, E)})_{\cal P_{G,\Sigma}}) \ar[r]^-{(g\cdot\alpha)_*} & \Gamma_{\dR}(\overset{\circ}{D}_{\underline {\tilde x}}, \widehat{\fr g}_{\cal D}^{(\kappa, E)}),
}
$$
so we obtain the commutativity of \eqref{eq-equivariant-submodule-diagram}. \qed(Lemma \ref{lem-liealgd-ideal})

\subsection{Descent to $\op{Bun}_G$} We continue to fix the $S$-point $(\fr g^{\kappa}, E)$ of $\op{Par}_G$. The goal of this section is to ``descend'' the classical quasi-twisting $\tilde{\cal T}_G^{(\kappa, E)}$ to $\op{Bun}_G$. Recall the action of $H:=S\times\cal L_x^+G$ on $\tilde{\cal S}=S\times\op{Bun}_{G,\infty x}$, whose quotient is given by $\tilde{\cal S}/H\xrightarrow{\sim} S\times\op{Bun}_G$. Let $\fr k:=\fr g^{\kappa}(\cal O_x)$. Then $(\fr k, H)$ forms a classical action pair (see \S\ref{sec-action-pair}).

\subsubsection{} We now equip \eqref{eq-qtw-full-level} with the structure of a $(\fr k, H)$-action. Indeed, applying the functor $\Gamma(D_{\underline {\tilde x}},-)$ to \eqref{eq-kac-moody-extension-twist} and using $\Gamma_{\dR}(D_{\underline {\tilde x}}, \omega_{\tilde{\cal X}/\tilde{\cal S}})=0$, we obtain a commutative diagram:
\begin{equation}
\label{eq-arc-splitting}
\xymatrix@C=1.5em@R=1.5em{
	& \Gamma_{\dR}(D_{\underline {\tilde x}}, \widehat{\fr g}_{\cal D}^{(\kappa,E)}) \ar[r]^-{\sim}\ar[d] & \Gamma(D_{\underline {\tilde x}}, \fr g^{\kappa}\boxtimes\cal O_{\op{Bun}_{G,\infty x}\times X}) \ar[d]^{\eta} \ar@{.>}[dl]_{\widehat{\eta}} \\
	\Gamma_{\dR}(\overset{\circ}{D}_{\underline{\tilde x}}, \omega_{\tilde{\cal X}/\tilde{\cal S}}) \ar[r] & \Gamma_{\dR}(\overset{\circ}{D}_{\underline{\tilde x}}, \widehat{\fr g}_{\cal D}^{(\kappa,E)}) \ar[r] & \fr g^{\kappa}(\cal K_x)\widehat{\boxtimes}\cal O_{\op{Bun}_{G,\infty x}}
}
\end{equation}
where the splitting $\widehat{\eta}$ exists for obvious reasons. Since $\Gamma(D_{\underline {\tilde x}}, \fr g^{\kappa}\boxtimes\cal O_{\op{Bun}_{G,\infty x}\times X})$ is canonically isomorphic to $\fr k\widehat{\otimes}\cal O_{\tilde{\cal S}}$, we obtain the $(\fr k, H)$-action datum on $\widehat{\cal L}^{(\kappa, E)}$ via the composition:
$$
\fr k\widehat{\otimes}\cal O_{\tilde{\cal S}} \xrightarrow{\widehat{\eta}} \Gamma_{\dR}(\overset{\circ}{D}_{\underline {\tilde x}}, \widehat{\fr g}_{\cal D}^{(\kappa,E)}) \rightarrow \widehat{\cal L}^{(\kappa, E)},
$$
which we again denote by $\widehat{\eta}$.

\begin{rem}
Ideally, we would like to define $\cal T^{(\kappa,E)}$ as the quotient $\mathbf Q^{(\fr k, H)}(\tilde{\cal T}^{(\kappa,E)})$. However, we run into problems because $\tilde{\cal S}$ is not locally of finite type (so we cannot use $\mathbf Q^{(H,H^{\flat})}$ \eqref{eq-quotient-geom}), and $\widehat{\eta}$ is not injective (so we cannot use $\mathbf Q_{\op{inj}}^{(\fr k, H)}$ \eqref{eq-quotient-inj}). In what follows, we circumvent this technical problem using a combination of the two functors.
\end{rem}

\subsubsection{} For each integer $n\ge 0$, let $\op{Bun}_{G,nx}$ denote the stack classifying pairs $(\cal P_G, \alpha_n)$ where $\cal P_G$ is a $G$-bundle on $X$ and $\alpha_n : \cal P_G\big|_{\op{Spec}(\cal O_x^{(n)})} \xrightarrow{\sim} \cal P_G^0$ is a trivialization over the $n$th infinitesimal neighborhood $\op{Spec}(\cal O_x^{(n)})$ of $x$. Then $\op{Bun}_{G,nx}$ is an $\cal L_{nx}G$-torsor over $\op{Bun}_G$, where $\cal L_{nx}G$ classifies maps from $\op{Spec}(\cal O_x^{(n)})$ to $G$.

\begin{rem}
In particular, $\cal L_{nx}G$ is a group scheme of finite type.
\end{rem}

Set $H_n:=S\times\cal L_{nx}G$, and we have an exact sequence of group schemes over $S$:
$$
1 \rightarrow H^n \rightarrow H\rightarrow H_n \rightarrow 1.
$$
Define $\fr k^n:=\cal K\otimes\fr m_x^n$, and $\fr k_n:=\fr k/\fr k^n\cong\cal K\otimes\cal O_x^{(n)}$. Then the above sequence extends to an exact sequence of action pairs (see \S\ref{sec-normal-subpair}):
\begin{equation}
\label{eq-relevant-action-pair-sequence}
1 \rightarrow (\fr k^n, H^n) \rightarrow (H,\fr k) \rightarrow (H_n, \fr k_n) \rightarrow 1.
\end{equation}

\subsubsection{} We briefly review the Harder-Narasimhan truncation of $\op{Bun}_G$. For this, we need to fix a Borel $B\hookrightarrow G$, whose quotient torus is denoted by $T$. There are canonical maps
$$
\xymatrix@C=1em@R=1em{
 & \op{Bun}_B \ar[dl]_{\fr p}\ar[dr]^{\fr q} & \\
\op{Bun}_G & & \op{Bun}_T.
}
$$

Let $\Lambda_G$ denote the coweight lattice of $G$, and $\Lambda_G^+, \Lambda^{\op{pos}}_G\subset\Lambda_G$ denote the submonoid of dominant coweights, respectively the submonoid generated by positive simple coroots. Denote by $\Lambda_G^{+,\mathbb Q}$ and $\Lambda_G^{\op{pos}, \mathbb Q}$ the corresponding rational cones.

There is a partial ordering on $\Lambda_G^{\mathbb Q}$, given by:
$$
\lambda_1 \underset{G}{\leq}\lambda_2 \iff \lambda_2-\lambda_1\in\Lambda_G^{\op{pos},\mathbb Q}.
$$
Given $\lambda\in\Lambda_G^{\mathbb Q}$, define $\op{Bun}_B^{\lambda}$ as the pre-image of $\lambda$ under the composition:
$$
\op{Bun}_B\xrightarrow{\fr q}\op{Bun}_T \xrightarrow{\deg} \Lambda_T^{\mathbb Q}\cong\Lambda_G^{\mathbb Q}.
$$
For each $\theta\in\Lambda_G^{+,\mathbb Q}$, define $\op{Bun}_G^{(\le\theta)}$ as the substack of $\op{Bun}_G$ classifying $G$-bundles $\cal P_G$ with the following property:
\begin{itemize}
	\item for each $B$-bundle $\cal P_B\in\op{Bun}_B^{\lambda}$ with $\fr p(\cal P_B)\cong\cal P_G$, we have $\lambda\underset{G}{\le}\theta$.
\end{itemize}

The following result is proved in \cite{DG11}:
\begin{lem}
$\op{Bun}_G^{(\le\theta)}$ is an open, quasi-compact substack of $\op{Bun}_G$. \qed
\end{lem}

\begin{rem}
The definition of $\op{Bun}_G^{(\le\theta)}$ in \cite{DG11} refers to all standard parabolics $P$ of $G$, rather than just the Borel. However, the two definitions are equivalent; see the discussion in \S7.3.3 in \emph{loc.cit}.
\end{rem}

\subsubsection{} For each integer $n\ge 0$ (as well as $n=\infty$), we let $\op{Bun}_{G,nx}^{(\le\theta)}$ denote the preimage of $\op{Bun}_G^{(\le\theta)}$ under the canonical map $\op{Bun}_{G,nx} \rightarrow \op{Bun}_G$. We denote the universal $G$-bundle over $\op{Bun}_G^{(\le\theta)}\times X$ by $\cal P_G$, and that over $\op{Bun}_{G,\infty x}^{(\le\theta)}\times X$ by $\tilde{\cal P}_G$; their pullbacks to $S\times\op{Bun}_G^{(\le\theta)}\times X$ and $S\times\op{Bun}_{G,\infty x}^{(\le\theta)}\times X$ are denoted by the same characters.

The key technical assertion we need is:

\begin{prop}
\label{prop-lattice-property}
For each $\theta\in\Lambda_G^{+,\mathbb Q}$, there exists an integer $N(\theta)$ such that whenever $n\ge N(\theta)$, we have
$$
(\fr g^{\kappa}(\fr m_x^n)\widehat{\boxtimes}\cal O_{\op{Bun}^{(\le\theta)}_{G,\infty x}}) \cap \Gamma(\Sigma, \fr g^{\kappa}_{\tilde{\cal P}_G}) = 0
$$
as submodules of $\fr g^{\kappa}(\cal K_x)\widehat{\boxtimes}\cal O_{\op{Bun}_{G,\infty x}^{(\le\theta)}}$ (via $\eta$ and $\gamma$).
\end{prop}
\begin{proof}
Fix $\theta\in\Lambda_G^{+,\mathbb Q}$. For each integer $n\ge 0$, we have an isomorphism:
$$
(\fr g^{\kappa}(\fr m_x^n)\widehat{\boxtimes}\cal O_{\op{Bun}^{(\le\theta)}_{G,\infty x}}) \cap \Gamma(\Sigma, \fr g^{\kappa}_{\tilde{\cal P}_G}) \xrightarrow{\sim} \op R^0(\op{pr}_{\infty x})_*\fr g^{\kappa}_{\tilde{\cal P}_G}(-nx),
$$
where $\op{pr}_{\infty x}$ is the projection map in the following Cartesian diagram:
$$
\xymatrix@C=1.5em@R=1.5em{
	S\times\op{Bun}_{G,\infty x}^{(\le\theta)} \times X \ar[r]\ar[d]^{\op{pr}_{\infty x}} & S\times\op{Bun}_G^{(\le\theta)}\times X \ar[d]^{\op{pr}} \\
	S\times\op{Bun}_{G,\infty x}^{(\le\theta)} \ar[r] & S\times\op{Bun}_G^{(\le\theta)}.
}
$$

Since $\tilde{\cal P}_G$ is the pullback of the universal $G$-bundle ${\cal P}_G$ over $S\times\op{Bun}_{G}^{(\le\theta)}\times X$, it suffices to show that $\op R^0(\op{pr})_*\fr g_{{\cal P}_G}^{\kappa}(-nx)$ vanishes for sufficiently large $n$ (relative to $\theta$).\footnote{Identification of $\op R^0(\op{pr}_{\infty x})_*\fr g_{\cal P_G}^{\kappa}(-nx)$ with the pullback of $\op R^0(\op{pr})_*\fr g_{\cal P_G}^{\kappa}(-nx)$ follows from flatness of the projection $S\times\op{Bun}_{G,\infty x}^{(\le\theta)}\rightarrow S\times\op{Bun}_G^{(\le\theta)}$.} We shall choose $n$ such that $\op H^0(X, \fr g_{\underline{\cal P}_G}(-nx))$ vanishes for all $\underline{\cal P}_G\in\op{Bun}_G^{(\le\theta)}$.

\begin{claim}
For such $n$, $\op R^0(\op{pr})_*\fr g_{\cal P_G}^{\kappa}(-nx)$ is a vector bundle.
\end{claim}
\noindent
Indeed, representing $\op{pr}_*\fr g_{\cal P_G}^{\kappa}(-nx)$ by a two-term complex of vector bundles, it suffices to show that $\op R^1(\op{pr})_*\fr g_{\cal P_G}^{\kappa}(-nx)$ is flat. However, its fiber at a $k$-point $(\kappa,\underline{\cal P}_G)$ is given by:
	$$
	\op H^1(\op L\iota_{(\kappa,\underline{\cal P}_G)}^*\circ\op R\op{pr}_*\fr g_{\cal P_G}^{\kappa}(-nx)) \cong \op H^1(X, \fr g_{\underline{\cal P}_G}(-nx)),
	$$
	since $\kappa\cong\fr g$ as a $G$-representation (Corollary \ref{cor-subspace-representation}). The Riemann-Roch theorem shows:
	$$
	\dim\op H^1(X,\fr g_{\underline{\cal P}_G}(-nx)) = -\deg(\fr g_{\underline{\cal P}_G}(-nx)) - \dim(\fr g)\cdot(1-g) = \dim(\fr g)(n+g-1),
	$$
	which is constant as the $k$-point $(\kappa,\underline{\cal P}_G)$ varies.

Now that $\op R^0(\op{pr})_*\fr g_{\cal P_G}^{\kappa}(-nx)$ is a vector bundle, its fiber at any $k$-point $(\kappa, \underline{\cal P}_G)$ can be computed as follows:
$$
\op R^0(\op{pr})_*\fr g_{{\cal P}_G}^{\kappa}(-nx) \big|_{(\kappa,\underline{\cal P}_G)} \xrightarrow{\sim} \op H^0(X, \fr g_{\underline{\cal P}_G}(-nx)) \cong 0.
$$
This establishes the required vanishing.
\end{proof}

It follows from Proposition \ref{prop-lattice-property} that the $(\fr k, H)$-algebroid $\cal L^{\kappa}$ (hence also $\cal L^{(\kappa, E)}$) is an object of $\mathbf{LieAlgd}^{(\fr k^n ,H^n)}_{\op{inj}}(S\times\op{Bun}_G^{(\le\theta)}/S)$ whenever $n\ge N(\theta)$.

\subsubsection{} For each $\theta\in\Lambda_G^{+,\mathbb Q}$, denote by $\tilde{\cal T}^{(\le\theta)}_G$ the restriction of the classical quasi-twisting $\tilde{\cal T}_G^{(\kappa, E)}$ to $S\times\op{Bun}_{G,\infty x}^{(\le\theta)}$.\footnote{We temporarily suppress the notational dependence on the parameter $(\fr g^{\kappa}, E)$.} Given $n\ge N(\theta)$, we can define a quasi-twisting over $S\times\op{Bun}_G^{(\le\theta)}$ by the formula:
\begin{equation}
\label{eq-descent-formula}
\cal T^{(\le\theta)}_{G, n} := \mathbf Q^{(H_n, H_n^{\flat})} \circ \mathbf Q^{(\fr k^n, H^n)}_{\op{inj}}(\tilde{\cal T}^{(\le\theta)}_G),
\end{equation}
where $H_n^{\flat}$ denotes the quotient $H_n/\exp(\fr k_n)$ (see \S\ref{sec-classical-to-geom-action-pair}).

\begin{rem}
Note that $\mathbf Q^{(\fr k^n, H^n)}_{\op{inj}}(\cal T^{(\le\theta)}_{G})$ is well-defined as a classical quasi-twisting over $S\times\op{Bun}_{G,nx}^{(\le\theta)}$, equipped with a $(\fr k_n, H_n)$-action. Since the stack $S\times\op{Bun}_{G,nx}^{(\le\theta)}$ is locally of finite type, any classical quasi-twisting gives rise to a quasi-twisting, and the $(\fr k_n,H_n)$-action induces an $(H_n, H_n^{\flat})$-action (see \S\ref{sec-classical-to-geom-action}). Hence the formula \eqref{eq-descent-formula} makes sense.
\end{rem}

\subsubsection{} Suppose $n_1\ge n_2\ge N(\theta)$. We would like to construct a canonical isomorphism of quasi-twistings
\begin{equation}
\label{eq-descent-compare-n}
\cal T_{G, n_1}^{(\le\theta)} \xrightarrow{\sim} \cal T_{G, n_2}^{(\le\theta)}.
\end{equation}
Indeed, let $(\fr k', H')$ be the kernel of the map $(\fr k_{n_1}, H_{n_1}) \rightarrow (\fr k_{n_2}, H_{n_2})$. In particular, $H'$ is of finite type. Furthermore, we have an exact sequence of classical action pairs:
$$
1 \rightarrow (\fr k^{n_1}, H^{n_1}) \rightarrow (\fr k^{n_2}, H^{n_2}) \rightarrow (\fr k', H') \rightarrow 1.
$$
Hence, there are isomorphisms:
\begin{align*}
\cal T_{G, n_1}^{(\le\theta)} \xrightarrow{\sim} & \mathbf Q^{(H_{n_2}, H_{n_2}^{\flat})}\circ \mathbf Q^{(H', (H')^{\flat})}\circ \mathbf Q^{(\fr k^{n_1}, H^{n_1})}_{\op{inj}}(\tilde{\cal T}_{G}^{(\le\theta)}) \\
& \xrightarrow{\sim} \mathbf Q^{(H_{n_2}, H_{n_2}^{\flat})}\circ \mathbf Q^{(\fr k',H')}_{\op{inj}}\circ \mathbf Q^{(\fr k^{n_1}, H^{n_1})}_{\op{inj}}(\tilde{\cal T}_{G}^{(\le\theta)}) \xrightarrow{\sim}  \cal T_{G, n_2}^{(\le\theta)},
\end{align*}
using Propositions \ref{prop-normal-subpair-geom}, \ref{prop-quotient-compare}, and \ref{prop-normal-subpair-cl}. In light of the isomorphism \eqref{eq-descent-compare-n}, we may let $\cal T^{(\le\theta)}_G$ denote the quasi-twisting $\cal T_{G, n}^{(\le\theta)}$ over $S\times\op{Bun}_G^{(\le\theta)}$ for any $n\ge N(\theta)$.

\subsubsection{} Finally, we check that the quasi-twistings $\cal T^{(\le\theta)}_G$ glue along various Harder-Narasimhan truncations. Indeed, suppose $\theta_1,\theta_2\in\Lambda_G^{+,\mathbb Q}$. Then we have isomorphisms:
\begin{align*}
\cal T_{G, n}^{(\le\theta_1)}\big|_{S\times(\op{Bun}_G^{(\le\theta_1)}\cap\op{Bun}_G^{(\le\theta_2)})} \xrightarrow{\sim} & \mathbf Q^{(H_n, (H_n)^{\flat})}\circ\mathbf Q_{\op{inj}}^{(\fr k^n, H^n)}(\cal T_{\infty x}\big|_{S\times(\op{Bun}_{G,\infty x}^{(\le\theta_1)}\cap\op{Bun}_{G,\infty x}^{(\le\theta_2)})}) \\
& \xrightarrow{\sim} \cal T_{G, n}^{(\le\theta_2)}\big|_{S\times(\op{Bun}_G^{(\le\theta_1)}\cap\op{Bun}_G^{(\le\theta_2)})},
\end{align*}
whenever $n\ge N(\theta_1), N(\theta_2)$. Therefore we obtain a quasi-twisting $\cal T^{(\kappa,E)}_G$ on $S\times\op{Bun}_G$ (relative to $S$) whose restriction to each $S\times\op{Bun}_G^{(\le\theta)}$ agrees with $\cal T^{(\le\theta)}_G$.

\begin{nota}
We write $\cal T_G^{(\kappa, E)}=\mathbf Q^{(\fr g^{\kappa}(\cal O_x), \cal L_x^+G)}(\tilde{\cal T}_G^{(\kappa,E)})$, although it is tacitly understood that the construction of $\cal T_G^{(\kappa, E)}$ requires two quotient steps and gluing. In a similar way, we write:
\begin{equation}
\label{eq-qtw-at-level-n}
\cal T_{G,n}^{(\kappa, E)} := \mathbf Q^{(\fr g^{\kappa}(\fr m_x^{(n)}), H^n)}(\tilde{\cal T}_G^{(\kappa, E)}),
\end{equation}
for the corresponding quasi-twisting on $S\times\op{Bun}_{G,nx}$.

Since the construction of $\cal T_G^{(\kappa,E)}$ (resp.~$\cal T_{G,n}^{(\kappa, E)}$) is functorial in $S$, we obtain a \emph{universal} quasi-twisting $\cal T_G^{\op{univ}}$ over $\op{Par}_G\times\op{Bun}_G$ (resp.~$\cal T_{G,n}^{\op{univ}}$ over $\op{Par}_G\times\op{Bun}_{G,nx}$.)
\end{nota}

\begin{rem}
Note that the DG category $\cal T^{(\kappa, E)}_G\Mod$ is naturally a $\op{QCoh}(S)$-module. Again from the functoriality in maps $(\fr g^{\kappa}, E): S\rightarrow \op{Par}_G$, we obtain a sheaf of DG categories over $\op{Par}_G$, denoted by $\cal T_G^{\op{univ}}\Mod$.

The \emph{na\"ive} version of the quantum Langlands duality claims an equivalence of sheaves of DG categories:
\begin{equation}
\label{eq-naive-quantum-langlands}
\cal T^{\op{univ}}_G\Mod \xrightarrow{\sim} \cal T^{\op{univ}}_{\check G}\Mod
\end{equation}
over the common base $\op{Par}_G\xrightarrow{\sim}\op{Par}_{\check G}$ (by \eqref{eq-duality-for-par}). However, the hypothetical equivalence \eqref{eq-naive-quantum-langlands} is false whenever $G$ is not a torus, and a renormalization procedure is required for stating the correct version of quantum Langlands duality.
\end{rem}

\subsubsection{Recovering the classical TDOs}
\label{sec-identify-det}
Suppose $G$ is simple, and we fix a $k$-valued parameter $(\fr g^{\kappa}, 0)$ of $\op{Par}_G$ corresponding to some bilinear form $\kappa$ on $\fr g$. Let $\lambda$ and $c$ be as in Example \ref{eg-dual-coxeter-number}.

Let $\cal L_{G,\det}$ denote the determinant line bundle over $\op{Bun}_G$, and $\tilde{\cal L}_{G,\det}$ its pullback to $\op{Bun}_{G,\infty x}$.

\begin{prop}
\label{prop-recover-classical-tdo}
The classical quasi-twisting \eqref{eq-qtw-full-level} at the parameter $(\fr g^{\op{Kil}}, 0)$:
$$
0 \rightarrow \cal O_{\op{Bun}_{G,\infty x}} \rightarrow \widehat{\cal L}^{(\op{Kil}, 0)} \rightarrow \cal L^{\op{Kil}} \rightarrow 0
$$
identifies with the Picard algebroid $\op{Diff}^{\le 1}(\tilde{\cal L}_{G,\det})$.
\end{prop}
\begin{proof}
Via the isomorphism $\op{pr}_{\fr g} : \fr g^{\op{Kil}} \xrightarrow{\sim} \fr g$, the lower triangle of \eqref{eq-main-splitting} identifies with:
\begin{equation}
\label{eq-tate-extension}
0 \rightarrow \cal O_{\op{Bun}_{G,\infty x}} \rightarrow \widehat{\fr g}^{\op{Tate}} \widehat{\boxtimes}\cal O_{\op{Bun}_{G,\infty x}} \rightarrow \fr g(\cal K_x) \widehat{\boxtimes}\cal O_{\op{Bun}_{G,\infty x}} \rightarrow 0.
\end{equation}
where $\widehat{\fr g}^{\op{Tate}}$ is the central extension of $\fr g(\cal K_x)$ defined by the cocycle
$$
(\xi\otimes f, \xi'\otimes f') \leadsto \op{Kil}(\xi, \xi')\cdot\op{Res}(df\cdot f').
$$
Recall that \eqref{eq-tate-extension} is a classical quasi-twisting, where the Lie algebroid brackets are induced from the $\cal L_xG$-action on $\op{Bun}_{G,\infty x}$.

It is well known (see, e.g. \cite[\S7, \S10]{So00}) that $\widehat{\fr g}^{\op{Tate}}$ comes from a central extension of group ind-schemes:
$$
1 \rightarrow \mathbb G_m \rightarrow \widehat G^{\op{Tate}} \rightarrow \cal L_xG \rightarrow 1,
$$
and the $\cal L_xG$-action on $\op{Bun}_{G,\infty x}$ extends to an action of $\widehat G^{\op{Tate}}$ on $\widetilde{\cal L}_{G,\det}$. Hence $\widehat{\fr g}^{\op{Tate}}$ acts as derivations on $\widetilde{\cal L}_{G,\det}$, and we obtain a morphism $\widehat{\fr g}^{\op{Tate}}\widehat{\boxtimes}\cal O_{\op{Bun}_{G,\infty x}} \rightarrow \op{Diff}^{\le 1}(\widetilde{\cal L}_{G,\det})$ of Lie algebroids. Note that the following diagram commutes:
$$
\xymatrix@C=1.5em@R=1.5em{
0 \ar[r] & \cal O_{\op{Bun}_{G,\infty x}} \ar[r]\ar[d]^{\rotatebox{90}{$\sim$}} & \widehat{\fr g}^{\op{Tate}} \widehat{\boxtimes}\cal O_{\op{Bun}_{G,\infty x}} \ar[r]\ar[d] & \fr g(\cal K_x) \widehat{\boxtimes}\cal O_{\op{Bun}_{G,\infty x}} \ar[r]\ar[d] & 0. \\
0 \ar[r] & \cal O_{\op{Bun}_{G,\infty x}} \ar[r] & \op{Diff}^{\le 1}(\widetilde{\cal L}_{G,\det}) \ar[r] & \cal T_{\op{Bun}_{G,\infty x}} \ar[r] & 0
}
$$
Furthermore, the $\cal O_{\op{Bun}_{G,\infty x}}$-submodule $\Gamma(\Sigma, \fr g_{\widetilde{\cal P}_G})$ of $\widehat{\fr g}^{\op{Tate}} \widehat{\boxtimes}\cal O_{\op{Bun}_{G,\infty x}}$ acts by zero on $\widetilde{\cal L}_{G,\det}$, so by modding out $\Gamma(\Sigma, \fr g_{\widetilde{\cal P}_G})$, we obtain a morphism of classical quasi-twistings:
$$
\xymatrix@C=1.5em@R=1.5em{
0 \ar[r] & \cal O_{\op{Bun}_{G,\infty x}} \ar[r]\ar[d]^{\rotatebox{90}{$\sim$}} & \widehat{\cal L}^{(\op{Kil},0)} \ar[r]\ar[d] & \cal L^{\op{Kil}} \ar[r]\ar[d] & 0. \\
0 \ar[r] & \cal O_{\op{Bun}_{G,\infty x}} \ar[r] & \op{Diff}^{\le 1}(\widetilde{\cal L}_{G,\det}) \ar[r] & \cal T_{\op{Bun}_{G,\infty x}} \ar[r] & 0
}
$$
where the last terms $\cal L^{\op{Kil}}$ and $\cal T_{\op{Bun}_{G,\infty x}}$ are identified. As such, it is an isomorphism of classical quasi-twistings.
\end{proof}

It follows from Proposition \ref{prop-recover-classical-tdo} that the classical quasi-twisting at $(\fr g^{\kappa}, 0)$ operates on the virtual line bundle $\tilde{\cal L}_{G,\det}^{\lambda}$. Since quotient by the action pair $(\fr g(\cal O_x), \cal L_x^+G)$ agrees with strong quotient of Picard algebroids, we obtain an equivalence
$$
\cal T^{(\kappa,\op{triv})}_G\Mod \xrightarrow{\sim} \op{Diff}(\cal L_{G,\det}^{\lambda})\Mod(\op{Bun}_G).
$$
In particular, the hypothetical equivalence \eqref{eq-naive-quantum-langlands} specializes to \eqref{eq-quantum-langlands}.

\medskip

\section{Recovering $\op{QCoh}(\op{LocSys}_G)$ at $\kappa=\infty$}
\label{sec-infty}

In this section, we show that at level $\infty$, the quasi-twisting $\cal T_G^{(\kappa, E)}$ constructed in \S\ref{sec-construction} recovers the DG algebraic stack $\op{LocSys}_G$ in the following sense: $\cal T_G^{(\infty, 0)}$ is the \emph{inert} quasi-twisting on some triangle $\cal O_{\op{Bun}_G} \rightarrow \widehat{\cal Q}_{\op{desc}}^{(\infty, 0)} \rightarrow \cal Q_{\op{desc}}^{(\infty, 0)}$ in $\op{QCoh}(\op{Bun}_G)$ (see \S\ref{sec-inert-qtw-on-triangle} for what this means). Furthermore, the corresponding stack $\mathbb V(\widehat{\cal Q}_{\op{desc}}^{(\infty,0)})_{\lambda=1}$ over $\op{Bun}_G$ identifies with $\op{LocSys}_G$, so we obtain an equivalence of DG categories $\cal T_G^{(\infty, 0)}\Mod\xrightarrow{\sim}\op{QCoh}(\op{LocSys}_G)$.

Finally, we comment on the role of certain additional parameters $E$ when $\fr g^{\kappa}=\fr g^{\infty}$.

\subsection{The underlying $\cal O_{S\times\op{Bun}_G}$-modules of $\cal T_{G,n}^{(\kappa,0)}$}
\label{sec-underlying-o-modules}
We adopt the following notations from the previous section: let $\cal S_n:=S\times\op{Bun}_{G,nx}$, and $\cal X_n:=S\times\op{Bun}_{G,nx}\times X$ which is a curve over $\cal S_n$. The tautological $G$-bundle over $\cal X_n$ is denoted by $\cal P_G^{(n)}$. Write $\widetilde S:=S\times\op{Bun}_{G,\infty x}$ and similarly for $\widetilde{\cal X}$ and $\widetilde{\cal P}_G$.

Recall the quasi-twisting $\cal T_{G,n}^{(\kappa,0)}$ and $\cal T_G^{(\kappa,0)}=\cal T_{G,0}^{(\kappa,0)}$ which are special cases of \eqref{eq-qtw-at-level-n} for the $S$-valued parameter $(\fr g^{\kappa}, 0)$. Suppose $\cal T_{G,n}^{(\kappa,0)}$ is expressed as a map of some formal moduli problems $\widehat{\cal S}_n^{\flat}\rightarrow\cal S_n^{\flat}$ under $\cal S_n$.

\subsubsection{} Since $\cal T_{G,n}^{(\kappa,0)}$ is the quotient of $\tilde{\cal T}_G^{(\kappa,0)}$ by the pair $(\fr g^{\kappa}(\fr m_x^n), H^n)$, the underlying ind-coherent sheaves of $\widehat{\cal S}_n^{\flat}$ and $\cal S_n^{\flat}$ arise from a triangle in $\op{QCoh}(\cal S_n)$:
\begin{equation}
\label{eq-underlying-o-modules}
\cal O_{\cal S_n} \rightarrow \widehat{\cal Q}^{(\kappa,0)}_{n,\op{desc}} \rightarrow \cal Q^{\kappa}_{n,\op{desc}},
\end{equation}
where $\widehat{\cal Q}_{n,\op{desc}}^{(\kappa,0)}$ is the descent of the $H^n$-equivariant complex of $\cal O_{\tilde{\cal S}}$-modules $\widehat{\cal Q}^{(\kappa,0)}_n:=\op{Cofib}(\fr g^{\kappa}(\fr m_x^n)\boxtimes\cal O_{\op{Bun}_{G,\infty x}} \rightarrow \widehat{\cal L}^{(\kappa, 0)})$, and similarly for $\cal Q_{n,\op{desc}}^{\kappa}$.

\subsubsection{} The Atiyah bundle construction gives rise to a triangle $\omega_{\cal X_n/\cal S_n} \rightarrow \op{At}(\cal P_G^{(n)})^* \rightarrow \fr g_{\cal P_G^{(n)}}^*$ over $\cal X_n$. Its pullback along the projection $\fr g_{\cal P_G^{(n)}}^{\kappa}\rightarrow\fr g^*_{\cal P_G^{(n)}}$ is denoted by:
\begin{equation}
\label{eq-atiyah-at-kappa}
\omega_{\cal X_n/\cal S_n} \rightarrow \cal E^{\kappa}(\cal P_G^{(n)}) \rightarrow \fr g^{\kappa}_{\cal P_G^{(n)}}.
\end{equation}
Note that there is a canonical isomorphism $\cal Q_{n,\op{desc}}^{\kappa} \xrightarrow{\sim} \op R\Gamma(X, \fr g_{\cal P_G^{(n)}}^{\kappa}(-nx))[1]$.

\begin{prop}
\label{prop-underlying-o-modules-identify}
The triangle \eqref{eq-underlying-o-modules} identifies with the push-out of
\begin{equation}
\label{eq-underlying-o-modules-identify}
\op R\Gamma(X, \omega_{\cal X_n/\cal S_n}(-nx))[1] \rightarrow \op R\Gamma(X, \cal E^{\kappa}(\cal P_G^{(n)})(-nx))[1] \rightarrow \op R\Gamma(X, \fr g_{\cal P_G^{(n)}}^{\kappa}(-nx))[1]
\end{equation}
along the trace map $\op R\Gamma(X, \omega_{\cal X_n/\cal S_n}(-nx))[1] \rightarrow \cal O_{\cal S_n}$.
\end{prop}

\subsubsection{}
We now begin the proof of Proposition \ref{prop-underlying-o-modules-identify}. Since both triangles in question are descent of triangles over $\tilde{\cal S}$, we ought to establish an $H^n$-equivariant isomorphism between the triangle:
\begin{equation}
\label{eq-underlying-o-modules-infty}
\cal O_{\tilde{\cal S}}\rightarrow\widehat{\cal Q}^{(\kappa,0)}_n \rightarrow \cal Q_n^{\kappa}
\end{equation}
and the push-out of the analogous triangle:
\begin{equation}
\label{eq-underlying-o-modules-identify-infty}
\op R\Gamma(X, \omega_{\widetilde{\cal X}/\widetilde{\cal S}}(-nx))[1] \rightarrow \op R\Gamma(X, \cal E^{\kappa}(\widetilde{\cal P}_G)(-nx))[1] \rightarrow \op R\Gamma(X, \fr g_{\widetilde{\cal P}_G}^{\kappa}(-nx))[1]
\end{equation}
under the trace map $\op R\Gamma(X, \omega_{\widetilde{\cal X}/\widetilde{\cal S}}(-nx))[1] \rightarrow \cal O_{\widetilde{\cal S}}$.

\subsubsection{} We describe more explicitly the $\cal D_{\tilde{\cal X}/\tilde{\cal S}}$-modules underlying the extension sequence of Lie-$*$ algebras \eqref{eq-kac-moody-extension-twist}:
$$
0 \rightarrow \omega_{\tilde{\cal X}/\tilde{\cal S}} \rightarrow (\widehat{\fr g}_{\cal D}^{(\kappa,0)})_{\tilde{\cal P}_G} \rightarrow (\fr g_{\cal D}^{\kappa})_{\tilde{\cal P}_G} \rightarrow 0,
$$
in the case where the $E=0$. Namely, consider the $\cal D_{\tilde{\cal X}/\tilde{\cal S}}$-modules induced from the sequence \eqref{eq-atiyah-at-kappa} (where we use $\tilde{\cal X}$ instead of $\cal X^{(n)}$ in the Atiyah bundle construction):
$$
0 \rightarrow (\omega_{\tilde{\cal X}/\tilde{\cal S}})_{\cal D} \rightarrow \cal E^{\kappa}(\widetilde{\cal P}_G)_{\cal D} \rightarrow (\fr g_{\cal D}^{\kappa})_{\tilde{\cal P}_G} \rightarrow 0
$$
Let $\cal E^{\kappa}(\widetilde{\cal P}_G)_{\cal D}^{\op{push}}$ be the push-out along $\op{act} : (\omega_{\tilde{\cal X}/\tilde{\cal S}})_{\cal D} \rightarrow \omega_{\tilde{\cal X}/\tilde{\cal S}}$ of the $\cal D_{\tilde{\cal X}/\tilde{\cal S}}$-module $\cal E^{\kappa}(\tilde{\cal P}_G)_{\cal D}$.

\begin{lem}
The $\cal D_{\tilde{\cal X}/\tilde{\cal S}}$-module underlying the extension $(\widehat{\fr g}_{\cal D}^{(\kappa, 0)})_{\tilde{\cal P}_G}$ identifies with $\cal E^{\kappa}(\widetilde{\cal P}_G)_{\cal D}^{\op{push}}$.
\end{lem}
\begin{proof}
Recall that $(\widehat{\fr g}_{\cal D}^{(\kappa, 0)})_{\tilde{\cal P}_G}$ is the $\tilde{\cal P}_G$-twist of the trivial extension $\widehat{\fr g}_{\cal D}^{(\kappa,0)} \xrightarrow{\sim} \omega_{\tilde{\cal X}/\tilde{\cal S}}\oplus\fr g_{\cal D}^{\kappa}$. Consider the push-out diagram:
\begin{equation}
\label{eq-underlying-d-module-pushout}
\xymatrix@C=1.5em@R=1.5em{
(\omega_{\tilde{\cal X}/\tilde{\cal S}})_{\cal D} \ar[r]\ar[d]^{\op{act}} & (\omega_{\tilde{\cal X}/\tilde{\cal S}} \oplus (\fr g^{\kappa}\otimes\cal O_{\tilde{\cal X}}))_{\cal D} \ar[d] \\
\omega_{\tilde{\cal X}/\tilde{\cal S}} \ar[r] & \omega_{\tilde{\cal X}/\tilde{\cal S}}\oplus\fr g_{\cal D}^{\kappa}.
}
\end{equation}
Note that the entire diagram is acted on by the sheaf of groups $\cal G$, as described below:
\begin{itemize}
	\item the $\cal G$-actions on $(\omega_{\tilde{\cal X}/\tilde{\cal S}})_{\cal D}$ and $\omega_{\tilde{\cal X}/\tilde{\cal S}}$ are trivial, and the action on $\omega_{\tilde{\cal X}/\tilde{\cal S}}\oplus\fr g_{\cal D}^{\kappa}$ is given by \S\ref{sec-g-action-on-kac-moody};
	\item the $\cal G$-action on $(\omega_{\tilde{\cal X}/\tilde{\cal S}} \oplus (\fr g^{\kappa}\otimes\cal O_{\tilde{\cal X}}))_{\cal D}$ is the $\cal D_{\tilde{\cal X}/\tilde{\cal S}}$-linear extension of the following $\cal G$-action on $\omega_{\tilde{\cal X}/\tilde{\cal S}} \oplus (\fr g^{\kappa}\otimes\cal O_{\tilde{\cal X}})$ centralizing $\omega_{\tilde{\cal X}/\tilde{\cal S}}$:
	\begin{equation}
	\label{eq-twist-to-atiyah}
	g_{\cal U} \cdot (\xi\oplus\varphi) = \varphi(g_{\cal U}^{-1}dg_{\cal U}) + (\op{Ad}_{g_{\cal U}}(\xi) \oplus \op{Coad}_{g_{\cal U}}(\varphi))
	\end{equation}
	where $g_{\cal U}\in\cal G(\cal U)$ and $\xi\oplus\varphi \in \fr g^{\kappa}\otimes\cal O_{\cal U}$.
\end{itemize}
If we twist the trivial $\cal O_{\tilde{\cal X}}$-module extension equipped with the $\cal G$-action \eqref{eq-twist-to-atiyah}:
$$
0 \rightarrow \omega_{\tilde{\cal X}/\tilde{\cal S}} \rightarrow \omega_{\tilde{\cal X}/\tilde{\cal S}} \oplus (\fr g^{\kappa}\otimes\cal O_{\tilde{\cal X}}) \rightarrow \fr g^{\kappa}\otimes\cal O_{\tilde{\cal X}} \rightarrow 0
$$
by the $G$-bundle $\tilde{\cal P}_G$, we obtain precisely the Atiyah sequence (pulled back along $\fr g_{\tilde{\cal P}_G}^{\kappa}\rightarrow \fr g^*_{\tilde{\cal P}_G}$):
$$
0 \rightarrow \omega_{\tilde{\cal X}/\tilde{\cal S}} \rightarrow \cal E^{\kappa}(\tilde{\cal P}_G) \rightarrow \fr g^{\kappa}_{\tilde{\cal P}_G} \rightarrow 0.
$$
Therefore, twisting the diagram \eqref{eq-underlying-d-module-pushout} by $\tilde{\cal P}_G$, we obtain a push-out diagram:
$$
\xymatrix@C=1.5em@R=1.5em{
(\omega_{\tilde{\cal X}/\tilde{\cal S}})_{\cal D} \ar[r]\ar[d]^{\op{act}} & \cal E^{\kappa}(\tilde{\cal P}_G)_{\cal D} \ar[d] \\
\omega_{\tilde{\cal X}/\tilde{\cal S}} \ar[r] & (\widehat{\fr g}_{\cal D}^{(\kappa,0)})_{\tilde{\cal P}_G}.
}
$$
This proves the Lemma.
\end{proof}

\subsubsection{} By construction of $\widehat{\cal Q}_n^{(\kappa,0)}$ and $\cal Q_n^{\kappa}$, the required isomorphism shall follow from a general claim. We first explain the set-up (which is quite involved): let $\cal S$ be a scheme, and $\cal X:=X\times\cal S$ with section $\underline x$ given by the closed point $x\in X$. Suppose we have an exact sequence of $\cal O_{\cal X}$-modules:
$$
0 \rightarrow \omega_{\cal X/\cal S} \rightarrow \cal E \rightarrow\cal F \rightarrow 0.
$$
Let $\cal E_{\cal D}$ denote the induced $\cal D$-module of $\cal E$ and $\cal E_{\cal D}^{\op{push}}$ its push-out along $\op{act}:(\omega_{\cal X/\cal S})_{\cal D}\rightarrow\omega_{\cal X/\cal S}$.

Then we may form a map between exact sequences:
$$
\xymatrix@C=1.5em@R=1.5em{
	0 \ar[r] & \Gamma_{\dR}(\Sigma, \omega_{\cal X/\cal S}) \ar[r]\ar[d]^0 & \Gamma_{\dR}(\Sigma, \cal E_{\cal D}^{\op{push}}) \ar[r]\ar[d] & \Gamma(\Sigma, \cal F) \ar[r]\ar[d]^{\gamma}\ar@{.>}[dl]_{\widehat{\gamma}} & 0 \\
	0 \ar[r] & \Gamma_{\dR}(\overset{\circ}{D}_{\underline x}, \omega_{\cal X/\cal S}) \ar[r] & \Gamma_{\dR}(\overset{\circ}{D}_{\underline x}, \cal E_{\cal D}^{\op{push}}) \ar[r] & \Gamma(\overset{\circ}{D}_{\underline x}, \cal F) \ar[r] & 0,
}
$$
as well as a section $\widehat{\gamma}$ from the residue theorem. On the other hand, let $\cal E_{\cal D}^{\op{push}}(\fr m^{(n)})$ denote the $\cal O_{\cal S}$-submodule of $\Gamma_{\dR}(D_{\underline x}, \cal E_{\cal D}^{\op{push}})$ annihilated by the restriction to $D_{\underline x}^{(n)}$; we use the notation $\cal F(\fr m^{(n)})$ for a similar meaning. We have a triangle:
\begin{equation}
\label{eq-underlying-o-modules-abstract}
\cal O_{\cal S} \rightarrow \widehat{\cal Q}  \rightarrow \cal Q
\end{equation}
where:
\begin{itemize}
	\item $\widehat{\cal Q}:=\op{Cofib}(\Gamma(\Sigma,\cal F)\rightarrow\Gamma_{\dR}(\overset{\circ}{D}_{\underline x}, \cal E_{\cal D}^{\op{push}})/\cal E_{\cal D}^{\op{push}}(\fr m^{(n)}))$;
	\item $\cal Q:=\op{Cofib}(\Gamma(\Sigma, \cal F)\rightarrow \Gamma(\overset{\circ}{D}_{\underline x}, \cal F)/\cal F(\fr m^{(n)}))$.
\end{itemize}

\begin{rem}
For $\cal S:=\tilde{\cal S}$, $\cal E:=\cal E^{\kappa}(\tilde{\cal P}_G)$, and $\cal F:=\fr g_{\tilde{\cal P}_G}^{\kappa}$, we see from the construction of \eqref{eq-underlying-o-modules-infty} that it identifies with the triangle \eqref{eq-underlying-o-modules-abstract}.
\end{rem}

\begin{claim}
The triangle \eqref{eq-underlying-o-modules-abstract} identifies with the push-out of the canonical triangle:
\begin{equation}
\label{eq-underlying-o-modules-identify-abstract}
\op R\Gamma(X, \omega_{\cal X/\cal S}(-nx))[1] \rightarrow \op R\Gamma(X, \cal E(-nx))[1] \rightarrow \op R\Gamma(X, \cal F(-nx))[1]
\end{equation}
along the trace map $\op R\Gamma(X, \omega_{\cal X/\cal S}(-nx))[1]\rightarrow\cal O_{\cal S}$.
\end{claim}
\begin{proof}
Recall the identification:
$$
\cal Q = \op{Cofib}(\Gamma(\Sigma, \cal F)\rightarrow \Gamma(\overset{\circ}{D}_{\underline x}, \cal F)/\cal F(\fr m^{(n)})) \xrightarrow{\sim}  \op R\Gamma(X, \cal F(-nx))[1],
$$
which is also valid when $\cal F$ is replaced by any $\cal O_{\cal X}$-module. It suffices to produce a morphism of triangles from \eqref{eq-underlying-o-modules-identify-abstract} to \eqref{eq-underlying-o-modules-abstract}, whose first and third terms are the trace map, respectively the above isomorphism.

Consider the diagram defining $\cal E_{\cal D}^{\op{push}}$:
$$
\xymatrix@C=1.5em@R=1.5em{
0 \ar[r] & (\omega_{\cal X/\cal S})_{\cal D} \ar[r]\ar[d] & \cal E_{\cal D} \ar[r]\ar[d] & \cal F_{\cal D} \ar[d]^{\rotatebox{90}{$\sim$}} \ar[r] & 0 \\
0 \ar[r] & \omega_{\cal X/\cal S} \ar[r] & \cal E_{\cal D}^{\op{push}} \ar[r] & \cal F_{\cal D} \ar[r] & 0.
}
$$
Using the functors $\Gamma_{\dR}(\overset{\circ}{D}_{\underline x},-)$ and $\cal M\leadsto \cal M(\fr m^{(n)})$, we obtain a diagram:
$$
\xymatrix@C=1.5em@R=1.5em{
	0 \ar[r] & \overset{\circ}{\omega}/\omega(\fr m^{(n)})\ar[d]^{\op{res}} \ar[r] & \Gamma(\overset{\circ}{D}_{\underline x}, \cal E)/\cal E(\fr m^{(n)}) \ar[r]\ar[d] & \Gamma(\overset{\circ}{D}_{\underline x}, \cal F)/\cal F(\fr m^{(n)}) \ar[r]\ar[d]^{\rotatebox{90}{$\sim$}} & 0 \\
	0 \ar[r] & \cal O_{\tilde{\cal S}} \ar[r] & \Gamma_{\dR}(\overset{\circ}{D}_{\underline x}, \cal E_{\cal D}^{\op{push}})/\cal E_{\cal D}^{\op{push}}(\fr m^{(n)}) \ar[r] & \Gamma(\overset{\circ}{D}_{\underline x}, \cal F)/\cal F(\fr m^{(n)}) \ar[r] & 0
}
$$
where the rows are still exact sequences by the Snake lemma. We now take cofibers of the map from the triangle $\Gamma(\Sigma,\omega)\rightarrow\Gamma(\Sigma, \cal E)\rightarrow\Gamma(\Sigma,\cal F)$ to the top row, and the cofibers of the map from $0\rightarrow\Gamma(\Sigma, \cal F)\rightarrow\Gamma(\Sigma,\cal F)$ to the bottom row:
$$
\xymatrix@C=1.5em@R=1.5em{
	\op R\Gamma(X, \omega_{\cal X/\cal S}(-nx))[1] \ar[r]\ar[d] & \op R\Gamma(X, \cal E(-nx))[1] \ar[r]\ar[d] & \op R\Gamma(X, \cal F(-nx))[1] \ar[d]^{\rotatebox{90}{$\sim$}} \\
	\cal O_{\cal S} \ar[r] & \widehat{\cal Q}  \ar[r] & \cal Q
}
$$
This is a morphism between triangles. Finally, we make the observation that the residue morphism from $\overset{\circ}{\omega}/\omega(\fr m^{(n)})$ passes to the trace map from $\op R\Gamma(X, \omega_{\cal X/\cal S}(-nx))[1]$.
\end{proof}

We have now constructed an isomorphism from \eqref{eq-underlying-o-modules-infty} to the push-out of \eqref{eq-underlying-o-modules-identify-infty} along the trace map $\op R\Gamma(X, \omega_{\widetilde{\cal X}/\widetilde{\cal S}}(-nx))[1] \rightarrow \cal O_{\widetilde{\cal S}}$. We omit checking that this map is compatible with the $H^n$-equivariance structure. \qed(Proposition \ref{prop-underlying-o-modules-identify})

\begin{rem}
Combined with \S\ref{sec-identify-det}, we have showed that the Picard algebroid $\op{Diff}^{\le 1}(\cal L_{G,\det})$ has as its underlying triangle of $\cal O_{\op{Bun}_G}$-modules constructed explicitly by the following procedure:
\begin{itemize}
	\item Consider the triangle $\op R\Gamma(X, \omega_{\cal X/\cal S})[1] \rightarrow \op R\Gamma(X, \cal E^{\kappa}(\cal P_G))[1] \rightarrow \op R\Gamma(X,\fr g_{\cal P_G}^*)[1]$;
	\item Obtain from the above triangle a push-out along the trace map $\op R\Gamma(X, \omega_{\cal X/\cal S})[1]\rightarrow\cal O_{\cal S}$:
	$$
	\cal O_{\cal S} \rightarrow \cal E \rightarrow\op R\Gamma(X,\fr g_{\cal P_G}^*)[1]
	$$
	\item The extension associated to $\op{Diff}^{\le 1}(\cal L_{G,\det})$ is the pullback of the above triangle along:
$$
\cal T_{\op{Bun}_G} \xrightarrow{\sim} \op R\Gamma(X, \fr g_{\cal P_G})[1] \xrightarrow{\op{Kil}} \op R\Gamma(X, \fr g^*_{\cal P_G})[1].
$$
where the Killing form $\op{Kil}$ is regarded as a $G$-invariant isomorphism $\fr g\xrightarrow{\sim}\fr g^*$.
\end{itemize}
\end{rem}

\subsection{An alternative description of $\op{LocSys}_G$} Recall that $\op{LocSys}_G$ is defined as the mapping stack $\underline{\op{Maps}}(X_{\dR}, \op BG)$; it is represented by a DG algebraic stack (\cite[\S10]{AG15}). We give an alternative description of $\op{LocSys}_G$ in terms of ``$G$-bundles with connections.'' This latter description is the form in which $\op{LocSys}_G$ will appear at level $\infty$.

\subsubsection{} Let $\op{LocSys}_G'$ denote the prestack over $\op{Bun}_G$ such that for every affine DG scheme $S$, the groupoid $\op{Maps}(S, \op{LocSys}_G')$ classifies:
\begin{itemize}
	\item a $G$-bundle $\cal P_G$ over $S\times X$;
	\item a splitting of the canonical triangle:
	\begin{equation}
	\label{eq-atiyah-sequence}
	\fr g_{\cal P_G} \rightarrow \op{At}(\cal P_G) \rightarrow \cal T_{S\times X/S}.
	\end{equation}
\end{itemize}
It is not hard to see that $\op{LocSys}_G'$ is represented by a DG algebraic stack.

\subsubsection{} Note that a lift of $\cal P_G$ to an $S$-point of $\op{LocSys}_G$ supplies the dotted arrow in the following commutative diagram:
$$
\xymatrix@R=1.5em@C=1.5em{
S\times X \ar[r]^{\cal P_G}\ar[d] & S\times\op BG \ar[d] \\
S\times X_{\dR} \ar[r]\ar@{.>}[ur] & S
}
$$
This arrow gives rise to a splitting of \eqref{eq-atiyah-sequence} because $\cal T_{S\times X/S\times X_{\dR}}$ is isomorphic to $\cal T_{S\times X/S}$. In other words, we have a morphism of stacks over $\op{Bun}_G$:
\begin{equation}
\label{eq-locsys-compare}
\op{LocSys}_G \rightarrow \op{LocSys}_G'.
\end{equation}

\begin{prop}
\label{prop-locsys-compare}
The morphism \eqref{eq-locsys-compare} is an isomorphism.
\end{prop}

\subsubsection{} Recall that any eventually co-connective affine DG scheme $S$ fits into a sequence:
$$
S_0 \rightarrow S_1 \rightarrow \cdots \rightarrow S_n = S
$$
where $S_0$ is a classical affine DG scheme, and each morphism $S_i\rightarrow S_{i+1}$ is a square-zero extension (\cite[\S III.1, Proposition 5.5.3]{GR16}). Furthermore, for any prestack $\cal Y$ admitting deformation theory and a point $y : S\rightarrow\cal Y$, a lift of $y$ along a square-zero extension $S\rightarrow S'$ is governed by maps out of the cotangent complex $\cal T^*_{\cal Y}\big|_y$ (see \cite[\S III, 0.1.5]{GR16}). Hence Proposition \ref{prop-locsys-compare} reduces to the following two lemmas.

\begin{lem}
The morphism \eqref{eq-locsys-compare} is an isomorphism when restricted to classical test (affine) schemes.
\end{lem}
\begin{proof}
This is the classical statement that any connection over a curve is automatically flat.
\end{proof}

\begin{lem}
\label{lem-identify-cotangent-spaces}
The morphism \eqref{eq-locsys-compare} identifies the relative cotangent complexes $\cal T_{\op{LocSys}_G/\op{Bun}_G}^*$ and $\cal T_{\op{LocSys}_G'/\op{Bun}_G}^*$ at any $S$-point.
\end{lem}
\begin{proof}
Given an $S$-point of $\op{LocSys}_G$ represented by the $G$-bundle $\cal P_G$ over $S\times X_{\dR}$, the cotangent complex $\cal T^*_{\op{LocSys}_G/\op{Bun}_G}\big|_S$, regarded as a functor $\op{QCoh}(S)^{\le 0}\rightarrow\infty\op{-}\mathbf{Gpd}$, is given by:\footnote{We implicitly use the Dold-Kan equivalence between connective complexes of abelian groups and group objects in $\infty\op{-}\mathbf{Gpd}$.}
\begin{align*}
\cal M \leadsto \tau^{\le 0}\op{Fib}(\op R\Gamma(S\times X_{\dR}, \cal M\otimes\fr g_{\cal P_G}[1]) \rightarrow & \op R\Gamma(S\times X, \cal M\otimes\fr g_{\cal P_G}[1])) \\
& \xrightarrow{\sim} \tau^{\le 0}\op R\Gamma(S\times X, \cal M\otimes\fr g_{\cal P_G}\otimes\omega_X).
\end{align*}

On the other hand, the cotangent complex $\cal T^*_{\op{LocSys}_G'/\op{Bun}_G}\big|_S$ sends $\cal M\in\op{QCoh}(S)^{\le 0}$ to the $\infty$-groupoid $\op{Maps}_{S//\op{Bun}_G}(\widetilde S, \op{LocSys}_G')$, where $\widetilde S$ is the split square-zero extension of $S$ by $\cal M$, equipped with the canonical map $\widetilde{\cal P}_G : \widetilde S\rightarrow S\rightarrow\op{Bun}_G$. The $\infty$-groupoid $\op{Maps}_{S//\op{Bun}_G}(\widetilde S, \op{LocSys}_G')$ classifies null-homotopies of the Atiyah sequence of $\widetilde{\cal P}_G$, regarded as an element in
\begin{align*}
\op{Fib}&(\tau^{\le 0}\op R\Gamma(\widetilde S\times X, \fr g_{\widetilde{\cal P}_G}\otimes\omega_X[1]) \rightarrow \tau^{\le 0}\op R\Gamma(S\times X, \fr g_{\cal P_G}\otimes\omega_X[1])) \\
& \xrightarrow{\sim} \tau^{\le 0}\op R\Gamma(S\times X, \cal M\otimes\fr g_{\cal P_G}\otimes\omega_X[1]),
\end{align*}
where the isomorphism uses the facts that $\tau^{\le 0}$ commutes with limits, and that the ideal sheaf of the embedding $\iota : S\hookrightarrow\widetilde S$ agrees with $\iota_*\cal M$. However, the Atiyah sequence of $\widetilde{\cal P}_G$ already admits a splitting by the map $\widetilde S \rightarrow S \rightarrow\op{LocSys}_G'$. Hence $\op{Maps}_{S//\op{Bun}_G}(\widetilde S, \op{LocSys}_G')$ classifies null-homotopies of the zero element in $\tau^{\le 0}\op R\Gamma(S\times X, \cal M\otimes\fr g_{\cal P_G}\otimes\omega_X[1])$, i.e., elements of $\tau^{\le 0}\op R\Gamma(S\times X, \cal M\otimes\fr g_{\cal P_G}\otimes\omega_X)$.
\end{proof}

\noindent
We conclude the proof of Proposition \ref{prop-locsys-compare}.\qed(Proposition \ref{prop-locsys-compare})

\begin{rem}
If one appeals to \cite[\S III.1, Proposition 8.3.2]{GR16}, it would suffice to check Lemma \ref{lem-identify-cotangent-spaces} only for classical affine schemes $S$.
\end{rem}

\subsection{Identification of the fiber at $\infty$} We now specialize to the parameter $(\fr g^{\infty}, 0) : \op{pt}\rightarrow\op{Par}_G$, where $\fr g^{\infty}$ identifies with the subspace $\fr g^*\hookrightarrow\fr g\oplus\fr g^*$. The quasi-twisting $\cal T_G^{(\infty, 0)}$ over $\op{Bun}_G$ is obtained as the quotient of $\widetilde{\cal T}_G^{(\infty,0)}$ (i.e., \eqref{eq-qtw-full-level} at parameter $(\fr g^{\infty}, 0)$) by the pair $(\fr g^{\infty}(\cal O_x), \cal L_x^+G)$ along the $\cal L_x^+G$-torsor $\op{Bun}_{G,\infty x}\rightarrow\op{Bun}_G$.

\begin{prop}
\label{prop-fiber-at-infty}
\begin{itemize}
	\item $\cal T_G^{(\infty, 0)}$ is the inert quasi-twisting associated to the triangle \eqref{eq-underlying-o-modules} (for $n=0$):
	\begin{equation}
	\label{eq-underlying-o-modules-zero}
	\cal O_{\op{Bun}_G} \rightarrow \widehat{\cal Q}^{(\infty,0)}_{\op{desc}} \rightarrow \cal Q^{\infty}_{\op{desc}}
	\end{equation}
	
	\item there is a canonical isomorphism of DG stacks:
	$$
	\mathbb V(\widehat{\cal Q}^{(\infty,0)}_{\op{desc}})_{\lambda = 1} \xrightarrow{\sim} \op{LocSys}_G.
	$$
\end{itemize}
\end{prop}

Combined with \eqref{eq-module-cat-abelian-qtw}, we obtain an equivalence of DG categories:
$$
\cal T_G^{(\infty, 0)}\Mod \xrightarrow{\sim} \op{QCoh}(\op{LocSys}_G).
$$
\begin{proof}[Proof of Proposition \ref{prop-fiber-at-infty}]
It is clear from the construction that the classical quasi-twisting $\widetilde{\cal T}_G^{(\infty,0)}$ is given by the central extension of Lie algebroids (with zero Lie bracket and anchor map)
$$
0\rightarrow \cal O_{\op{Bun}_{G,\infty x}} \rightarrow  \widehat{\cal L}^{(\infty,0)} \rightarrow \cal L^{\infty} \rightarrow 0.
$$
Since $\cal T_G^{(\infty, 0)}$ arises from the quotient of $\widetilde{\cal T}_G^{(\infty,0)}$ by $(\fr g^{\infty}(\cal O_x), \cal L_x^+G)$, the paradigm of \S\ref{sec-inert-qtw-geom} applies, and $\cal T_G^{(\infty, 0)}$ is the inert quasi-twisting on the triangle \eqref{eq-underlying-o-modules-zero}.

For the second statement, note that we have a push-out diagram in $\op{QCoh}(\op{Bun}_G)$:
$$
\xymatrix@R=1.5em@C=1.5em{
	\op R\Gamma(X, \cal O_{\op{Bun}_G\times X})^* \ar[r]\ar[d] & \op R\Gamma(X, \op{At}(\cal P_G)\otimes\omega_X)^* \ar[d] \\
	\cal O_{\op{Bun}_G} \ar[r] & \widehat{\cal Q}^{(\infty,0)}_{\op{desc}},
}
$$
by Proposition \ref{prop-underlying-o-modules-identify} and Serre duality. Hence $\mathbb V(\widehat{\cal Q}^{(\infty,0)}_{\op{desc}})_{\lambda=1}$ fits into the commutative diagram:
$$
\xymatrix@R=1.5em@C=1.5em{
	\mathbb V(\op R\Gamma(X, \cal O_{\op{Bun}_G\times X})^*) & \mathbb V(\op R\Gamma(X, \op{At}(\cal P_G)\otimes\omega_X)^*) \ar[l] \\
	\op{Bun}_G \ar[u]^{\{1\}} & \mathbb V(\widehat{\cal Q}_{\op{desc}}^{(\infty,0)})_{\lambda=1}. \ar[l]\ar[u]
}
$$

For any DG scheme $S$ mapping to $\op{Bun}_G$ (represented by the $G$-bundle $\cal P_G$ over $S\times X$), a computation using the projection formula shows:
\begin{itemize}
	\item $\op{Maps}_{\op{Bun}_G}(S, \mathbb V(\op R\Gamma(X, \op{At}(\cal P_G)\otimes\omega_X)^*)) \xrightarrow{\sim} \tau^{\le 0}\op R\Gamma(S\times X, \op{At}(\cal P_G)\otimes\omega_X)$, and
	\item $\op{Maps}_{\op{Bun}_G}(S, \mathbb V(\op R\Gamma(X, \cal O_{\op{Bun}_G\times X})^*))\xrightarrow{\sim} \tau^{\le 0}\op R\Gamma(S\times X, \cal O_{S\times X})$.
\end{itemize}
Hence $\op{Maps}_{\op{Bun}_G}(S, \mathbb V(\widehat{\cal Q}_{\op{desc}}^{(\infty,0)})_{\lambda=1})$ identifies with the $\infty$-groupoid
$$
\tau^{\le 0}\op R\Gamma(S\times X, \op{At}(\cal P_G)\otimes\omega_X) \underset{\tau^{\le 0}\op R\Gamma(S\times X, \cal O_{S\times X})}{\times} \{1\}
$$
i.e., the $\infty$-groupoid of splittings of the Atiyah sequence $\fr g_{\cal P_G}\rightarrow\op{At}(\cal P_G)\rightarrow \cal T_{S\times X/S}$. We obtain an isomorphism $\mathbb V(\widehat{\cal Q}^{(\infty,0)}_{\op{desc}})_{\lambda = 1} \xrightarrow{\sim} \op{LocSys}_G'$ so the result follows from Proposition \ref{prop-locsys-compare}.
\end{proof}

\begin{rem}
An alternative argument (one that avoids using the results of \S\ref{sec-underlying-o-modules}) runs as follows: by a local computation, one identifies the universal envelope of the classical quasi-twisting \eqref{eq-qtw-full-level} with the (topological) ring of functions over $\op{LocSys}_{G,\infty x}(\Sigma)$, the stack classifying $(\cal P_G,\alpha)\in\op{Bun}_{G,\infty x}$ together with a connection over $\cal P_G\big|_{\Sigma}$. One then shows that the closed subscheme $\mathbb V(\widehat{\cal Q}^{(\infty, 0)})_{\lambda =1}$ identifies with $\op{LocSys}_{G,\infty x}$, and \eqref{eq-quotient-abelian-qtw} gives rise to isomorphisms:
$$
\mathbb V(\widehat{\cal Q}_{\op{desc}}^{(\infty,0)})_{\lambda =1}\xrightarrow{\sim} \op{LocSys}_{G,\infty x}/\cal L_x^+G \xrightarrow{\sim} \op{LocSys}_G.
$$
\end{rem}

\subsubsection{}
We comment on the role of \emph{integral} additional parameters at $\infty$, i.e., the ones arising from $Z(G)$-bundles. More precisely, let $E:=\op{At}(\cal P_{Z(G)})^*$ for some $Z(G)$-bundle $\cal P_{Z(G)}$. Then $E$ is an extension of $\fr z_G^*\otimes\cal O_X$ by $\omega_X$, so $(\fr g^{\infty}, E)$ is a well defined $k$-point of $\op{Par}_G$.

\begin{prop}
Let $E=\op{At}(\cal P_{Z(G)})^*$ for a $Z(G)$-bundle $\cal P_{Z(G)}$. Then there is a canonical isomorphism of DG stacks:
\begin{equation}
\label{eq-fiber-at-infty-shift}
\mathbb V(\widehat{\cal Q}_{\op{desc}}^{(\infty, E)})_{\lambda=1} \xrightarrow{\sim} \op{LocSys}_G \underset{\op{Bun}_G}{\times} \op{Bun}_G,
\end{equation}
where the second map is the central shift $-\otimes\cal P_{Z(G)}$.
\end{prop}
\begin{proof}
Note that the $\cal D_{\op{Bun}_{G,\infty x}\times X/\op{Bun}_{G,\infty x}}$-module \eqref{eq-kac-moody-extension-twist} at parameter $(\fr g^{\infty}, E)$ is induced from the following sequence:
$$
0 \rightarrow \omega_{\op{Bun}_{G,\infty x}\times X/\op{Bun}_{G,\infty x}} \rightarrow \op{At}(\cal P_{Z(G)}\otimes\cal P_G)^* \rightarrow \fr g_{\cal P_G}^* \rightarrow 0
$$
via the functor $(-)_{\cal D}$ and pushing out (see \S\ref{sec-underlying-o-modules}). An argument similar to above shows that $\cal T_G^{(\infty, E)}$ is the inert quasi-twisting associated to the triangle in $\op{QCoh}(\op{Bun}_G)$:
$$
\cal O_{\op{Bun}_G} \rightarrow \widehat{\cal Q}^{(\infty, E)}_{\op{desc}} \rightarrow \cal Q^{\infty}_{\op{desc}},
$$
where we have a canonical isomorphism $\widehat{\cal Q}_{\op{desc}}^{(\infty, E)}\big|_{\cal P_G}\xrightarrow{\sim}\widehat{\cal Q}_{\op{desc}}^{(\infty, 0)}\big|_{\cal P_{Z(G)}\otimes\cal P_G}$. Hence the result follows from Proposition \ref{prop-fiber-at-infty}.
\end{proof}

\begin{rem}
A connection on $\cal P_{Z(G)}$ gives rise to a splitting of $E$, hence an isomorphism $\mathbb V(\widehat{\cal Q}_{\op{desc}}^{(\infty, E)})_{\lambda=1} \xrightarrow{\sim} \mathbb V(\widehat{\cal Q}_{\op{desc}}^{(\infty, 0)})$. Geometrically, this corresponds to a lift of the isomorphism $-\otimes\cal P_{Z(G)}:\op{Bun}_G\xrightarrow{\sim}\op{Bun}_G$ to $\op{LocSys}_G$.
\end{rem}

\begin{rem}
Specializing the hypothetical equivalence \eqref{eq-naive-quantum-langlands} to the parameter $(\fr g^{\op{crit}}, 0)$, we obtain the usual, na\"ive statement of the geometric Langlands correspondence:
$$
\op{Diff}(\cal L_{G,\det}^{-\frac{1}{2}})\Mod(\op{Bun}_G) \xrightarrow{\sim} \op{QCoh}(\op{LocSys}_{\check G}).
$$

Specializing to $(\fr g^{\op{crit}}, E)$ where $E=\op{At}(\cal P_{Z(\check G)})^*$, we obtain from \eqref{eq-fiber-at-infty-shift} a hypothetical equivalence:
$$
\op{Diff}(\cal L_{G,\det}^{-\frac{1}{2}} \otimes \cal M)\Mod(\op{Bun}_G) \xrightarrow{\sim} \op{QCoh}(\op{LocSys}_{\check G}\underset{\op{Bun}_{\check G}}{\times}\op{Bun}_{\check G})
$$
where $\cal M$ is the pullback to $\op{Bun}_G$ of the line bundle on $\op{Bun}_{G/[G,G]}$ corresponding to $\cal P_{Z(\check G)}$. This equivalence can be viewed as an expected compatibility of the geometric Langlands duality with central shift. Let us reiterate that when $G$ is not a torus, none of these equivalences are true without a renormalization process.
\end{rem}

\medskip

\end{document}